\documentclass[12pt,oneside,reqno]{amsart}
\usepackage{geometry}
\newcommand{\A}{\mathbb{A}}

\newcommand{\C}{\mathbb{C}}

\newcommand{\G}{\mathbb{G}}

\newcommand{\Q}{\mathbb{Q}}
\newcommand{\R}{\mathbb{R}}

\newcommand{\T}{\mathbb{T}}

\newcommand{\Z}{\mathbb{Z}}

\newcommand{\Ocal}{\mathcal{O}}

\newcommand{\leqnomode}{\tagsleft@true\let\veqno\@@leqno}

\usepackage[utf8]{inputenc}
\usepackage{tgtermes}
\usepackage[english]{babel}
\usepackage{amsmath,
	mathtools,
	amssymb,
	graphicx,
	bbm,
	xcolor,
	amsthm,
	mathrsfs,
	imakeidx,
	mathdots}
\usepackage[all]{xy}
\usepackage[backref=page]{hyperref}
\hypersetup{}
\usepackage{hyperref}
\usepackage{tikz-cd}
\usetikzlibrary{decorations.pathmorphing}
\usepackage{dynkin-diagrams}
\usepackage{listings}
\usepackage{bm}
\theoremstyle{plain}

\newtheorem{theorem}{Theorem}[section]
\newtheorem{corollary}[theorem]{Corollary}
\newtheorem{definition}[theorem]{Definition}
\newtheorem{lemma}[theorem]{Lemma}
\newtheorem{proposition}[theorem]{Proposition}

\newtheorem{assumption}[theorem]{Assumption}
\newtheorem{remark}[theorem]{Remark}

\makeatletter
\@namedef{subjclassname@2020}{%
	\textup{2020} Mathematics Subject Classification}
\makeatother

\usepackage[shortlabels]{enumitem}

\date{}

\makeindex
\begin{document}
	\title{Automorphic side of Taylor-Wiles method for orthogonal and symplectic groups}
	\author{Xiaoyu Zhang}
	\address{Universität Duisburg-Essen, Fakultät für Mathematik, Mathematikcarrée, Thea-Leymann-Straße 9, 45127 Essen, Germany}
	\email{xiaoyu.zhang@uni-due.de}
	\maketitle

	\begin{abstract}		
		The core of the Taylor-Wiles and Taylor-Wiles-Kisin method in proving modularity lifting theorems is the construction of Taylor-Wiles primes satisfying certain conditions relating automorphic side and Galois side. In this article, we construct such primes and develop the automorphic side of Taylor-Wiles method for definite special orthogonal or symplectic groups $G$ over a totally real number field $F$, beyond the only known case for definite unitary groups (except for $\mathrm{GSp}_4$).
		As an application of our result, we prove a minimal $R=\T$ theorem for \( G \), extending the scope of modularity lifting results to this setting. As a direct consequence, we deduce the Bloch--Kato conjecture for the adjoint of the Galois representation \( r_\pi \) associated to an automorphic representation \( \pi \) of \( G(\mathbb{A}_F) \). Our approach combines deformation theory with automorphic methods, providing new evidence towards the Langlands program for orthogonal and symplectic groups.
	\end{abstract}

	\tableofcontents

	\section{Introduction}

	One of the most powerful tools in advancing the Langlands program has been the development of modularity lifting theorems, originating in the groundbreaking work of Wiles and Taylor-Wiles' work on Fermat's Last Theorem (\cite{Wiles1995,TaylorWiles1995}). Since this pioneering breakthrough, their method has been extended to various settings, including unitary groups (\cite{ClozelHarrisTaylor2008,Geraghty2019,Guerberoff2011}), to the symplectic group $\mathrm{GSp}_4$ (\cite{Pilloni2012}), to a derived version for the general linear group $\mathrm{GL}_n$ (\cite{CalegariGeraghty2018}), to a non-minimal case (the so-called Taylor-Wiles-Kisin patching, see for example \cite{Kisin2009,Kisin2009b}), among others. There have been recently also quite a lot of research dealing with the case of residually reducible Galois representations related to these groups  (see for example \cite{NewtonThorne2023} and many references therein). These developments have led to significant applications in number theory, for example, the \emph{potential} modularity of abelian surfaces over totally real number fields (\cite{BoxerCalegariGeePilloni2021}), the modularity of symmetric powers of modular forms (\cite{NewtonThorne}), to name just a few. Almost all of the known works in modularity lifting theorems relies on certain definite unitary group. Little is known for other definite classical groups (except for $\mathrm{GSp}_4$ as in \cite{Pilloni2012}). Consequently, extending modularity lifting theorems to new classes of reductive groups remains a central and compelling objective in the Langlands program. It would also be extremely useful to obtain new advances in Langlands program for such classical groups.

	One of the crucial ingredients in the Taylor-Wiles method (and Taylor-Wiles-Kisin method) is the construction of the so-called Taylor-Wiles primes $\ell$ in the process of Taylor-Wiles deformations such the local component at $\ell$ of the irreducible automorphic representation $\pi$ of suitably ramified level at $\ell$ is related to the local Galois representation at $\ell$ of the global Galois representation attached to $\pi$ (this is \emph{the automorphic side of Taylor-Wiles method}). This often requires a rather careful analysis of the admissible/smooth representations of the corresponding $\ell$-adic group and of its relation to the local Galois representation. In the previous literature, such results are only available for definite unitary groups (except for $\mathrm{GSp}_4$ as in \cite{Pilloni2012}). In this article, we develop the Taylor-Wiles method for definite special orthogonal or symplectic groups. This has the immediate application to minimal modularity lifting theorems for these groups as well as to Bloch-Kato conjectures for the adjoint motives attached to the corresponding Galois representations.


	Building on Arthur’s work \cite{Arthur2013}, several studies have explored \emph{potential} modularity lifting for $\mathrm{GSpin}_{2n+1}$-valued Galois representations (see, for example, \cite{BoxerCalegariGeePilloni2021,PatrikisTang2022} and the discussion below). These results rely on potential modularity lifting for the unitary group $U_n$ rather than on establishing an $R = \mathbb{T}$ theorem. As noted in \cite{ClozelHarrisTaylor2008}, proving an $R = \mathbb{T}$ theorem would be of great interest for understanding congruences among automorphic forms and special $L$-values.
	In this article, we establish such a theorem for definite special orthogonal and symplectic groups $G$, from which we deduce the Bloch-Kato conjecture for the corresponding adjoint Galois representation. Similar ideas were employed in \cite{LiuTianXiaoZhangZhu2022} for the Beilinson-Bloch-Kato conjectures, though in the setting where $G = U_n$ is a unitary group.

	\subsection*{Taylor-Wiles method}
	Now let's give more details about this article.
	The original Taylor-Wiles method for a residual Galois representation $\overline{r}\colon\Gamma_{F}\to\mathbf{G}(\overline{\mathbb{Q}}_p)$, valued in a reductive group $\mathbf{G}$, relies on the following ``numerical coincidence":
	\[
	[F\colon\mathbb{Q}](\mathrm{dim}(\mathbf{G})-\mathrm{dim}(\mathbf{B}))
	=
	\sum_{v|\infty}H^0(F_v,\mathrm{Ad}^0\overline{r}).
	\]
	Here, $\mathbf{B}$ denotes a Borel subgroup of $\mathbf{G}$ (see \cite[§1]{ClozelHarrisTaylor2008}). Considerable effort has been devoted to extending this method to groups $\mathbf{G}$ that do not satisfy this numerical coincidence (see, for example, \cite{CalegariGeraghty2018} for $\mathbf{G} = \mathrm{GL}_N$).\footnote{The bold letter $\mathbf{G}$ is used only in this introduction to denote certain kind of dual group of $G$ (the $L$-dual group or the $C$-dual group). In the main body of the article, we will use the $C$-dual group of $G$.}
	
	On the other hand, for $\mathbf{G} = \mathrm{SO}_n$ or $\mathrm{Sp}_n$, where the numerical coincidence does hold, there has been little work on $R = \mathbb{T}$ theorems thus far, apart from potential modularity lifting theorems that build on the case $G = U_n$, a unitary group (i.e. $\mathbf{G} = \mathcal{G}_n$ as in \cite{ClozelHarrisTaylor2008}; see \cite{CalegariEmertonGee2022} for the most recent developments). This article fills this gap, establishing important arithmetic applications as a consequence, and can be seen as the first step in developing Taylor-Wiles method (and eventually Taylor-Wiles-Kisin method) for these classical groups.

	One of the key steps in proving an $R = \mathbb{T}$ theorem for a group $G$ satisfying the numerical coincidence condition is:
	\begin{enumerate}
		\item
		On the Galois side: Establishing the existence of a set $Q$ of Taylor–Wiles primes such that the corresponding local Galois representations satisfy a specific property.

		\item 
		On the automorphic side: Demonstrating that suitably ramified admissible representations of $G(F_v)$ for $v \in Q$ give rise to a Galois representation of $\Gamma_{F_v}$ of the prescribed form from the previous step.
	\end{enumerate}
	With these conditions in place, one can apply arguments from Galois cohomology, along with the Taylor-Wiles patching method, to establish the isomorphism $R = \mathbb{T}$.
	
	For $G(F_v) = \mathrm{GL}_N(F_v)$, conditions (1) and (2) have already been established in \cite{ClozelHarrisTaylor2008} via the notion of ``bigness" and later in \cite{Thorne2012} through the concept of ``adequateness". Both notions pertain to (1), but additional work is required to verify (2).
	
	For an arbitrary reductive group $\mathbf{G}$ (not necessarily satisfying the numerical coincidence condition), Whitmore \cite{Whitmore2022} generalized the ``adequateness" notion and provided preliminary computations for (2).\footnote{Whitmore’s notion of ``adequateness" is weaker than ours and may yield stronger results when the full Taylor–Wiles method is applicable under this assumption. In this article, we do not undertake a thorough comparison of these two versions. However, our formulation is explicit enough to be adapted to the automorphic side mentioned above. See \S\ref{smooth representations of G(F_v)} for further details.}

	In this article, we restrict our focus to $G$ as a definite special orthogonal or symplectic group and establish (2), from which we deduce the $R = \mathbb{T}$ theorem and thus provide new evidence for the Langlands program for such groups.

	\subsection*{$R=\T$ theorem}
	We now state the main results of this article. We first state the $R=\T$ theorem and modularity lifting result that we obtained before we move on to the technical result on automorphic side of Taylor-Wiles method.

	Let $F$ be a totally real number field, and let $G$ be a definite special orthogonal or symplectic group that is an inner form of a quasi-split special orthogonal group over $F$ satisfying (C1). Thus, we have  
	\[
	G(F \otimes \mathbb{C}) = \mathrm{SO}_n(\mathbb{C})^{[F:\mathbb{Q}]} \quad \text{or} \quad \mathrm{Sp}_n(\mathbb{C})^{[F:\mathbb{Q}]}.
	\]  
	Fix a finite extension $F^\circ/F$ such that $G$ splits over $F^\circ$. In this article, we focus on $C$-cohomological automorphic representations and, accordingly, use the $C$-dual group ${}^C G$ instead of the $L$-dual group ${}^L G$ when considering Galois representations (see §\ref{Special orthogonal groups and automorphic Galois representations} for details).

	Fix an \emph{odd} prime $p$ and an isomorphism $\iota \colon \mathbb{C} \simeq \overline{\mathbb{Q}}_p$. Let $\pi=\otimes_v'\pi_v$ be an irreducible $C$-algebraic automorphic representation of $G(\A_F)$, satisfying certain natural regularity, unramifiedness, and local-global compatibility conditions. Following the work of Arthur and Shin (\cite{Arthur2013,Shin2024}), we associate to $\pi$ a semi-simple Galois representation:\footnote{Throughout this article, all representations of global and local Galois groups are assumed to be continuous.}  
	\[
	r_\pi = r_{\pi, \iota} \colon \Gamma_F \to {}^C G (\overline{\mathbb{Q}}_p),
	\]  
	which is totally odd and unramified at all but finitely many places of $F$ (see Theorem \ref{Galois representations attached to auto rep} for details).

	This representation $r_\pi$ is constructed via the weak transfer $\pi^\sharp$ of $\pi$, which is an automorphic representation of $\mathrm{GL}_N(\mathbb{A}_F)$ (Here, $N$ is as defined in (\ref{N})). We then show that the Galois representation associated with $\pi^\sharp$ factors through ${}^C G (\overline{\mathbb{Q}}_p)$ and satisfies additional properties listed in Theorem \ref{Galois representations attached to auto rep}.

	Fix a sufficiently large finite extension $E$ of $\mathbb{Q}_p$ with ring of integers denoted by $\mathcal{O}$, such that $r_\pi$ factors through ${}^C G(\mathcal{O}) \hookrightarrow {}^C G(\overline{\mathbb{Q}}_p)$. Let $W$ be an algebraic irreducible representation of $G$ over $\mathcal{O}$. For a compact open subgroup $K=\prod_vK_v\subset G(\mathbb{A}_{F,f})$, we denote by $M_{W,p}(K, \mathcal{O})$ the space of automorphic forms on $G(\mathbb{A}_{F,f})$ of level $K$, weight $W$, and coefficients in $\mathcal{O}$ (see (\ref{M_{W,p}(K,R)})).  
	
	Let $\mathbb{T}_W(K, \mathcal{O})$ be the spherical Hecke algebra over $\mathcal{O}$ generated by Hecke operators at places $v$ where $K$ is hyperspecial, acting on $M_{W,p}(K, \mathcal{O})$. Now, let $\mathfrak{P}$ be a minimal prime ideal of $\mathbb{T}_W(K, \mathcal{O})$, and let $\pi$ be an irreducible automorphic representation of $G(\mathbb{A}_f)$ satisfying condition (C2) such that $\pi\cap M_{W,p}(K, \mathcal{O})\neq0$ and for which the action of $\mathbb{T}_W(K, \mathcal{O})$ on $\pi$ factors through the quotient $\mathbb{T}_W(K, \mathcal{O}) / \mathfrak{P}$.  
	
	Let $\mathfrak{m}$ be a maximal ideal of $\mathbb{T}_W(K, \mathcal{O})$ containing $\mathfrak{P}$. Starting from $r_{\pi}$, we construct a residual semisimple Galois representation  
	\[
	\overline{\rho}_{\mathfrak{m}}
	\colon
	\Gamma_F \to {}^C G(\mathbb{T}_W(K, \mathcal{O}) / \mathfrak{m}),
	\]
	satisfying certain properties (see Proposition \ref{residual Galois rep associated to auto rep} for the precise statement).  
	
	Suppose now that $\overline{\rho}_\mathfrak{m}$ is absolutely irreducible (\emph{i.e.} the composition of $\overline{\rho}_\mathfrak{m}$ with the inclusion ${}^C G \hookrightarrow \mathrm{GL}_N$ is absolutely irreducible). Then $\overline{\rho}_\mathfrak{m}$ lifts to a Galois representation  
	\[
	\rho_{\mathfrak{m}}
	\colon
	\Gamma_{F} \to {}^C G(\mathbb{T}_W(K, \mathcal{O})_{\mathfrak{m}})
	\]
	satisfying certain properties (see Theorem \ref{Galois representation valued in Hecke algebra-char=0} for the precise statement).  
	
	Now, assume that $p$ is unramified in $F$ and that $\overline{\rho}_\mathfrak{m}$ is Fontaine-Laffaille at places $v \mid p$ (see §\ref{Fontaine-Lafaille deformations}). We define a deformation problem $\mathcal{S}$ for $\overline{\rho}_{\mathfrak{m}}$, which is of Fontaine-Laffaille type at places $v \mid p$. This problem is representable by $(R_\mathcal{S}, r_\mathcal{S})$, where $R_\mathcal{S}$ is the universal deformation ring, and  
	\[
	r_\mathcal{S} \colon \Gamma_{F} \to {}^C G(R_\mathcal{S})
	\]  
	is the universal deformation of $\overline{\rho}_\mathfrak{m}$. The above Galois representation $\rho_{\mathfrak{m}}$ is a lift of $\overline{\rho}_\mathfrak{m}$ of type $\mathcal{S}$, inducing a morphism of $\mathcal{O}$-algebras $R_\mathcal{S} \to \mathbb{T}_W(K, \mathcal{O})_\mathfrak{m}$. We now state our main result, establishing the first case of the $R = \mathbb{T}$ theorem for orthogonal and symplectic groups (see Theorem \ref{R=T}):

	\begin{theorem}\label{Theorem-1 R=T}
		Let $K = G(\widehat{\mathcal{O}_F})$. Suppose the following are satisfied:
		\begin{enumerate}
			\item 
			$\overline{\rho}_{\mathfrak{m}}$ is absolutely irreducible,

			\item 
			$\overline{\rho}_{\mathfrak{m}}(\Gamma_{F(\zeta_p)})$ is adequate (see Definition \ref{big} for the definition of ``adequate"),

			\item 
			$p \nmid [F^\circ : F]$,

			\item 
			for any $g\in G(\A_{F,f})$, $gG(F)g^{-1}\cap K$ contains no elements of order $p$.
		\end{enumerate}
		Then we have an isomorphism of \emph{local complete intersection} $\mathcal{O}$-algebras:  
		\[
		R_\mathcal{S}
		\simeq
		\mathbb{T}_W(K, \mathcal{O})_\mathfrak{m}.
		\]
	\end{theorem}

	\begin{remark}\rm
		\begin{enumerate}
			\item 
			It is certainly possible to consider smaller compact open subgroups $K$, similar to the case of definite unitary groups (see, for example, \cite{ClozelHarrisTaylor2008,Thorne2012}). However, since the main objective of this article is to establish the automorphic side of the Taylor-Wiles method, we restrict ourselves to the full-level subgroup $K=G(\widehat{\mathcal{O}_F})$.

			\item 
			The Taylor-Wiles and Taylor-Wiles-Kisin methods (for definite unitary groups) have also deep applications in $p$-adic Langlands program (see for example \cite{CaraianiEmertonGeraghtyPaskunasShin2016}), in which the Breuil-Schneider conjecture (\cite{BreuilSchneider2007}) is reduced to certain modularity lifting theorems. Our result would also give parallel applications in $p$-adic Langlands program for $G$, which seems to be still absent in the present literature.
		\end{enumerate}
	\end{remark}
	
	As a direct consequence, we obtain a minimal modularity lifting theorem for Galois representations valued in ${}^C G$, providing new evidence for the Langlands reciprocity conjecture in these cases (see Theorem \ref{modularity}):
	
	\begin{theorem}\label{main theorem-2}
		Keep the notations and assumptions as above. Let
		\[
		r\colon
		\Gamma_F
		\rightarrow
		\,^C
		G(\mathcal{O})
		\]
		be a Galois representation satisfying the following conditions (where $\overline{r}$ denotes
		its reduction modulo $\mathfrak{p}$):
		\begin{enumerate}
			\item
			$r$ is unramified at all finite places of $F$ not over $p$.
			
			\item 
			$\overline{r} \simeq \overline{r}_{\pi}$ and $\mu_r = \mu_{r_\pi}$, where $\mu_r$ is the character of $\Gamma_{F}$ attached to $r$ (see (\ref{mu_r})).
			
			\item 
			For all $v | p$, the restriction $r|_{\Gamma_{F_v}}$ lies in $\mathcal{D}^\mathrm{FL}_{v,\overline{r}_{\pi,v}}(\mathcal{O})$ (\emph{cf.} Definition \ref{crystalline local deformation problem}), where $\overline{r}_{\pi,v} = \overline{r}_\pi|_{\Gamma_{F_v}}$.
		\end{enumerate}
		
		Then there exists a minimal prime ideal $\mathfrak{P}'$ of $\mathbb{T}_W(K,\mathcal{O})$ contained in $\mathfrak{m}$ and an irreducible discrete series constituent $\pi'$ of $\mathcal{A}(G(\mathbb{A}_F))$, generated by a nonzero vector in  $M_{W,p}(K,\mathcal{O})$, such that the action of $\mathbb{T}_W(K,\mathcal{O})$ on $\pi'$ factors through $\mathbb{T}_W(K,\mathcal{O})/\mathfrak{P}'$, and we have
		\[
		r\simeq
		r_{\pi'}.
		\]
	\end{theorem}

	\begin{remark}\rm
		\begin{enumerate}
			\item 
			In this article, we consider only two types of deformation problems: crystalline deformations at places $v | p$ and Taylor-Wiles deformations at places $v \in Q$. It is certainly possible to study other types of deformation problems, such as minimal, ordinary, or Ramakrishna deformations, as in \cite[§2.4]{ClozelHarrisTaylor2008}. We leave this to the interested reader.
			
			\item 
			There are also many important works addressing the case where the Galois representation $r$ is residually reducible; see, for example, \cite{Thorne2015,FakhruddinKharePatrikis2022}. It would be interesting to adapt the arguments from these works to our setting for $G$, which we plan to investigate in future research.
		\end{enumerate}
	\end{remark}

	\subsection*{Bloch-Kato conjecture}
	We apply the $R=\mathbb{T}$ theorem to the Bloch-Kato conjecture and obtain the following result (Theorems \ref{Bloch-Kato, rank part} and \ref{Bloch-Kato, special value part}).
	
	\begin{theorem}\label{Theorem 3 Bloch-Kato}
		Let the assumptions be as in \ref{Theorem-1 R=T}.
		The adjoint Galois representation associated with $r_\pi$ is given by
		\[
		\rho_\pi=\mathrm{Ad}^0r_\pi\colon\Gamma_{F}\to\mathrm{Aut}(\widehat{\mathfrak{g}}(\Ocal)).
		\]
		We attach to $\rho_\pi$ the $L$-function $L(\rho_\pi,s)$ and the Bloch-Kato Selmer group $H^1_\mathrm{BK}(F,\rho_\pi\otimes E)$. Then we obtain the rank part of the Bloch-Kato conjecture for $\rho_\pi$:\footnote{The usual formulation of the rank part of the Bloch-Kato conjecture for $\rho_\pi$ replaces $\mathrm{ord}_{s=1}L(\rho_\pi,s)$ with $\mathrm{ord}_{s=0}L(\rho_\pi^\ast(1),s)$. Since $\rho_\pi$ is self-dual, these two quantities are equal.}
		\[
		\mathrm{ord}_{s=1}L(\rho_\pi,s)
		=0
		=
		\dim_E H^1_\mathrm{BK}(F,\rho_\pi\otimes E)
		-
		\dim_E H^0(F,\rho_\pi\otimes E).		
		\]
		Moreover, if we choose an automorphic form $f_\pi\in M_{W,p}(K,\mathcal{O})$, which is a vector in $\pi$ satisfying $f_\pi\not\equiv0(\mathrm{mod}\,\mathfrak{p})$,\footnote{Such an $f_\pi$ is necessarily unique up to multiplication by a unit in $\mathcal{O}$.} and write $(f_\pi,f_\pi)_\mathrm{Pet}$ for the Petersson product of $f_\pi$ with itself, then we obtain an identity of nonzero ideals in $\mathcal{O}$:
		\[
		(f_\pi,f_\pi)_\mathrm{Pet}\mathcal{O}
		=
		\chi_\mathcal{O}(H^1_\mathrm{BK}(F,\rho_\pi\otimes E/\mathcal{O})^\vee).
		\]
	\end{theorem}
	
	Note that $L(\rho_\pi,s)$ is, in fact, the adjoint $L$-function $L(\pi,\mathrm{Ad},s)$ of $\pi$.

	\begin{remark}\rm
		For odd prime $p$, the first part of the theorem is a special case of \cite[Theorem B]{Allen2016}. However, our proof is quite different from \emph{loc.cit}, where the author obtained the result using Taylor-Wiles-Kisin patching for unitary groups instead of orthogonal or symplectic groups. Our method of relating congruence ideals to $R=\mathbb{T}$-theorem allows us to go beyond the rank of Bloch-Kato conjecture, which seems not possible in \emph{loc.cit}.
	\end{remark}

	\begin{remark}\rm
		In the second part of the theorem, we fall short of proving the following identity:
		\[
		L^\ast(\rho_\pi,1)\mathcal{O}
		\stackrel{?}{=}
		\chi_\mathcal{O}(H^1_\mathrm{BK}(F,\rho_\pi\otimes E/\mathcal{O})^\vee).
		\]
		Here, $L^\ast(\rho_\pi,1)$ denotes the normalized $L$-value of $L(\rho_\pi,1)$, divided by a suitable period. The main difficulty is the lack of a known relation between $(f_\pi,f_\pi)_\mathrm{Pet}$ and $L(\rho_\pi,s)$.

		On the other hand, for $G$ quasi-split over $F$, Lapid and Mao (\cite{LapidMao2015}) conjectured a relation between the Petersson product of $f_\pi$ and $L^\ast(\pi,\mathrm{Ad},1)$ (the conjecture itself is an analogue of the Ichino-Ikeda formula (\cite{IchinoIkeda2010}) for the Gross-Prasad conjecture). Our result suggests that for $G$ compact at infinity, there might be a similar \emph{arithmetic} relation between $(f_\pi,f_\pi)_\mathrm{Pet}$ and $L^\ast(\pi,\mathrm{Ad},1)$, that is, their $p$-adic valuations should be equal to each other.
	\end{remark}

	\subsection*{Automorphic side of Taylor-Wiles method}
	For the proof of Theorem \ref{Theorem-1 R=T}, as mentioned above, the main difficulty lies on the automorphic side—specifically, demonstrating that suitably ramified admissible representations of $G(F_v)$ at $v\in Q$ give rise to a Galois representation of $\Gamma_{F_v}$ in a prescribed form. We now explain this in more detail: the condition imposed on $\overline{r}_v := \overline{r}_\pi|_{\Gamma_{F_v}}$ (where $v \in Q$ is a Taylor-Wiles prime) is the existence of a decomposition of the representation space $\kappa^N$ (equipped with a bilinear form $(-,-)$) as follows:
	\[
	(\overline{r}_v,(-,-))
	=
	(\overline{s}_v,(-,-)|_{\overline{s}_v})
	\bigoplus
	(\overline{\psi}_v,(-,-)|_{\overline{\psi}_v}).
	\]
	Here, $\overline{\psi}_v$ corresponds to the eigenspace associated with a certain eigenvalue of the Frobenius $\sigma_v$ (which is of one of the types described in Proposition \ref{projection of parahoric invariant gives iso}), and $\sigma_v$ acts semi-simply on $\overline{\psi}_v$. Moreover, $\overline{\psi}_v$ is not isomorphic to any subquotient of $\overline{s}_v$ (see Assumption \ref{assumption on Taylor-Wiles deformations-Thorne}).
	
	We then obtain the following central technical result of this article (see Theorem \ref{decomposition of the Galois representation} for a precise statement).
	
	\begin{theorem}\label{Theorem 4}
		Let $v\nmid p\infty$ be a place of $F$ where $G$ is split, and let $\pi$ be an irreducible admissible representation of $G(F_v)$ that is a local component of an irreducible automorphic representation $\Pi$ of $G(\mathbb{A}_F)$. Suppose that $\pi$ admits a nonzero vector invariant under a specific subgroup of a parahoric subgroup of $G(\mathcal{O}_{F,v})$. Further, assume that the Galois representation $\rho_{\Pi}$ is unramified at $v$ and that the Frobenius element $\overline{\rho}_{\Pi}(\sigma_v)$ has an eigenvalue $\overline{\alpha}$ of the form given in Proposition \ref{projector induces an isomorphism}. 
		
		Writing $(-,-)$ for the non-degenerate symmetric or symplectic bilinear form corresponding to the Galois representation $\rho_\Pi$, we then obtain a decomposition into non-degenerate symmetric or symplectic subspaces:
		\[
		(\rho_\Pi|_{\Gamma_{F_v}},(-,-))=(s,(-,-)|_s)\bigoplus(\psi,(-,-)|_\psi),
		\]
		where $s$ is unramified, and $\psi$ is tamely ramified with its restriction to the inertial subgroup $I_v$ given by a pair of scalar characters $(\psi',(\psi')^{-1})$.
	\end{theorem}
	
	A key difference between the above theorem and \cite[Proposition 5.12]{Thorne2012} is that we do not require $\pi$ to be generic. In the case $G(F_v) \simeq \mathrm{GL}_n(F_v)$, the genericity of $\pi$ is used to describe $\pi$ as an induced representation in a precise form. However, in our setting, even if we assume genericity, $\pi$ is not necessarily of such a form but only a subquotient thereof.
	
	Our key observation is that if $\pi$ is of level a particular parahoric subgroup $\mathfrak{p}_\Omega$ of $G(F_v)$, then, using Frobenius reciprocity, we can embed $\pi$ into an induced representation where at most one factor is ramified (a Steinberg representation of $\mathrm{GL}_2(F_v)$), while all other factors remain unramified. The parahoric subgroup $\mathfrak{p}_\Omega$ corresponds roughly to the setup in \emph{loc. cit.} with $(n_1,n_2)=(n-1,1)$. We refer to §\ref{smooth representations of G(F_v)} for further details.

	\begin{remark}
		The \emph{main purpose} of this article is to provide the working foundations for Taylor-Wiles and Taylor-Wiles-Kisin methods for the group $G$ considered in this article, rather than the applications that we give in this article (these applications are quite modest (or very immediate) and may even seem quite restrictive in some sense). However, as we have mentioned in several places about potential applications of our results, we hope that our result can be readily used by other researchers to prove much stronger results in modularity lifting theorems, Block-Kato conjectures, $p$-adic Langlands programs related to the group $G$.
	\end{remark}
	
	\subsection*{Outline}
	In §\ref{Special orthogonal groups and automorphic Galois representations}, we associate Galois representations to automorphic representations of $G(\mathbb{A}_F)$, following \cite{Arthur2013,Shin2024}.	
	In §\ref{Hecke algebras and rep}, we study the action of Hecke operators on smooth representations of the group $G(F_v)$ for a finite place $v$ of $F$.	
	In §\ref{smooth representations of G(F_v)}, we study the smooth representation theory of $G(F_v)$ when $G$ is split at $v$. In particular, we show that if a smooth representation of $G(F_v)$ has a nonzero fixed vector under a certain parahoric-like subgroup and arises from a global automorphic representation $\pi$, then the associated local Galois representation $r_\pi|_{\Gamma_{F_v}}$ decomposes as expected in the Taylor-Wiles method (Theorem \ref{decomposition of the Galois representation}).	This is the automorphic side of the Taylor-Wiles method.
	In §\ref{deformation theory}, we introduce deformation problems and analyze several types of local deformation problems in detail, including Fontaine-Laffaille and Taylor-Wiles deformations.	
	In §\ref{modularity lifting}, we establish a minimal modularity lifting theorem and then deduce Theorem \ref{main theorem-2}.	
	Finally, in §\ref{Application to Bloch-Kato conjectures}, we apply the $R=\mathbb{T}$ theorem to the Bloch-Kato conjecture, proving Theorem \ref{Theorem 3 Bloch-Kato}.

	\subsection*{Notations}
	\begin{enumerate}
		\item We fix an \emph{odd} prime number $p$	and a totally real number field	$F$ throughout this article. We use the following notations: for a finite set $S$ of places of $F$, $\mathbb{A}_F^S$ is the set of adèles in $\mathbb{A}_F$
		without the component $S$; for a finite place $v$ of $F$, we fix a uniformiser $\omega_v$ of $\mathcal{O}_{F,v}$, $\kappa_v$ the residual field of $\mathcal{O}_{F,v}$, $q_v$ its cardinal.

		We write $S_p$ for the set of places of $F$ over $p$. For a reductive group $G$ over $F$, we write $G_v$ for the base change of $G$ from $F$ to $F_v$ ($v$ is a place of $F$). Moreover, we write $G_\infty=\prod_{v|\infty}G_v$ and $G_{\infty,\C}=G_\infty\times_\R\C$. For a finite place $v$ of $F$, we write
		\[
		\mathrm{Art}_{F_v}
		\colon
		F_v^\times\to\Gamma_{F_v^\mathrm{ab}/F_v}
		\]
		for the local Artin map.
		
		\item We fix a sufficiently large finite extension $E$ of $\mathbb{Q}_p$
		containing all embeddings $F\hookrightarrow\overline{\mathbb{Q}}_p$
		and write $\mathcal{O}$ for the ring of integers of $E$, $\mathfrak{p}$
		the its maximal ideal and $\kappa$ its residual field.

		\item We fix an isomorphism of fields
		$\iota\colon\mathbb{C}\simeq\overline{\mathbb{Q}}_p$ that is compatible with field embeddings of $\overline{\mathbb{Q}}$ into $\mathbb{C}$ and $\overline{\mathbb{Q}}_p$. For any field $E$, we write $\Gamma_E$
		for the absolute Galois group $\mathrm{Gal}(\overline{E}/E)$. For a connected reductive group $G$ over $F$ and a place $v$ of $F$, the $v$-th component of an element	$g\in G(\mathbb{A}_F)$ is denoted by $g_v\in G(F_v)$.
		
		\item We write $\mathrm{M}_{r,s}(A)$ for the set of $r\times s$-matrices with entries in $A$.
		
		\item For a local ring $A$, we write $\mathfrak{m}_A$ for the maximal ideal of $A$.

		\item 
		For a real number $x\in\R$, we write $\lfloor x\rfloor$ for the largest integer $n\leq x$ (the floor function of $x$).

		\item 
		The $p$-adic cyclotomic character has Hodge-Tate weight $-1$.
	\end{enumerate}

	\section*{Funding}
	This work is supported by Research Training Group funding (RTG 2553) of German Research Foundation (DFG).

	\section*{Acknowledgment}
	The author would like to thank Patrick Allen, Jeremy Booher, Gaëtan Chenevier, Arno Kret, Gordan Savin, Sug-Woo Shin, Jack Thorne, Jacques Tilouine and Dimitri Whitmore for a lot of useful discussions, for answering questions and for pointing out several inaccuracies in a previous version.

	\section{Automorphic Galois representations on $G$}\label{Special orthogonal groups and automorphic Galois representations}

	\subsection{Orthogonal/symplectic groups}
	For a positive integer $m$, we define the following anti-diagonal matrices
	\[
	A_m
	:=
	\begin{pmatrix}
		& & & & 1 \\
		& & & 1 & \\
		& & \iddots & & \\
		& 1 & & & \\
		1 & & & &
	\end{pmatrix},
	\quad
	A_{2m}'
	:=
	\begin{pmatrix}
		0 & A_m \\ -A_m & 0
	\end{pmatrix}
	\index{A@$A_m,A_{2m}'$}
	\]
	Fix a field extension $F_{\eta}/F$ of degree $\leq2$, we write
	\[
	\eta\colon\Gamma_{F_\eta/F}\to\{\pm1\}
	\]
	for the unique injective character. If $\eta=1$ (the trivial character), we write $A_m^\eta=A_m$\index{A@$A_m^\eta$}; if $\eta\neq1$ and $m$ is even, we choose $\alpha\in\mathcal{O}_F^\times$ such that $F_\eta=F[\sqrt{\alpha}]$, then we define $A_m^\eta$ to be $A_m$ with the $2\times2$-matrix $\begin{pmatrix}
		0 & 1 \\ 1 & 0
	\end{pmatrix}$ in the middle of $A_m$ replaced by $\begin{pmatrix}
		1 & 0 \\ 0 & -\alpha
	\end{pmatrix}$ (we do not consider the case where $\eta\neq1$ and $m$ odd).
	
	We fix a positive integer throughout this article and we write $G^\ast$ to be the following symplectic group or the quasi-split special orthogonal group scheme over $\Ocal_F$ (the classical group considered in the first 8 chapters of \cite{Arthur2013}):
	\begin{align*}
		\mathrm{Sp}_n
		&
		:=
		\left\{
		g\in\mathrm{SL}_n/\Ocal_F\mid g^\mathrm{t}A_n'g
		=A_n'
		\right\}
		(\text{$n$ is even})
		\\
		\mathrm{SO}_n^\eta
		&
		:=
		\{
		g\in\mathrm{SL}_n/\mathcal{O}_F\mid g^{\mathrm{t}}A_n^\eta g=A_n^\eta
		\}.
	\end{align*}
	For $\eta=1$, we simply write $\mathrm{SO}_n$ for $\mathrm{SO}_n^\eta$. Note that in $\mathrm{SO}_n^\eta$, if $\eta\neq1$, then $n$ even.

	We write $G$ for an inner form of $G^\ast$, that is, $G$ is a connected reductive group over $F$ together with an isomorphism $\phi\colon G^\ast_{\overline{F}}\simeq G_{\overline{F}}$ over $\overline{F}$ such that $\phi^{-1}(\sigma(\phi(\sigma^{-1})))\in\mathrm{Aut}(G^\ast_{\overline{F}})$ is an inner automorphism of $G^\ast_{\overline{F}}$ for any $\sigma\in\Gamma_{F}$.
	Throughout this article we assume the following is satisfied
	
	\begin{equation}\label{(C1)}
		\makebox[0pt][l]{\hspace{-2cm}(C1)}
		\text{$G(F\otimes_{\mathbb{Q}}\mathbb{R})$
			is compact and both $G$ and $G^\ast$ are unramified at all finite places $v$ of $F$.}\index{C@(C1)}
		\notag
	\end{equation}

	In particular, we have an isomorphism $G^\ast(\mathbb{A}_F^\infty)\simeq G(\mathbb{A}_F^\infty)$. For each finite place $v$ of $F$, we fix a hyperspecial subgroup $K_v$ of $G(F_v)$.

	We write $F^\circ/F$ for the splitting field of $G$.\index{F@$F^\circ$}

	We write $\widehat{G^\ast}$ for the dual group of $G^\ast$, which is a connected reductive group over $\C$ equipped with an action of the Galois group $\Gamma_{F}$ (factoring through $\Gamma_{F^\circ/F}$). More explicitly,
	we define
	\begin{equation}\label{N}
		N
		=
		\begin{cases*}
			2\lfloor n/2\rfloor,
			& $G^\ast=\mathrm{SO}_n^\eta$;
			\\
			n+1,
			&
			$G^\ast=\mathrm{Sp}_n$.
		\end{cases*}
	\end{equation}
	Then we have
	\[
	\widehat{G^\ast}
	=
	\begin{cases*}
		\mathrm{SO}_N,
		&
		$G^\ast=\mathrm{Sp}_n$;
		\\
		\mathrm{SO}_N,
		&
		$G^\ast=\mathrm{SO}_n^\eta$ with $n$ even;
		\\
		\mathrm{Sp}_N,
		&
		$G^\ast=\mathrm{SO}_n$ with $n$ odd.
	\end{cases*}
	\]
	By definition, the $L$-dual group $^LG^\ast$ of $G^\ast$ is the semi-direct product $\widehat{G^\ast}\rtimes\Gamma_{F}$, which can be simplified in our case as follows: write $F'$ to be $F$ unless $G^\ast=\mathrm{SO}_n^\eta$, in which case $F'=F_\eta$. Then the action of $\Gamma_{F}$ on $\widehat{G^\ast}$ factors through $\Gamma_{F'/F}$ and we will use this $F'/F$-form of the $L$-dual group of $G^\ast$:
	\[
	^LG^\ast=\widehat{G^\ast}\rtimes\Gamma_{F'/F}.
	\]
	For $\eta\neq1$ (thus $n$ is even), the non-trivial element in $\Gamma_{F'/F}$ acts on $\widehat{G^\ast}=\mathrm{SO}_N$ via the outer automorphism sending $g$ to $\vartheta g\vartheta^{-1}$ where
	\[
	\vartheta
	=
	\mathrm{diag}
	\left(
	1_{n/2-1},
	\begin{pmatrix}
		0 & 1 \\ 1 & 0
	\end{pmatrix},
	1_{n/2-1}
	\right).
	\]
	In this case, we can identify $^LG^\ast$ with $\mathrm{O}_N$. For later use, we write down explicitly the $L$-dual group of $G^\ast$ as follows (\cite[§2.1]{Shin2024})
	\[
	^L\mathrm{Sp}_n
	=
	\mathrm{SO}_N,
	\quad
	^L\mathrm{SO}_n^\eta
	=
	\begin{cases*}
		\mathrm{O}_N,
		&
		$n$ even and $\eta\neq1$;
		\\
		\mathrm{SO}_N,
		&
		$n$ even and $\eta=1$;
		\\
		\mathrm{Sp}_N,
		&
		$n$ odd.
	\end{cases*}
	\]
	Since $G$ is an inner form of $G^\ast$, the $F$-pinnings of $G$ and $F$-pinnings of $G^\ast$ are naturally identified (via $\phi$). This allows us to identify the $L$-dual group of $G$ with that of $G^\ast$ (see \cite[§1.4]{Shin2024} for more details)
	\[
	^LG\simeq\,^LG^\ast.
	\]

	We write
	\begin{equation}\label{xi: ^LG to GL_N}
		\widetilde{\xi}
		\colon\,^LG\to\mathrm{GL}_N
	\end{equation}
	for the natural embedding.

	We fix a maximal torus $T_{\infty,\C}=\prod_{v|F}T_{v}$ inside $G_{\infty,\C}$ and a Borel subgroup $B_{\infty,\C}=\prod_{v|F}B_v$ containing $T_{\infty,\C}$. Correspondingly, we fix a maximal torus $\widehat{T}$ of $\widehat{G}$ and a Borel subgroup $\widehat{B}$ containing $\widehat{T}$ (this gives rise to an isomorphism between the based root datum of $G$ and that of $\widehat{G}$). We write $\Omega_v$ for the Weyl group of the pair $(G_v,T_v)$ (we put $\Omega_\infty=\prod_v\Omega_v$) and
	\[
	\rho_G=(\rho_{G,v})_{v|\infty}\in
	X^\ast(T_{\infty,\C})_\Q=X_\ast(\widehat{T})_\Q
	\]
	for the half sum of positive roots in $G_{\infty,\C}$ with respect to $(T_{\infty,\C},B_{\infty,\C})$. More explicitly, we have (\cite[§2.1]{Shin2024})
	\begin{enumerate}
		\item 
		$G^\ast=\mathrm{Sp}_n$, then for $v|\infty$, we take
		\[
		T_{v}=\{\mathrm{diag}(t_1,\cdots,t_{n/2},1/t_{n/2},\cdots,1/t_1)|t_i\in\mathbb{G}_m\}.
		\]
		We identify $X^\ast(T_v)$ with $\Z^{n/2}$ in a natural way. The Weyl group is $\Omega_v=\{\pm1\}^{n/2}\rtimes S_{n/2}$ (here $(x_1,\cdots,x_{n/2})\in\{\pm1\}^{n/2}$ sends $\mathrm{diag}(t_1,\cdots,1/t_1)$ to $\mathrm{diag}(t_1^{x_1},\cdots,1/t_1^{x_1})$) and we have
		\[
		\rho_{G,v}=(n/2,n/2-1,\cdots,1).
		\]

		\item 
		$G^\ast=\mathrm{SO}_n^\eta$ with $n$ even, we take
		\[
		T_v=\{\mathrm{diag}(t_1,\cdots,t_{n/2-1},s,1/t_{n/2-1},\cdots,1/t_1)|t_i\in\G_m,s\in\mathrm{SO}_2^\eta\}.
		\]
		Again we have $X^\ast(T_v)=\Z^{n/2}$, $\Omega_v$ consists of $(x_1,\cdots,x_{n/2};\sigma)\in\{\pm1\}^{n/2}\rtimes S_{n/2}$ such that $\prod_ix_i=1$ and
		\[
		\rho_{G,v}=(n/2-1,n/2-2,\cdots,0).
		\]

		\item 
		$G^\ast=\mathrm{SO}_n$ with $n$ odd, we take
		\[
		T_v=\{\mathrm{diag}(t_1,\cdots,t_{(n-1)/2},1,1/t_{(n-1)/2},\cdots,1/t_1)|t_i\in\G_m\}.
		\]
		We have $X^\ast(T_v)=\Z^{(n-1)/2}$, $\Omega_v=\{\pm1\}^{(n-1)/2}\rtimes S_{(n-1)/2}$ and
		\[
		\rho_{G,v}
		=
		(n/2-1,n/2-2,\cdots,1/2).
		\]
	\end{enumerate}

	From the above description, we know that an element in $X^\ast(T_v)_\R/\Omega_v$ is represented by a tuple $(a_1,a_2,\cdots,a_{\lfloor n/2\rfloor})\in\R^{\lfloor n/2\rfloor}$ such that
	\begin{equation}\label{a_1>a_2...}
		\begin{cases*}
			a_1\ge a_2\ge\cdots\ge a_{N/2}\ge0,
			&
			if $G^\ast=\mathrm{Sp}_n$ or $G^\ast=\mathrm{SO}_n$ with $n$ odd;
			\\
			a_1\ge a_2\ge\cdots\ge a_{n/2-1}\ge|a_{n/2}|,
			&
			if $G^\ast=\mathrm{SO}_n^\eta$ with $n$ even.
		\end{cases*}
	\end{equation}

	\subsection{Automorphic Galois representations}\label{Automorphic Galois representations}
	In this subsection, we recall how to associate Galois representations to certain $C$-algebraic automorphic representations of $G(\mathbb{A}_F)$. We following closely the presentation in \cite{Shin2024}.

	For an irreducible discrete series automorphic representation $\pi=\otimes_v'\pi_v$ of $G(\mathbb{A}_F)$, the infinitesimal character $\zeta_{\pi,\infty}=(\zeta_{\pi,v})_{v|\infty}$ of $\pi_{\infty}$
	is an element in $X^\ast(T_{\infty,\mathbb{C}})_{\mathbb{C}}/\Omega_\infty=\oplus_{v|\infty} X^\ast(T_v)_{\mathbb{C}}/\Omega_v$. For a finite place $v$ of $F$, if $\pi_v$ is an unramified representation, we write $c(\pi_v)$ for the Satake parameter of $\pi_v$, which is a semi-simple $\widehat{G}$-conjugacy class inside $^LG_v$. This gives rise to an unramified $L$-parameter of $\pi_v$,
	\[
	\phi_{\pi_v}^L\colon W_{F_v}\to\,^LG_v,
	\]
	where $W_{F_v}$ is the Weil-Deligne group of $F_v$ and $\phi_{\pi_v}^L$ sends the geometric Frobenius to $c(\pi_v)$.
	\begin{definition}\label{C-algebraic}
		We say $\pi$ is \emph{$C$-algebraic} if $\zeta_{\pi,\infty}$ lies in the image of $X^\ast(T_{\infty,\mathbb{C}})+\rho_\infty$ in the quotient $X^\ast(T_{\infty,\mathbb{C}})_{\mathbb{C}}/\Omega_\infty$ (\cite[§3]{BuzzardGee2014} and \cite[Definition 3.1.1]{Shin2024}).
	\end{definition}

	It follows from definition that
	\begin{lemma}\label{criterion for C-algebraic}
		The automorphic representation $\pi$ is $C$-algebraic if and only if for any $v|\infty$, $\zeta_{\pi,v}$ is represented by a tuple $(a_{v,1},a_{v,2},\cdots,a_{v,\lfloor n/2\rfloor})$ satisfying (\ref{a_1>a_2...}) and
		\[
		\begin{cases*}
			a_{v,1},\cdots,a_{v,\lfloor n/2\rfloor}\in\Z,
			&
			if $G^\ast=\mathrm{Sp}_n$ or $G^\ast=\mathrm{SO}_n^\eta$ with $n$ even;
			\\
			a_{v,1},\cdots,a_{v,\lfloor n/2\rfloor}\in\frac{1}{2}+\Z,
			&
			 if $G^\ast=\mathrm{SO}_n$ with $n$ odd.
		\end{cases*}
		\]
	\end{lemma}

	The corresponding $C$-dual group of $G$ is given as follows (\cite[§3.1]{Shin2024})
	\[
	^CG=\,^LG\rtimes\G_m,
	\text{ where }
	(1,t)(g,1)(1,t)^{-1}=(\mathrm{Ad}\rho_G(t)g,1),
	\,
	\forall
	g\in\,^LG,t\in\G_m.
	\]
	Note that this definition is the one given in \cite{Zhu2021} but is in fact equivalent to the original definition given in \cite{BuzzardGee2014}. In our case, we have an explicit expression for $^CG$ (\cite[Examples 3.1.6 and 3.1.7]{Shin2024}):
	\begin{equation}\label{^CG}
		^CG
		=
		\begin{cases*}
			\,^LG\times\G_m,
			&
			$G^\ast=\mathrm{Sp}_n$ or $G^\ast=\mathrm{SO}_n^\eta$ with $n$ even;
			\\
			\mathrm{GSp}_N,
			&
			$G^\ast=\mathrm{SO}_n$ with $n$ odd.
		\end{cases*}
	\end{equation}
	For later considerations on Galois deformations, we will view $^CG$ as a group scheme over $\mathbb{Z}$.

	For a place $v|\infty$ of $F$, we modify the infinitesimal character $\zeta_{\pi,v}$ as follows
	\begin{equation}\label{zeta_{pi,v}^C}
		\zeta_{\pi,v}^C
		=
		(\zeta_{\pi,v},-1)
		\in
		X_\ast(\widehat{T}\times\G_m)_\Q/\Omega_\infty.
	\end{equation}

	For a finite place $v$ of $F$ such that $\pi_v$ is unramified, we define the unramified $C$-parameter of $\pi_v$ as follows (with $^CG$ given as in (\ref{^CG})):
	\[
	\phi_{\pi_v}^C
	\colon
	W_{F_v}\to\,^CG_v,
	\quad
	x
	\mapsto
	\begin{cases*}
		(\phi_{\pi_v}^L(x)|x|^{(1-N)/2},|x|^{-1}),
		&
		$G^\ast=\mathrm{Sp}_n$ or $G^\ast=\mathrm{SO}_n^\eta$ with $n$ even;
		\\
		\phi_{\pi_v}^L(x)|x|^{-1/2},
		&
		$G^\ast=\mathrm{SO}_n$ with $n$ odd.
	\end{cases*}
	\]
	Here $|-|_v$ is the absolute value on $F_v^\times$ or $W_{F_v}$ which takes the uniformiser in $F_v^\times$ or a lift of the geometric Frobenius in $W_{F_v}$ to $1/\# (\kappa_v)$.

	By \cite[Corollary 2.5.3]{Shin2024}, we have
	\begin{theorem}\label{weak transfer pi to Pi}
		We fix a finite set $S$ of places of $F$ containing those above $p$ or $\infty$.
		Let $\pi$ be an irreducible discrete series automorphic representation of $G(\mathbb{A}_F)$ unramified outside $S$. Then there is an irreducible automorphic representation
		\begin{equation}\label{Arthur parameter}
			\Pi=\pi^\sharp\index{p@$\pi^\sharp$}
		\end{equation}
		of $\mathrm{GL}_N(\mathbb{A}_F)$ unramified outside $S$, which is an isobaric sum of cuspidal automorphic representations $\Pi=\boxtimes_{i=1}^r\Pi_i$ such that $\Pi^\vee\simeq\Pi$ and $(\zeta_{\pi,\infty},c^S(\pi))$ is mapped via the map $\widetilde{\zeta}$ in (\ref{xi: ^LG to GL_N}) to $(\zeta_{\Pi,\infty},c^S(\Pi))$. We call $\pi^\sharp$ the \emph{weak transfer} of $\pi$.
	\end{theorem}
	\begin{proof}
		The proof of \cite[Corollary 2.5.3]{Shin2024} relies on the assumptions that the weighted fundamental lemma is true for non-split groups and its non-standard version is also true (this is the assumption (H1) in \emph{loc.cit}), which is used to ensure the stabilized trace formula. However, this is already proved in \cite{Arthur2001,Arthur2002,Arthur2003}.
	\end{proof}

	\begin{definition}\label{regular and std-regular}
		We say $\pi$ \emph{regular} if $\zeta_{\pi,v}$ has trivial stabilizer in $\Omega_v$ for each $v|\infty$. We say $\pi$ \emph{std-regular} if $\widetilde{\xi}(\zeta_{\pi,v})$ has trivial stabilizer in the Weyl group $S_{N}$ of $\mathrm{GL}_N$ for each $v|\infty$ (\cite[Definition 3.2.1]{Shin2024}). 		
	\end{definition}
	It follows that (see also \cite[Lemma 3.2.2 and Example 3.2.3]{Shin2024})
	\begin{lemma}\label{regular and std-regular, characterisation}
		Let $\pi$ be $C$-algebraic and $(a_{v,1},\cdots,a_{v,\lfloor n/2\rfloor})$ be as in Lemma \ref{criterion for C-algebraic}. Then
		\begin{enumerate}
			\item 
			$\pi$ is regular if and only if for each $v|\infty$, we have
			\[
			\begin{cases*}
				a_{v,1}>a_{v,2}>\cdots>a_{v,\lfloor n/2\rfloor}>0,
				&
				$G^\ast=\mathrm{Sp}_n$ or $G^\ast=\mathrm{SO}_n$ with $n$ odd;
				\\
				a_{v,1}>a_{v,2}>\cdots>|a_{v,\lfloor n/2\rfloor}|,
				&
				$G^\ast=\mathrm{SO}_n^\eta$ with $n$ even.
			\end{cases*}
			\]

			\item 
			$\pi$ is std-regular if and only if for each $v|\infty$, $a_{v,1}>a_{v,2}>\cdots>|a_{v,\lfloor n/2\rfloor}|>0$.
		\end{enumerate}

		Suppose $\pi$ is std-regular, then the infinitesimal character $(\zeta_{\pi^\sharp,v})_{v|\infty}$ of the weak transfer $\pi^\sharp$ is given by
		\[
		\zeta_{\pi^\sharp,v}
		=
		(\zeta_{\pi^\sharp,v,i})_{i=1}^N
		=
		\begin{cases*}
			(a_{v,1},\cdots,a_{v,\lfloor n/2\rfloor-1},|a_{v,\lfloor n/2\rfloor}|,-|a_{v,\lfloor n/2\rfloor}|,\cdots,-a_{v,1}),
			&
			$G^\ast=\mathrm{SO}_n^\eta$;
			\\
			(a_{v,1},\cdots,a_{v,\lfloor n/2\rfloor-1},|a_{v,\lfloor n/2\rfloor}|,0,-|a_{v,\lfloor n/2\rfloor}|,\cdots,-a_{v,1}),
			&
			$G^\ast=\mathrm{Sp}_n$.
		\end{cases*}
		\]
	\end{lemma}

	We consider the following condition on $\pi$ an irreducible discrete series automorphic representation of $G(\mathbb{A}_F)$ (this is the assumptions (H2) and (H3) in \cite{Shin2024}):
	\begin{equation}\label{(C2)}
		\makebox[0pt][l]{\hspace{-2cm}(C2)}
		\begin{cases*}
			\text{The infinitesimal character of the Arthur parameter $\pi^\sharp=\boxplus_i^r\Pi_i$ of $\pi$}
			\\
			\text{is std-regular and each $\Pi_i$ is self-dual.}
			\index{C@(C2)}
		\end{cases*}
	    \notag
	\end{equation}

	\begin{theorem}\label{Galois representations attached to auto rep}
		Let $\pi$ be an irreducible discrete series automorphic representation of $G(\mathbb{A}_F)$ that is $C$-algebraic, unramified at all finite places of $F$ and satisfies (C2).
		Then there is a semisimple Galois representation
		\[
		r_{\pi}
		\colon
		\Gamma_F\to\,^CG(\overline{\mathbb{Q}}_p)
		\]
		such that
		\begin{enumerate}
			\item For each finite place $v\nmid p$ of $F$, $r_\pi$ is unramified at $v$ and 
			\[
			(r_\pi|_{W_{F_v}})^{\mathrm{Fr}-\mathrm{ss}}\simeq\iota\circ\phi_{\pi_v}^C,
			\]
			the isomorphism is unique up to $^LG(\overline{\mathbb{Q}}_p)$-conjugates.

			\item For each embedding $\tau\colon F\to\overline{\Q}_p$ inducing a place $v|p$ of $F$, $r_\pi$ is crystalline at $v$. Moreover, its Hodge-Tate weights at $\tau$ are given by
			\begin{equation}\label{Hodge-Tate weights of r_pi}
				\left(
				(\zeta_{\pi^\sharp,\iota^{-1}\tau,i}+N-i)_{i=1}^N;-1
				\right)
				\in
				X_\ast(\iota(\widehat{T})\times\G_m)/\Omega_\infty.
			\end{equation}

			\item $r_\pi$ is totally odd.
		\end{enumerate}
	\end{theorem}
	\begin{proof}
		(1) is proved in \cite[Theorem 3.2.7]{Shin2024}. Since $\pi^\sharp=\boxtimes_{i=1}^r\Pi_i$ is unramified at each $v|p$, so is each cuspidal isobaric factor $\Pi_i$ and thus by \cite[Theorem 3.2.3(c)]{ChenevierHarris2013}, $r_\pi$ is crystalline at $v|p$. Moreover, the Hodge-Tate weights are given as in Theorem 3.2.3(b) of \emph{loc.cit}. This proves (2) (see also \cite[Theorem 2.1.1]{Barnet-LambGeeGeraghtyTaylor2014}). Finally, the Galois representation associated to each $\Pi_i$ is totally real by \cite{Taylor2012,Taibi2016,CaraianiLeHung2016}, thus so is $r_\pi$.
	\end{proof}

	\section{Hecke algebras}\label{Hecke algebras and rep}	
	In this section we
	establish some notions and results on
	automorphic forms and Hecke algebras
	on definite special orthogonal groups that we will use throughout this article.

	\subsection{Automorphic forms on $G(\mathbb{A}_F)$}\label{Automorphic forms on G(A_F)}
	In this subsection, we recall the theory of
	algebraic modular forms on $G$
	as in \cite[§4]{Gross1999}.
	We fix an algebraic representation
	$W$ of $ G$ over $\mathcal{O}_F$
	which is a free $\mathcal{O}_F$-module of finite rank.

	\begin{definition}
		For any compact open subgroup $K$ of $G(\mathbb{A}_{F,f})$ and any $\mathcal{O}_F$-algebra $R$, write $W(R)=W\otimes_{\mathcal{O}_F}R$ and put
		\[
		M_W(K,R)
		=
		\{
		f\colon
		G(\mathbb{A}_F)/(G(\mathbb{A}_{F,\infty})\times K)
		\rightarrow
		W(R)\mid 
		f(\gamma g)=\gamma f(g),
		\forall
		\gamma\in
		G(F)
		\}.
		\]
		This is the space of automorphic forms on $G(\mathbb{A}_F)$ of weight $W$, of level $K$ and of coefficients in $R$. Then we put
		\[
		M_W(R):=
		\lim\limits_{\overrightarrow{K}}
		M_W(K,R)\index{M@$M_W(K,R),M_W(R)$}
		\]
		where $K$ runs through the set of compact open subgroups of	$G(\mathbb{A}_{F,f})$. We let $G(\mathbb{A}_F)$ act on $M_W(R)$
		by right translation,
		\[
		(gf)(g'):=
		f(g'g),
		\quad
		\forall
		g,g'\in G(\mathbb{A}_F),
		\,
		f\in M_W(R).
		\]
	\end{definition}

	\begin{remark}\label{equivalence of various spaces of modular forms}\rm 
		We can define a space analogue to $M_W(K,R)$ for each place $v$ of $F$ as follows: suppose that $R$ is an $F_v$-algebra,
		then we set
		\[
		M_{W,v}(K,R)
		:=
		\{
		f'\colon
		G(F)\backslash G(\mathbb{A}_F)
		\rightarrow
		W(R)\mid 
		f'(gk)=k_v^{-1}f'(g),
		\forall
		k\in
		G(\mathbb{A}_{F,\infty})\times K
		\}.
		\]
		Similarly, we set
		$M_{W,v}(R)=
		\lim\limits_{\overrightarrow{K}}
		M_{W,v}(K,R)$.

		We let $G(\mathbb{A}_F)$ act on $M_{W,v}(K,R)$ by \emph{left inverse} translation,
		\[
		(gf')(g')
		:=
		f'(g^{-1}g'),
		\quad
		\forall
		g,g'\in G(\mathbb{A}_F),
		\,
		f'\in M_{W,v}(K,R).
		\]
		We have an isomorphism of $R$-modules
		\[
		M_W(K,R)
		\rightarrow
		M_{W,v}(K,R),
		\quad
		f
		\mapsto
		(
		f'\colon
		g\mapsto
		g_v^{-1}f(g)
		),
		\]
		whose inverse is given by
		\[
		M_{W,v}(K,R)
		\rightarrow
		M_W(K,R),
		\quad
		f'\mapsto
		(
		f\colon
		g\mapsto
		g_vf'(g)
		).
		\]
		Passing to the injective limit over $K$, we have an isomorphism	$M_W(R)\simeq M_{W,v}(R)$. Note however that these maps are not
		$G(\mathbb{A}_F)$-equivariant.

		For $v|\infty$ and $R=\mathbb{C}$, $M_{W,v}(K,\mathbb{C})$ is a subspace of the space $\mathcal{A}(G(\mathbb{A}_F))$ of automorphic forms on
		$G(\mathbb{A}_F)$. Recall that $\mathcal{A}(G(\A_F))$ consists of smooth functions of moderate growth
		\[
		f\colon G(\mathbb{A}_F)\rightarrow\mathbb{C}
		\]
		such that $f(\gamma g)=f(g)$ for all $\gamma\in G(F)$, $f$ is right $G(\mathbb{A}_{F,\infty})\times K'$-finite for some compact open subgroup
		$K'$ of $G(\mathbb{A}_{F,f})$ and is also	$Z(U(\mathrm{Lie}G(\mathbb{A}_{F,\infty})))$-finite. We identify the $\Ocal_F$-module $W$ with $\mathcal{O}_F^r$ for some $r\in\mathbb{N}$. Then each element
		$f'\in M_{W,v}(K,\mathbb{C})$ gives rise to $r$ maps
		$f'_i\colon G(\mathbb{A}_F)\rightarrow\mathbb{C}$ such that $f'_i(g)$
		is the $i$-th component of the vector $f'(g)\in W\otimes_{\mathcal{O}_F}\mathbb{C}=\mathbb{C}^r$ ($i=1,\cdots,r$).
		The $Z(U(\mathrm{Lie}G(\mathbb{A}_{F,\infty})))$-finiteness of $f'_i$ comes from that fact that $f_i$ is invariant under translation of $G(\mathbb{A}_{F,\infty})$. Other conditions are similarly verified and thus
		$f_i'$ is an automorphic form on $G(\mathbb{A}_F)$.
	\end{remark}

	From now on we make the following assumption on our fixed compact open subgroup $K$ of the form $K=\prod_vK_v$ of $G(\mathbb{A}_{F,f})$: for each place $v|p$ of $F$, $K_v$ is  contained in the hyperspecial compact open subgroup $G(\mathcal{O}_{F,v})$ of $G(F_v)$.

	\begin{definition}\label{M_{W,v}space of automorphic forms}
		For any place $v |p$ of $F$ and any $\mathcal{O}_{F,v}$-algebra $R$,
		as in the preceding remark, we set
		\[
		M_{W,v}(K,R)
		=
		\{
		f'\colon
		G(F)\backslash G(\mathbb{A}_F)
		\rightarrow
		W(R)\mid 
		f'(gk)=k_v^{-1}f'(g),
		\forall
		k\in
		G(\mathbb{A}_{F,\infty})\times K
		\}.
		\]
		Similarly, we put
		\[
		M_{W,v}(R)
		:=
		\lim\limits_{\overrightarrow{K}}
		M_{W,v}(K,R).\index{M@$M_{W,v}(K,R),M_{W,v}(R)$}
		\]
	\end{definition}

	It is well-known that the space of automorphic forms on $G(\mathbb{A}_F)$
	is a semisimple admissible $G(\mathbb{A}_F)$-module, thus
	\begin{proposition}\label{the space of modular forms is semisimple}
		For an
		$\mathcal{O}_F$-algebra $R$ and a place
		$v |p$ of $F$,
		the space
		$M_{W,v}(R)$
		is a semisimple
		admissible
		$G(\mathbb{A}_F)$-module.		
	\end{proposition}

	More generally, for each place $v|p$ of $F$, we choose an irreducible representation $W_v$ of $G$ over $\mathcal{O}_{F,v}$ and write $W=\otimes_{v|p}W_v$ for the tensor product over $\mathbb{Z}_p$. For any $\mathcal{O}$-algebra $R$, we put
	\begin{equation}\label{M_{W,p}(K,R)}
		M_{W,p}(K,R)
		:=
		\{
		f'\colon
		G(F)\backslash G(\mathbb{A}_F)
		\rightarrow
		W(R)\mid 
		f'(gk)=k_p^{-1}f'(g),\forall
		k\in G(\mathbb{A}_{F,\infty})\times K
		\}
		\index{M@$M_{W,p}(K,R)$}
	\end{equation}
	and as before, we set
	\[
	M_{W,p}(R)
	:=
	\lim\limits_{\overrightarrow{K}}
	M_{W,p}(K,R).
	\]
	It is easy to see that
	there is a natural isomorphism of $R$-modules
	\[
	M_{W,p}(K,R)
	\simeq
	\otimes_{v|p}
	M_{W_v,v}(K,R).
	\]

	Following
	\cite[Proposition 4.3]{Gross1999},
	for any compact open subgroup $K$ of
	$G(\mathbb{A}_{F,f})$,
	we write
	\[
	\Sigma_K
	:=
	G(F)\backslash
	G(\mathbb{A}_F)/
	(G(\mathbb{A}_{F,\infty})\times K).\index{S@$\Sigma_K$}
	\]
	\begin{proposition}
		The double coset
		$\Sigma_K$
		is finite.
		For an $F$-algebra,
		resp.
		an $F_v$-algebra,
		an $\mathcal{O}_{F,v}$-algebra $R$
		with
		$v|p$,
		the $R$-module
		$M_W(K,R)$,
		resp.
		$M_{W,v}(K,R)$
		is of finite type.
	\end{proposition}
	\begin{proof}
		Write
		$G_{(r)}=\mathrm{Res}_{F/\mathbb{Q}}G$,
		then $K$ contains a compact open subgroup
		$\mathbf{K}$ of
		$G_{(r)}(\mathbb{A}_{\mathbb{Q},f})$.
		Since the double quotient
		$\mathbf{\Sigma}_\mathbf{K}
		=
		G_{(r)}(\mathbb{Q})
		\backslash
		G_{(r)}(\mathbb{A}_\mathbb{Q})/
		(G_{(r)}(\mathbb{R})\times\mathbf{K})$
		is finite by \cite{Gross1999}
		and we have a natural projection
		$\Sigma_K
		\rightarrow
		\mathbf{\Sigma}_\mathbf{K}$
		with finite fibers,
		we conclude that
		$\Sigma_K$ is finite.
		Since elements in
		$M_W(K,R)$
		are determined by their images on
		$\Sigma_K$,		
		$M_W(K,R)$
		is indeed of finite rank over $R$.
		The same argument applies to
		$M_{W,v}(K,R)$.
	\end{proof}

	We fix a set
	$\{g_\alpha\in G(\mathbb{A}_F)\}$
	of representatives for
	$\Sigma_K$
	(indexed by $\alpha\in\Sigma_K$),
	and write
	\begin{align*}
		G_{g_\alpha}
		&
		:=
		G(F)\bigcap
		g_\alpha(G(\mathbb{A}_{F,\infty})\times K)g_\alpha^{-1},
		\\
		G'_{g_\alpha}
		&
		:=
		G(F)
		\bigcap
		g_\alpha^{-1}
		(G(\mathbb{A}_{F,\infty})\times K)
		g_\alpha,
	\end{align*}
	which are arithmetic subgroups of $G(F)$,
	thus are finite by
	\cite[Proposition 1.4]{Gross1999}.

	\begin{corollary}\label{the space of modular forms is finite dimensional}
		For an $F$-algebra $R$,
		we have an isomorphism of
		$R$-modules
		\[
		M_W(K,R)
		\simeq
		\bigoplus_{\alpha\in\Sigma_K}
		W(R)^{G_{g_\alpha}},
		\quad
		f\mapsto
		(f(g_\alpha))_\alpha.
		\]
		Similarly,
		for $v |p$
		and an
		$\mathcal{O}_{F,v}$-algebra $R$,
		we have an isomorphism of $R$-modules
		\[
		M_{W,v}(K,R)
		\simeq
		\bigoplus_{\alpha\in\Sigma_K}
		W(R)^{G_{g_\alpha}'},
		\quad
		f\mapsto
		(f(g_\alpha))_\alpha.
		\]
	\end{corollary}
	\begin{proof}
		Indeed,
		for the first part,
		for each $g=g_\alpha g'g_\alpha^{-1}
		\in
		G_{g_\alpha}$
		with
		$g'\in
		G(\mathbb{A}_{F,\infty})\times K$,
		we have
		$g(f(g_\alpha))
		=
		f(gg_\alpha)
		=
		f(g_\alpha g')
		=
		f(g_\alpha)$
		by definition of $f$,
		thus
		$f(g_\alpha)
		\in
		W(R)^{G_{g_\alpha}}$
		for any $\alpha\in\Sigma_K$.
		The bijectivity is then clear. The second part is proved similarly.
	\end{proof}

	Recall the notion of ``neat" (\cite[Definition 1.4.1.8]{Lan2008}):
	\begin{definition}
		Fix a faithful representation
		$\rho
		\colon
		G\rightarrow
		\mathrm{GL}(M)$
		and for each finite place $v$ of $F$,
		write
		$H(g,v)$
		for the subgroup of
		$\overline{F_v}^\times$
		generated by the eigenvalues of
		$\rho(g_\alpha)$
		which are of finite multiplicative order
		(and thus lie in $\overline{\mathbb{Q}}^\times$).
		The element $g$ is
		\textbf{neat} if the intersection
		$\cap_vH(g,v)=\{1\}$
		is trivial.
		A compact open subgroup $K$ is
		neat if all its elements are neat.
	\end{definition}	
	Clearly $K$ is neat if and only if any conjugate of $K$ is neat. In the following, we will assume that the compact open subgroup
	$K$
	is neat.
	Now consider the groups
	$G_{g_\alpha}$ and
	$G'_{g_\alpha}$.
	Since each of them is finite,
	the eigenvalues of
	$\rho(g_v)$
	for any
	$g\in G_{g_\alpha}$
	or
	$g\in G'_{g_\alpha}$
	and any finite place
	$v$ is of finite order.
	Since
	$g\in G(F)$,
	all these $\rho(g_v)$
	have common eigenvalues and
	the intersection
	$\cap_vH(g,v)$
	is trivial if and only if
	$\rho(g_v)=1$
	for any $v$.
	Now that $\rho$ is faithful,
	we deduce that $g=1$.
	Thus  these groups
	$G_{g_\alpha}=G'_{g_\alpha}=\{1\}$
	for all $\alpha\in\Sigma_K$.
	From this we see that for any
	$R$-algebra $R'$,
	if either $R'$ is flat over $R$ or
	$K$ is neat, then
	there is a natural isomorphism
	\[
	M_{W,v}(K,R')
	\simeq
	M_{W,v}(K,R)
	\otimes_RR'.
	\]
	For two compact open subgroups
	$K,K'$ of
	$G(\mathbb{A}_{F,f})$
	and an element
	$g\in
	G(\mathbb{A}_{F,f})$
	such that
	$K'\subset gKg^{-1}$,
	there is a natural map
	\[
	g
	\colon
	M_{W,v}(K,R)
	\rightarrow
	M_{W,v}(K',R),
	\quad
	f\mapsto
	(gf\colon
	g'
	\mapsto
	g_vf(g'g)
	).
	\]
	Moreover it is easy to see that
	if $K'$ is a normal subgroup of
	$K$,
	then
	$K$ acts in this manner on the space
	$M_{W,v}(K',R)$
	and this action factors through
	the quotient
	$K/K'$.
	Moreover
	the subspace of
	$K$-invariants
	$M_{W,v}(K',R)^K$
	is 
	$M_{W,v}(K,R)$.

	The same proof as \cite[Lemma 6.4]{Thorne2012} gives
	\begin{proposition}\label{compare invariants of different cpt open subgroups}
		Let $R$ be an $\mathcal{O}_{F,v}$-algebra.
		Let $K$ be a compact open subgroup of $G(\A_{F,f})$ such that for any $g\in G(\A_{F,f})$, $gG(F)g^{-1}\cap K$ contains no elements of order $p$.
		Let $K'\vartriangleleft K$ be a normal compact open subgroup such that the quotient $K/K'$ is abelian of $p$-power order. Then $M_{W,v}(K',R)$
		is a finite free $R[K/K']$-module. Similarly, if $R$ is an $\mathcal{O}$-algebra, $M_{W,p}(K',R)$ is a finite free	$R[K/K']$-module
		and in this case taking the coinvariants gives an isomorphism of
		$R$-modules
		\[
		M_{W,p}(K',R)_{K/K'}
		\simeq
		M_{W,p}(K,R).
		\]
	\end{proposition}

	\subsection{Hecke algebras}
	In this subsection
	we define Hecke operators acting on
	$M_{W,v}(K,R)$. We follow the treatment in \cite{Gross1999} and we refer there for more details.

	We fix a pair
	$(\mathbf{B},\mathbf{T})$
	where $\mathbf{B}$ is a Borel subgroup of $G_{\overline{F}}$
	and $\mathbf{T}\subset \mathbf{B}$
	the maximal torus of $\mathbf{B}$.
	This gives a based root datum
	\[
	(X^\ast(\mathbf{T}),\Phi^\ast(\mathbf{T}),
	X_\ast(\mathbf{T}),\Phi_\ast(\mathbf{T})),
	\]
	where
	$X^\ast(\mathbf{T})$
	is the group of characters of $\mathbf{T}$ and
	$\Phi^\ast(\mathbf{T})$
	is the set of simple roots of
	$(G(F_v),\mathbf{T})$.
	The choice of $\mathbf{B}$
	gives rise to a subset
	$\Delta^\ast(\mathbf{T})$
	of $\Phi^\ast(\mathbf{T})$
	consisting of positive roots of $(\mathbf{B},\mathbf{T})$.
	Similarly, we have a subset
	$\Delta_\ast(\mathbf{T})
	\subset
	\Phi_\ast(\mathbf{T})$
	of positive coroots.
	By the perfect pairing
	between
	$X^\ast(\mathbf{T})$
	and
	$X_\ast(\mathbf{T})$,	
	one writes
	\[
	X^{\ast,+}(\mathbf{T})
	:=
	\{
	\lambda\in X^\ast(\mathbf{T})\mid 
	\langle\lambda,\alpha\rangle\geq0,
	\forall
	\alpha\in
	\Delta_\ast(\mathbf{T})
	\}
	\]
	of dominant characters of $X^\ast(\mathbf{T})$
	and similarly
	the subset
	$X^+_\ast(\mathbf{T})
	\subset
	X_\ast(\mathbf{T})$
	of dominant cocharacters
	of $\mathbf{T}$.
	Besides, the Galois group
	$\mathrm{Gal}(\overline{F}/F)$
	acts on this root datum.
	To ease notations,
	we suppose that
	the irreducible algebraic representation
	$W$ of $G$
	corresponds to a dominant weight
	$\lambda$ of
	$\mathbf{T}$
	and we will write
	\[
	M_{\lambda,v}(K,R)
	:=
	M_{W,v}(K,R).
	\]

	\subsubsection{Spherical Hecke operators}
	We follow
	\cite[§16]{Gross1999}.
	For a prime $\mathfrak{q}$ of $F$ over a prime $q$ of $\mathbb{Q}$,
	we have a hyperspecial subgroup
	$G(\mathcal{O}_{F,\mathfrak{q}})$
	of
	$G(F_\mathfrak{q})$.

	\begin{definition}\label{spherical Hecke algebra}
		We define the
		spherical Hecke algebra at $\mathfrak{q}$
		as
		\[
		\mathbb{T}_\mathfrak{q}
		=
		C_\mathrm{c}^\infty
		\left(
		G(F_\mathfrak{q})
		//
		G(\mathcal{O}_{F,\mathfrak{q}}),
		\mathbb{Z}[\frac{1}{q}]
		\right),\index{T@$\mathbb{T}_\mathfrak{q}$}
		\]
		the set of locally constant functions of compact support on
		the double quotient
		$G(F_\mathfrak{q})//G(\mathcal{O}_{F,\mathfrak{q}})$
		with values in
		$\mathbb{Z}[\frac{1}{q}]$.
		We can view elements in $\mathbb{T}_{\mathfrak{q}}$
		as functions on
		$G(F_\mathfrak{q})$
		and
		equip $\mathbb{T}_{\mathfrak{q}}$ with the usual convolution, then $\mathbb{T}_{\mathfrak{q}}$
		becomes a
		$\mathbb{Z}[\frac{1}{q}]$-algebra.
	\end{definition}

	We write
	$\mathbf{T}_\mathfrak{q}$
	for the base change of $\mathbf{T}$ from
	$\mathcal{O}_F$ to
	$\mathcal{O}_{F,\mathfrak{q}}$
	and
	$\mathbf{T}_\mathfrak{q}^\mathrm{s}
	\subset
	\mathbf{T}_\mathfrak{q}$
	for the maximal split sub-torus.
	Then we have an isomorphism
	$X_\ast(\mathbf{T}_\mathfrak{q})^{\Gamma_{F_\mathfrak{q}}}
	\simeq
	X_\ast(\mathbf{T}_\mathfrak{q}^\mathrm{s})$
	and also
	(\emph{cf.}
	\cite[(16.5)]{Gross1999})
	\[
	X_\ast(\mathbf{T}_\mathfrak{q}^\mathrm{s})
	\simeq
	\mathbf{T}_\mathfrak{q}(F_\mathfrak{q})/
	\mathbf{T}_\mathfrak{q}(\mathcal{O}_{F,\mathfrak{q}}),
	\quad
	\lambda
	\mapsto
	\lambda(\mathfrak{q}).
	\]	  
	
	We assume in the rest of this subsection that
	$K_{\mathfrak{q}}=G(\mathcal{O}_{F,\mathfrak{q}})$.
	
	We fix a prime $\mathfrak{Q}$ of $F^\circ$
	over $\mathfrak{q}$.
	We put
	$W^\mathrm{s}$
	for the Weyl group of
	$\mathbf{T}_\mathfrak{q}^\mathrm{s}$
	and
	$W$
	the Weyl group of $\mathbf{T}_\mathfrak{q}$ over
	$\overline{F}$.
	Then one has
	$W^{\Gamma_{F^\circ/F}}\simeq
	W^\mathrm{s}$.
	Denote by
	$\mathbf{U}_\mathfrak{q}$
	the unipotent radical of $\mathbf{B}_\mathfrak{q}$ over
	$\mathcal{O}_{F,\mathfrak{q}}$
	and
	by
	$du$
	the unique Haar measure on
	$\mathbf{U}_\mathfrak{q}
	(F_\mathfrak{q})$
	such that
	$\int_{\mathbf{U}_\mathfrak{q}(\mathcal{O}_{F,\mathfrak{q}})}du=1$.
	Then we have the Satake transformation
	of
	$\mathbb{Z}[\frac{1}{q}]$-algebras\footnote{Later in this article we will also use $\mathcal{S}$ to denote deformation problems (see, for example, Definition \ref{T-deformation or lifting of type S}). It should be clear from the context $\mathcal{S}$ refers to which one of the two meanings.}
	\[
	\mathcal{S}
	\colon
	\mathbb{T}_\mathfrak{q}
	\rightarrow
	C_\mathrm{c}^\infty
	(
	\mathbf{T}_\mathfrak{q}(F_\mathfrak{q})//
	\mathbf{T}_\mathfrak{q}(\mathcal{O}_{F,\mathfrak{q}}),
	\mathbb{Z}[\frac{1}{q}]
	),
	\quad
	f\mapsto
	\left(
	t\mapsto
	\Delta_{\mathbf{B}}^{1/2}(t)
	\int_{\mathbf{U}_\mathfrak{q}(F_\mathfrak{q})}
	f(tu)du
	\right). 
	\]
	Here
	$\Delta_{\mathbf{B}}
	\colon
	\mathbf{B}(F_\mathfrak{q})
	\rightarrow
	\mathbb{R}_+^\times$
	is the modular function of
	$\mathbf{B}(F_\mathfrak{q})$.
	This morphism is injective and its image is exactly
	the $W^\mathrm{s}$-invariants
	by the Satake isomorphism
	(\emph{cf.}
	\cite{Gross1998}):
	\begin{align}\label{Satake isomorphism}
		\begin{split}
			\mathcal{S}
			\colon
			\mathbb{T}_\mathfrak{q}
			&
			\simeq
			C_\mathrm{c}^\infty
			(
			\mathbf{T}_\mathfrak{q}(F_\mathfrak{q})//
			\mathbf{T}_\mathfrak{q}(\mathcal{O}_{F,\mathfrak{q}}),
			\mathbb{Z}[\frac{1}{q}]
			)^{W^\mathrm{s}}
			\simeq
			\mathbb{Z}[\frac{1}{q}][X_\ast(\mathbf{T}_\mathfrak{q}^\mathrm{s})]^{W^\mathrm{s}}
			\\
			&
			\simeq
			\mathbb{Z}[\frac{1}{q}][X_\ast(\mathbf{T})^{\Gamma_{F^\circ/F}}]^{W^\mathrm{s}}
			\simeq
			\mathbb{Z}[\frac{1}{q}]
			[X_\ast(\mathbf{T})]^W.
		\end{split}
	\end{align}

	\begin{definition}\label{unramified Hecke operators}
		Fix a place $v|p$ of $F$.
		We let
		$\mathbb{T}_\mathfrak{q}$
		act on
		$M_{\lambda,v}(K,R)$
		as follows: for any $g_{\mathfrak{q}}\in G(F_{\mathfrak{q}})$, we view
		\[
		K_\mathfrak{q}g_\mathfrak{q}K_\mathfrak{q}
		=
		\sqcup_ig_iK_\mathfrak{q}
		\]
		as an element in
		$\mathbb{T}_\mathfrak{q}$
		(the characteristic function of this double coset).
		Let
		$\widetilde{g}\in
		G(\mathbb{A}_{F,f})$
		be the element with $\mathfrak{q}$-th component equal to
		$g_\mathfrak{q}$
		while
		other components equal to $1$
		(similarly for $\widetilde{g_i}\in G(\mathbb{A}_{F,f})$).
		Then we have
		$K\widetilde{g}K
		=
		\sqcup_i\widetilde{g_i}K$
		and we define the double coset operator as follows
		\[
		[K_\mathfrak{q}gK_\mathfrak{q}]
		:=
		[KgK]
		\colon
		M_{\lambda,v}(K,R)
		\rightarrow
		M_{\lambda,v}(K,R),
		\quad
		f\mapsto
		(
		h\mapsto
		\sum_i(\widetilde{g_i})_v
		f(h\widetilde{g_i})
		).
		\]
		We extend
		this map by
		$\mathbb{Z}[\frac{1}{q}]$-linearity to the whole
		$\mathbb{T}_\mathfrak{q}$
		(note that $R$ is an
		$\mathcal{O}_{F,v}$-algebra and thus also a
		$\mathbb{Z}[\frac{1}{q}]$-algebra).
		We denote the image of
		$\mathbb{T}_\mathfrak{q}$
		by the same notation.
	\end{definition}
	Then by
	Proposition
	\ref{the space of modular forms is semisimple}
	and
	Corollary
	\ref{the space of modular forms is finite dimensional},
	$\mathbb{T}_\mathfrak{q}$
	is a finite semisimple
	$\mathbb{Z}[\frac{1}{q}]$-algebra.

	\subsubsection{Iwahori Hecke operators}
	Fix a finite place $v$ of $F$.
	We can identify
	$G\times_{\mathcal{O}_F}\mathcal{O}_{F,v}$
	with the special orthogonal group/symplectic group
	$\mathrm{SO}(V,q)$
	associated to a non-degenerate quadratic/symplectic module
	\[
	(V,(-,-))
	\]
	over
	$\mathcal{O}_{F,v}$ which is free of rank $n$
	and has good reduction mod
	$\varpi_v$.
	For computational purposes,
	we write down the pairing
	$(-,-)$ explicitly as follows:
	\begin{enumerate}
		\item 
		If $G^\ast=\mathrm{SO}_n^\eta$ with $n$ odd,
		there is a basis
		$(e_1,\cdots,e_n)$
		of $V$ and $u\in\mathcal{O}_{F,v}^\times$ such that
		\[
		(\sum_ix_ie_i,\sum_iy_ie_i)
		=
		x_1y_n+\cdots+x_{\frac{n-1}{2}}y_{\frac{n+3}{2}}+ux_{\frac{n+1}{2}}y_{\frac{n+1}{2}}+x_{\frac{n+3}{2}}y_{\frac{n-1}{2}}+\cdots+x_ny_1.
		\]
		In this case
		$G$ is split at $v$.

		\item 
		If $G^\ast=\mathrm{SO}_n^\eta$ with $n$ even,
		there is a basis
		$(e_1,\cdots,e_n)$
		of $V$ and an element $u\in\mathcal{O}_{F,v}^\times$ such that
		\[
		(\sum_ix_ie_i,\sum_iy_ie_i)
		=
		x_1y_n+\cdots+x_{\frac{n}{2}-1}y_{\frac{n}{2}+2}+(x_{\frac{n}{2}}y_{\frac{n}{2}}-ux_{\frac{n}{2}+1}y_{\frac{n}{2}+1})+x_{\frac{n}{2}+2}y_{\frac{n}{2}-1}+\cdots+x_ny_1.
		\]
		Moreover, $G$ is split at $v$ if and only if $u$ is a square in $F_v$.

		\item 
		If $G^\ast=\mathrm{Sp}_n$, there is a basis $(e_1,\cdots,e_n)$ of $V$ such that
		\[
		(\sum_ix_ie_i,\sum_iy_ie_i)
		=
		x_1y_n+\cdots+x_{\frac{n}{2}}y_{\frac{n}{2}+1}-x_{\frac{n}{2}+1}y_{\frac{n}{2}}-\cdots-x_ny_1.
		\]
	\end{enumerate}
	In the following
	we will fix such a basis for $V$
	and such a bilinear form $(-,-)$ on $V$
	and identify
	\[
	G_v
	:=
	G\times_{\mathcal{O}_F}\mathcal{O}_{F,v}
	\]
	with the special isometry group of $(V,(-,-))$.
	Therefore
	$G_v\times_{\mathcal{O}_{F,v}}\kappa_v$
	is smooth and we have a surjective reduction map
	$G_v(\mathcal{O}_{F,v})
	\rightarrow
	G_v(\kappa_v)$.
	We put
	\[
	U_r
	\subset
	B_r
	\subset
	\mathrm{GL}_r
	\]
	for the subgroup
	consisting of
	unipotent upper
	triangular matrices,
	upper triangular matrices
	in
	$\mathrm{GL}_r$.
	We then fix the standard Borel subgroup
	$\mathbf{B}_v$ of $G_v$
	as follows:
	\[
	\mathbf{B}_v
	=
	\begin{cases*}
		G_v\cap B_n,
		&
		$G_v$ split;
		\\
		\bigg\{
		g=
		\begin{pmatrix}
			A & \ast & \ast \\
			0 & B & \ast \\
			0 & 0 & C
		\end{pmatrix}
		\in
		G_v
		\Big|
		A,C\in
		B_{\frac{n}{2}-1}
		\bigg\},
		&
		$G_v$ non-split.
	\end{cases*}
	\]
	We then set
	\[
	\mathbf{U}_v,\,
	\mathbf{T}_v,\,
	\mathbf{T}_v^\mathrm{s}
	\subset
	\mathbf{B}_v
	\]
	to be the unipotent radical,
	the maximal torus,
	the maximal split torus of
	$\mathbf{B}_v$
	and put
	\begin{equation}\label{n_s}
		n_\mathrm{s}=n_{\mathrm{s},v}
		:=
		\mathrm{rk}(\mathbf{T}^\mathrm{s}_v)
		=
		\begin{cases*}
			\lfloor\frac{n}{2}\rfloor
			&
			if $\mathbf{T}_v=\mathbf{T}^\mathrm{s}_v$;
			\\
			\lfloor\frac{n}{2}\rfloor-1
			&
			otherwise.
		\end{cases*}
		\index{n@$n_{\mathrm{s}}=n_{\mathrm{s},v}$}
	\end{equation}

	\begin{definition}\label{Iwahori subgroups and the like}
		We write
		\[
		\mathrm{Iw}(v),\,
		\mathrm{Iw}_1(v)
		\subset
		G(\mathcal{O}_{F,v})
		\]
		for the preimages of
		$\mathbf{B}(\kappa_v)$,
		$\mathbf{U}(\kappa_v)$
		under the reduction map
		$G(\mathcal{O}_{F,v})
		\rightarrow
		G(\kappa_v)$.
		For positive integers
		$m,n_0>0$
		with
		$n_0\leq n_\mathrm{s}$,
		we write
		\[
		U_0^{(n_0)}(v^m)
		:=
		\Bigg\{
		g
		\in
		G(\mathcal{O}_{F,v})
		\bigg|
		g
		\equiv
		\begin{pmatrix}
			A & \ast & \ast \\
			0 & B & \ast \\
			0 & 0 & C
		\end{pmatrix}
		(\mathrm{mod}\,\varpi_v^m),
		A,C\in
		B_{n_0}
		(\mathcal{O}_{F,v}/\varpi_v^m)
		\Bigg\},
		\]
		and
		\[
		U_1^{(n_0)}(v^m)
		:=
		\Bigg\{
		g
		\in
		G(\mathcal{O}_{F,v})
		\bigg|
		g
		\equiv
		\begin{pmatrix}
			A & \ast & \ast \\
			0 & B & \ast \\
			0 & 0 & C
		\end{pmatrix}
		(\mathrm{mod}\,\varpi_v^m),
		A,C\in
		U_{n_0}
		(\mathcal{O}_{F,v}/\varpi_v^m)
		\Bigg\}.
		\]
		We define some special elements in
		$G(F_v)$
		as follows:
		for each integer
		$0\leq j\leq n_\mathrm{s}$,
		we write
		\[
		\varsigma_{v,j}
		=
		\mathrm{diag}
		(\varpi_v1_j,1_{n-2j},\varpi_v^{-1}1_j)
		\]
		and then we denote by $U_v^{(j)}$ the following Hecke operators (it should be clear from the context which operator $U_v^{(j)}$ refers to)
		\[
		U_v^{(j)}
		=
		[U_0^{(n_0)}(v^m)
		\varsigma_{v,j}
		U_0^{(n_0)}(v^m)],
		\,
		[U_1^{(n_0)}(v^m)
		\varsigma_{v,j}
		U_1^{(n_0)}(v^m)].
		\]
	\end{definition}
	It is clear that
	$\mathrm{Iw}_1(v)$
	is a normal subgroup of
	$\mathrm{Iw}(v)$
	and
	\[
	\mathrm{Iw}(v)/
	\mathrm{Iw}_1(v)
	\simeq
	(\kappa_v^\times)^{n_\mathrm{s}}.
	\]
	Similarly,
	$U_1^{(n_0)}(v^m)$
	is a normal subgroup of
	$U_0^{(n_0)}(v^m)$
	and
	\[
	U_0^{(n_0)}(v^m)/
	U_1^{(n_0)}(v^m)
	\simeq
	((\mathcal{O}_{F,v}/\varpi_v^m)^\times)^{n_0}.
	\]
	Then we have the following
	\begin{proposition}
		There is a finite set of elements
		$(b_i)_{i\in I}$ in $G(F_v)$
		depending on $j$ and $n_0$,
		independent of $m$,
		such that
		\[
		U_0^{(n_0)}(v^m)
		\varsigma_{v,j}
		U_0^{(n_0)}(v^m)
		=
		\bigsqcup_{i\in I}
		b_i
		U_0^{(n_0)}(v^m),
		\]
		\[
		U_1^{(n_0)}(v^m)
		\varsigma_{v,j}
		U_1^{(n_0)}(v^m)
		=
		\bigsqcup_{i\in I}
		b_i
		U_1^{(n_0)}(v^m).
		\]
	\end{proposition}
	\begin{proof}
		To ease notations,
		we write
		$U_?$
		for $U_?^{(n_0)}(v^m)$
		($?=0,1$)
		and
		$\varsigma=\varsigma_{v,j}$.
		We have the following decomposition of
		$U_0$:
		write
		$U_{0,0}$
		for the subgroup of
		$U_0$
		consisting of those $g$ such that
		$g
		\equiv
		\mathrm{diag}
		(A,B,C)(\mathrm{mod}\,\varpi_v^m)$
		with $A,C\in
		B_{n_0}
		(\mathcal{O}_{F,v}/\varpi_v^m)$,
		$U_{0,1}$
		for the subgroup of
		$U_0$
		consisting of those 
		$g$ such that
		$g
		=
		\begin{pmatrix}
			1_{n_0} & \ast & \ast \\
			0 & 1_{n-2n_0} & \ast \\
			0 & 0 & 1_{n_0}
		\end{pmatrix}
		(\mathrm{mod}\,\varpi_v^m)$
		is unipotent.
		Then we have
		a semi-direct product
		\[
		U_0
		=
		U_{0,1}\rtimes U_{0,0}.
		\]
		Similarly, for the case of
		$U_1$,
		we write
		$U_{1,0}
		=
		U_{0,0}\cap
		U_1$
		and
		$U_{1,1}
		=
		U_{0,1}\cap
		U_1$.
		Note that $U_{1,1}=U_{0,1}$.
		Then we also have a semi-direct product
		\[
		U_1
		=
		U_{1,1}\rtimes U_{1,0}.
		\]
		To decompose the double coset
		$U_?
		\varsigma
		U_?$
		into right cosets,
		it suffices to choose representatives in
		$U_?$
		of the quotient
		$U_?/
		\big(
		U_?
		\cap
		\varsigma
		U_?\varsigma^{-1}
		\big)$.
		Therefore to prove the lemma,
		it suffices to show that
		there is a common set of representatives
		$(u_i)_{i\in I}$ inside
		$U_1$
		such that
		$U_?\varsigma U_?=
		\sqcup_iu_i\varsigma U_?$
		($?=0,1$).
		For $?=0,1$,
		we have
		$\varsigma
		U_?
		\varsigma^{-1}
		=
		\varsigma U_{?,1}\varsigma^{-1}
		\rtimes
		\varsigma U_{?,0}\varsigma^{-1}$,
		and then
		$U_?\cap
		\varsigma U_?\varsigma^{-1}
		=
		(
		U_{?,1}\cap
		\varsigma
		U_{?,1}\varsigma^{-1}
		)
		\rtimes
		(
		U_{?,0}\cap
		\varsigma
		U_{?,0}
		\varsigma^{-1}
		)$,
		from which one deduces
		\[
		U_?/
		(U_?\cap
		\varsigma U_?\varsigma^{-1})
		=
		U_{?,1}/
		(U_{?,1}\cap
		\varsigma
		U_{?,1}\varsigma^{-1})
		\cdot
		U_{?,0}/
		(U_{?,0}\cap
		\varsigma
		U_{?,0}
		\varsigma^{-1}).
		\]
		Since
		$U_{1,1}=U_{0,1}$,
		it suffices to show that there is a common set of
		representatives
		$(v_k)_{k\in J}$
		inside
		$U_{1,0}$
		such that
		$U_{?,0}\varsigma U_{?,0}
		=
		\sqcup_kv_k\varsigma U_{?,0}$
		($?=0,1$).
		Write
		$\varsigma=
		\mathrm{diag}(a_1,a_2,a_3)$
		with
		$a_1,a_3\in
		B_{n_0}
		(F_v)$.
		We also write an element
		$b\in
		U_{0,0}$
		in the form of block matrix
		$b=(b_{i,j})_{i,j=1}^3$
		with
		$b_{1,1},b_{3,3}$ of size
		$n_0\times n_0$
		and
		$b_{2,2}$ of size
		$(n-2n_0)\times(n-2n_0)$.
		There are two cases to consider
		\begin{enumerate}
			\item 
			If $j\geq n_0$
			(thus $a_1=1=a_3$),
			then for $?=0,1$,
			$U_{?,0}\cap
			\varsigma U_{?,}\varsigma^{-1}$
			consists of those
			$b=(b_{i,j})_{i,j=1}^3\in U_{?,0}$ with
			\[
			\begin{cases*}
				b_{k,l},\,
				a_kb_{k,l}a_l^{-1} 
				\equiv
				0
				(\mathrm{mod}\,\varpi_v^m),
				&
				$k\neq l$,
				\\
				b_{k,k},\,
				a_kb_{k,k}a_k^{-1}
				(\mathrm{mod}\,\varpi_v^m)
				\in
				\mathrm{GL}_{n-2n_0}(\mathcal{O}_{F,v}/\varpi_v^m),
				&
				$k=2$,
				\\
				b_{k,k}
				(\mathrm{mod}\,\varpi_v^m)
				\begin{cases*}
					\in
					B_{n_0}
					(\mathcal{O}_{F,v}/\varpi_v^m),
					&
					$?=0$,
					\\
					=1,
					&
					$?=1$,
				\end{cases*}
				&
				$k=1,3$.
			\end{cases*}
			\]
			Note that elements $b$ in
			$U_{1,0}$
			are determined by the blocks
			$b_{k,l}$
			with
			$k<l$.
			Set
			$r_1=n_0,r_2=n-2n_0,r_3=n_0$.
			For $1\leq k<l\leq3$,
			we choose a set of representatives
			$I_{k,l}$ in
			$\mathrm{M}_{r_k,r_l}
			(\mathcal{O}_{F,v})$
			of the quotient
			$\mathrm{M}_{r_k,r_l}
			(\mathcal{O}_{F,v})/
			(\mathrm{M}_{r_k,r_l}\cap
			a_k\mathrm{M}_{r_k,r_l}(\mathcal{O}_{F,v})a_l^{-1})$.
			Then the set
			\begin{equation}\label{set of representatives-1}
				I=
				\big\{
				b\in U_{1,0}\mid 
				b_{1,1}=1=b_{3,3},
				b_{k,l}\in I_{k,l},
				\forall k<l,
				\,
				\&
				b_{k,l}=0,
				\forall
				k>l
				\big\}
			\end{equation}
			is a set of representatives for both
			$U_{1,0}/(U_{1,0}\cap
			\varsigma U_{1,0}\varsigma^{-1})$	
			and
			$U_{0,0}/(U_{0,0}\cap
			\varsigma U_{0,0}\varsigma^{-1})$:
			indeed, for the first case,
			note that
			for any
			$b\in
			U_{1,0}$,
			the matrices
			$b_{k,l}
			(\mathrm{mod}\,\varpi_v^m)$
			($k<l$)
			is necessarily the image of an element
			$b_{k,l}'$ in
			$I_{k,l}$
			mod $\varpi_v^m$.
			Now take an element
			$b'\in I$
			with
			$b'_{k,l}=b_{k,l}$
			for $k<l$
			(the existence is guaranteed by the surjectivity of
			$G(\mathcal{O}_{F,v})
			\rightarrow
			G(\mathcal{O}_{F,v}/\varpi_v^m)$),
			then it is clear that
			$(b')^{-1}b\in
			U_{1,0}\cap
			\varsigma U_{1,0}\varsigma^{-1}$;
			for the second case,
			for any $b\in U_{0,0}$,
			one can find an element
			$c=\mathrm{diag}(c_{1,1},c_{2,2},c_{3,3})
			\in
			U_{0,0}$
			such that
			$c_{k,k}
			\equiv
			b_{k,k}
			(\mathrm{mod}\,\varpi_v^m)$
			for all $k$.
			We set
			$d=bc^{-1}
			\in
			U_{1,0}$.
			As above we take
			$b'\in I$
			such that
			$(b')_{k,l}\equiv
			d_{k,l}
			(\mathrm{mod}\,\varpi_v^m)$
			for $k<l$.
			Then
			$(b')^{-1}d\in
			U_{1,0}
			\cap
			\varsigma
			U_{1,0}\varsigma^{-1}$
			and therefore
			$(b')^{-1}b\in
			U_{0,0}\cap
			\varsigma U_{0,0}\varsigma^{-1}$.
			Note that
			$I$ is independent of $m$.

			\item 
			If
			$j<n_0$,
			the argument is similar.
			For $k=1,3$,
			write
			$a_k=\mathrm{diag}((a_k)_1,(a_k)_2)$ in block matrix
			with
			$(a_1)_1,(a_3)_2$ of size $j\times j$.
			Similarly, write
			$b_{k,k}=((b_{k,k})_{s,t})_{s,t=1,2}$
			with
			$(b_{1,1})_{1,1},
			(b_{3,3})_{2,2}$ of size
			$j\times j$.
			For
			$?=0,1$,
			$U_{?,0}\cap
			\varsigma
			U_{?,0}\varsigma^{-1}$
			consists of those $b\in U_{?,0}$ with
			\[
			\begin{cases*}
				b_{k,l},\,
				a_kb_{k,l}a_l^{-1}		
				(\mathrm{mod}\,\varpi_v^m)
				\equiv
				0,
				&
				$k\neq l$,
				\\
				b_{k,k},\,
				a_kb_{k,k}a_k^{-1}
				(\mathrm{mod}\,\varpi_v^m)
				\in
				\mathrm{GL}_{n-2n_0}(\mathcal{O}_{F,v}/\varpi_v^m),
				&
				$k=2$,
				\\
				(b_{1,1})_{1,1},\,
				(b_{3,3})_{2,2}
				(\mathrm{mod}\,\varpi_v^m)
				\begin{cases*}
					\in
					B_j
					(\mathcal{O}_{F,v}/\varpi_v^m),
					&
					$?=0$,
					\\
					\equiv1,
					&
					$?=1$,
				\end{cases*}
				&
				\\
				(b_{1,1})_{2,2},\,
				(b_{3,3})_{1,1}
				(\mathrm{mod}\,\varpi_v^m)
				\begin{cases*}
					\in
					B_{n_0-j}
					(\mathcal{O}_{F,v}/\varpi_v^m),
					&
					$?=0$,
					\\
					\equiv1,
					&
					$?=1$,
				\end{cases*}
				\\
				(b_{1,1})_{2,1},\,
				(b_{3,3})_{1,2}^\mathrm{t}
				(\mathrm{mod}\,\varpi_v^{m+1})
				\equiv
				0.
				&
			\end{cases*}
			\]
			Note that
			$b_{3,3}
			(\mathrm{mod}\,\varpi_v^m)$
			is determined by
			$b_{1,1}
			(\mathrm{mod}\,\varpi_v^m)$.
			We fix a set of representatives
			$(I_{1,1})_{1,2}$
			in
			$\mathrm{M}_{j,n_0-j}
			(\mathcal{O}_{F,v})$
			for the quotient
			$\mathrm{M}_{j,n_0-j}
			(\mathcal{O}_{F,v}/\varpi_v)$.
			Write
			$I_{1,1}$
			for the set of elements
			$b_{1,1}\in
			\mathrm{M}_{j,j}(\mathcal{O}_{F,v})$
			such that
			$(b_{1,1})_{1,1}=(b_{1,1})_{2,2}=1,
			(b_{1,1})_{1,2}\in (I_{1,1})_{1,2},
			(b_{1,1})_{2,1}=0$.
			Then we take
			\begin{equation}\label{set of representatives-2}
				I:=
				\{
				b\in
				U_{1,0}\mid 
				b_{1,1}\in I_{1,1},
				b_{k,l}\in I_{k,l},
				\forall
				k<l,
				\,
				\&
				b_{k,l}=0,
				\forall
				k>l
				\}.
			\end{equation}
			Now one can verify as in the preceding case that
			$I$ is a common set of representatives for the quotients
			$U_{?,0}/
			(U_{?,0}\cap
			\varsigma U_{?,0}\varsigma^{-1})$
			($?=0,1$).
			Again we see that $I$ is independent of
			$m$.
		\end{enumerate}
	\end{proof}

	From this we deduce easily that
	\begin{corollary}
		For any integers $m_1>m_2>0$ and a smooth representation $W$ of
		$G(F_v)$, the inclusions
		\[
		W_1:=W^{U_0^{(n_0)}(v^{m_2})}
		\xhookrightarrow{\alpha_1}
		W_2:=W^{U_1^{(n_0)}(v^{m_2})}
		\xhookrightarrow{\alpha_2}
		W^{U_1^{(n_0)}(v^{m_1})}
		\]
		are compatible with the Hecke operators $U_v^{(j)}$ for any $j$, i.e. for any
		element $w_i\in W_i$ ($i=1,2$), one has $\alpha_i(U_v^{(j)}w_i)=U_v^{(j)}\alpha_i(w_i)$.
	\end{corollary}

	We need some auxiliary Hecke operators:
	\begin{definition}
		For each diagonal matrix
		$t=\mathrm{diag}(t_1,\cdots,t_{n_0})\in\mathrm{GL}_{n_0}(F_v)$, we put
		\[
		\hat{t}=
		\mathrm{diag}
		(t_{n_0},t_{n_0-1},\cdots,t_1).
		\]
		We denote the following two Hecke operators by the same symbol $V_t$
		(it should be clear from the context which operator $V_t$ refers to)
		\[
		V_t=
		[
		U_0^{(n_0)}(v)
		\mathrm{diag}
		(t,1,\hat{t}^{-1})
		U_0^{(n_0)}(v)
		],
		\quad
		[
		U_1^{(n_0)}(v)
		\mathrm{diag}
		(t,1,\hat{t}^{-1})
		U_1^{(n_0)}(v)
		].
		\]
	\end{definition}
	Then as in the preceding corollary,
	one has
	\begin{proposition}
		For a smooth representation $W$ of $G(F_v)$, the inclusion
		\[
		W^{U_0^{(n_0)}(v)}
		\hookrightarrow
		W^{U_1^{(n_0)}(v)}
		\]
		is compatible with the actions of the operator $V_t$.
	\end{proposition}
	\begin{proof}
		One can verify that for a matrix
		$a\in
		\mathrm{GL}_{n_0}(\mathcal{O}_{F,v})$
		with
		$a(\mathrm{mod}\,\varpi_v)
		\in
		B_{n_0}(\mathcal{O}_{F,v}/\varpi_v)$,
		if
		$tat^{-1}
		(\mathrm{mod}\,\varpi_v)
		\in
		U_{n_0}(\mathcal{O}_{F,v}/\varpi_v)$,
		then
		$a(\mathrm{mod}\,\varpi_v)
		\in
		U_{n_0}(\mathcal{O}_{F,v}/\varpi_v)$.
		From this one deduces
		\[
		U_1^{(n_0)}(v)
		\bigcap
		X
		U_1^{(n_0)}(v)
		X^{-1}
		=
		U_1^{(n_0)}(v)
		\bigcap
		X
		U_0^{(n_0)}(v)
		X^{-1},
		\]
		where $X=\mathrm{diag}(t,1,\hat{t}^{-1})$. Therefore the natural map between the quotients is injective
		\[
		U_1^{(n_0)}(v)/
		\big(
		U_1^{(n_0)}(v)\bigcap
		X
		U_1^{(n_0)}(v)
		X^{-1}
		\big)
		\rightarrow
		U_0^{(n_0)}(v)/
		\big(
		U_0^{(n_0)}(v)\bigcap
		X
		U_0^{(n_0)}(v)
		X^{-1}
		\big).
		\]
		On the other hand, it is easy to see that these two quotients have the same cardinality. Thus the above map is a bijection. Then one can choose a common set of representatives $(g_i\in U_1^{(n_0)}(v))_{i\in I}$ such that
		\[
		U_i^{(n_0)}(v)
		X
		U_i^{(n_0)}(v)
		=
		\bigsqcup_{i\in I}
		g_iXU_i^{(n_0)}(v),
		\quad
		i=0,1.
		\]
	\end{proof}
	
	For later computations, we need to make explicit the root datum for	$(G_v,\mathbf{B}_v,\mathbf{T}_v)$. For $i=1,\cdots,n_\mathrm{s}$,
	we write $e^\ast_i\colon\mathbf{T}_v\rightarrow\mathbb{G}_m$ for the character sending the element $\mathrm{diag}(t_1,\cdots,t_{n_\mathrm{s}},t,
	t_{n_\mathrm{s}}^{-1},\cdots,t_1^{-1})$ to $t_i$ (here $t$ is either an element in $\mathbb{G}_m$ or an empty set). Similarly, we write $e_{\ast,i}\colon
	\mathbb{G}_m\rightarrow\mathbf{T}_v$ for the cocharacter sending
	$t$ to the element $\mathrm{diag}(1_{i-1},t,1_{n-2i},t^{-1},1_{i-1})$.
	The maximal split sub-torus $\mathbf{T}_v^\mathrm{s}$ of $\mathbf{T}_v$
	consists of elements $\mathrm{diag}(t_1,\cdots,t_{n_\mathrm{s}},1_{n-2n_\mathrm{s}},
	t_{n_\mathrm{s}}^{-1},\cdots,t_1^{-1})$ and $\mathbf{T}_v$ is the centralizer of $\mathbf{T}_v^\mathrm{s}$ inside $G_v$. Therefore
	\[
	X^\ast
	(\mathbf{T}_v^\mathrm{s})
	=
	\mathbb{Z}\langle e_1^\ast,\cdots,e_{n_\mathrm{s}}^\ast\rangle,
	\quad
	X_\ast(\mathbf{T}_v^\mathrm{s})
	=
	\mathbb{Z}
	\langle
	e_{\ast,1},\cdots,e_{\ast,n_\mathrm{s}}\rangle.
	\]
	The Weyl group $W^\mathrm{s}$ of $\mathbf{T}_v^\mathrm{s}$ is then given by the semi-direct product $\{\pm1\}^{n_\mathrm{s}}\rtimes S_{n_\mathrm{s}}$, except in the case where $G_v$ is non-split (thus $G^\ast=\mathrm{SO}_n^\eta$ with $n$ even), where $W^\mathrm{s}$ consists of $(x_1,\cdots,x_{n_\mathrm{s}};\sigma)\in\{\pm1\}^{n_\mathrm{s}}\rtimes S_{n_\mathrm{s}}$ such that $\prod_ix_i=1$. Write the spherical Hecke operator (\emph{cf.} Definition \ref{spherical Hecke algebra})
	\begin{equation}\label{spherical Hecke operators}
		T_v^{(j)}
		=
		[
		G_v(\mathcal{O}_{F,v})
		\varsigma_{v,j}
		G_v(\mathcal{O}_{F,v})
		].\index{T@$T_v^{(j)}$}
	\end{equation}
	The Satake isomorphism
	(\ref{Satake isomorphism})
	becomes,
	up to a non-zero scalar
	\begin{align}\label{Satake isomorphism for quasi-split}
		\begin{split}
			\mathbb{T}_v
			&
			\simeq
			\mathbb{Z}[\frac{1}{q_v}]
			[Z_1^{\pm1},\cdots,Z^{\pm1}_{n_\mathrm{s}}]^{W^\mathrm{s}}
			\big(
			\simeq	  
			\mathbb{Z}[\frac{1}{q_v}]
			[X_\ast(\mathbf{T}_v^\mathrm{s})]^{W^\mathrm{s}}
			\big),  
			\\
			T_v^{(j)}
			&
			\mapsto
			\begin{cases*}
				q_v^{\frac{n_\mathrm{s}(n_\mathrm{s}-1)}{2}+(n-n_\mathrm{s}-j)j}
				\cdot
				\sum_{\sigma\in W^\mathrm{s}}
				s_j(\sigma(Z_1,Z_2,\cdots,Z_{n_\mathrm{s}})),
				&
				$G^\ast=\mathrm{SO}_n^\eta$;
				\\
				q_v^{\frac{n_\mathrm{s}(n_\mathrm{s}+1)}{2}+(n-n_\mathrm{s}-j)j}
				\cdot
				\sum_{\sigma\in W^\mathrm{s}}
				s_j(\sigma(Z_1,Z_2,\cdots,Z_{n_\mathrm{s}})),
				&
				$G^\ast=\mathrm{Sp}_n$.
			\end{cases*}
		\end{split}
	\end{align}
	Here  $v$ is the place corresponding to the prime $\mathfrak{q}$, $s_j(Z_1,\cdots,Z_{n_\mathrm{s}})$ is the elementary symmetric polynomial of degree $j$ on the variables $Z_1,\cdots,Z_{n_\mathrm{s}}$ and $\sigma=(\sigma_1,\cdots,\sigma_{n_\mathrm{s}})\in W^\mathrm{s}$
	takes $(Z_1,Z_2,\cdots,Z_{n_\mathrm{s}})$ to $(Z_1^{\sigma_1},\cdots,Z_{n_\mathrm{s}}^{\sigma_{n_\mathrm{s}}})$.
	The choice of the factor before the summation $\sum_{\sigma\in W^\mathrm{s}}$ will be justified in the following
	lemma:
	\begin{lemma}\label{Hecke polynomial for unramified rep}
		Let $\chi_1,\cdots,\chi_{n_\mathrm{s}}$	be unramified characters of
		$F_v^\times$. Write	$\mathbb{C}(\chi_1,\cdots,\chi_{n_\mathrm{s}})$
		for the complex representation space of	$\mathbf{B}_v(F_v)$
		induced by the characters $(\chi_1,\cdots,\chi_{n_\mathrm{s}})$
		if $n=2n_\mathrm{s}$, induced by the characters $(\chi_1,\cdots,\chi_{n_\mathrm{s}},\chi_0)$ if $n>2n_\mathrm{s}$
		(where $\chi_0$ is the trivial character on the non-split torus $\mathbf{T}_v/\mathbf{T}_v^\mathrm{s}$). Then the space of
		$G_v(\mathcal{O}_{F,v})$-invariants	$(\mathrm{n}\text{-}\mathrm{Ind}_{\mathbf{B}_v(F_v)}^{G_v(F_v)}
		\mathbb{C}
		(\chi_1,\cdots,\chi_{n_\mathrm{s}}))^{G_v(\mathcal{O}_{F,v})}$
		is of dimension $1$ and the spherical Hecke operators $T_v^{(j)}$
		acts on it by the scalar
		\[
		\begin{cases*}
			q_v^{\frac{n_\mathrm{s}(n_\mathrm{s}-1)}{2}
				+
				(n-n_\mathrm{s}-j)j}
			\sum_{I,J}
			\prod_{i\in J}
			(\chi_i(\varpi_v)q_v^{-1})
			\prod_{i\in I\backslash J}
			(\chi_i(\varpi_v)q_v^{-1})^{-1},
			&
			$G^\ast=\mathrm{SO}_n^\eta$;
			\\
			q_v^{\frac{n_\mathrm{s}(n_\mathrm{s}+1)}{2}
				+
				(n-n_\mathrm{s}-j)j}
			\sum_{I,J}
			\prod_{i\in J}
			(\chi_i(\varpi_v)q_v^{-1})
			\prod_{i\in I\backslash J}
			(\chi_i(\varpi_v)q_v^{-1})^{-1},
			&
			$G^\ast=\mathrm{Sp}_n$,
		\end{cases*}
		\]
		where $I$ runs through all subsets of $\{1,2,\cdots,n_\mathrm{s}\}$
		of cardinal $j$ and $J$ runs through all subsets of $I$. Now using the Satake isomorphism	from (\ref{Satake isomorphism}), if a Hecke operator $T$ corresponds to a rational polynomial $P(Z_1,\cdots,Z_{n_\mathrm{s}})$
		via (\ref{Satake isomorphism for quasi-split}),	then the eigenvalue of
		$T$ on the above space is given by $P(\chi_1(\varpi_v)q_v^{-1},\cdots,	\chi_{n_\mathrm{s}}(\varpi_v)q_v^{-1})$
	\end{lemma}
	\begin{proof}
		We prove the case $G^\ast=\mathrm{SO}_n^\eta$, the other case $G^\ast=\mathrm{Sp}_n$ is proved in the same way.
		Recall the modular function
		$\Delta$
		on
		$\mathbf{B}_v(F_v)$:
		for any
		$b\in
		\mathbf{B}_v(F_v)$
		with diagonal entries
		$(b_{i,i})_{i=1}^n$,
		one has
		\[
		\Delta
		(b)
		=
		|
		b_{1,1}^{n_\mathrm{s}-1}b_{2,2}^{n_\mathrm{s}-3}
		\cdots
		b_{n_\mathrm{s},n_\mathrm{s}}^{-(n_\mathrm{s}-1)}
		(
		b_{1,1}\cdots
		b_{n_\mathrm{s},n_\mathrm{s}}
		)^{n_\mathrm{s}+1}
		|_v.
		\]
		Here
		$|-|_v$
		is the $v$-adic absolute value
		$F_v$
		with
		$|p|_v=1/p$.
		Since we have the Iwasawa decomposition
		\begin{equation}\label{Iwasawa decomposition}
			\tag{$\ast$}
			G_v(F_v)
			=
			\mathbf{B}_v(F_v)
			G_v(\mathcal{O}_{F,v})
			=
			G_v(\mathcal{O}_{F,v})
			\mathbf{B}_v(F_v),
		\end{equation}
		the space
		$V:=
		(\mathrm{n}\text{-}\mathrm{Ind}_{\mathbf{B}_v(F_v)}^{ G_v(F_v)}
		(\chi_1,\cdots,\chi_{n_\mathrm{s}}))^{ G_v(\mathcal{O}_{F,v})}$
		is thus of dimension $1$. We take the following basis for this space
		\[
		\phi
		\colon
		G_v(F_v)=
		\mathbf{B}_v(F_v)
		G_v(\mathcal{O}_{F,v})
		\rightarrow
		\mathbb{C}
		=
		\mathbb{C}(\chi_1,\cdots,\chi_{n_\mathrm{s}}),
		\quad
		bu
		\mapsto
		\Delta(b)^{1/2}
		\prod_{i=1}^{n_\mathrm{s}}
		\chi_i(b_{i,i}).
		\]
		
		To write down explicitly the action of the Hecke operator $T_v^{(j)}$ on
		$V$, we need to find a set of representatives in the decomposition
		\[
		G_v(\mathcal{O}_{F,v})
		\varsigma_{v,j}
		G_v(\mathcal{O}_{F,v})
		=
		\bigsqcup_ig_i G_v(\mathcal{O}_{f,v}).
		\]
		For any $g\in G_v(\mathcal{O}_{F,v})$, according to the Iwasawa decomposition (\ref{Iwasawa decomposition}), we can write
		$g\varsigma_{v,j}=bh$ with $b\in\mathbf{B}_v(F_v)$ and $h\in G_v(\mathcal{O}_{F,v})$. Therefore it suffices to find representatives $g_i$
		in the Borel subgroup $\mathbf{B}_v(F_v)$. Note that in the matrix
		$g\varsigma_{v,j}=(a_{k,l})$, for each row, the minimal $v$-adic valuation
		of all the entries in this row is in $\{-1,0,1\}$. Thus the diagonal entries of
		$b$ have $v$-adic valuations in	$\{-1,0,1\}$. Let $\mathfrak{I}$ be the power set of $\{1,2,\cdots,n_\mathrm{s}\}$ and for each $I\in\mathfrak{I}$ write
		$2^I$ for the power set of $I$. For a subset $I$ of $\{1,2,\cdots,n_\mathrm{s}\}$, we write $I^c$ for the set $\{n-i\mid i\in I\}$.
		Fix a set of representatives $X$, resp. $Y$ in $\mathcal{O}_{F,v}$ for the quotient $\mathcal{O}_{F,v}/\varpi_v$, resp. $\mathcal{O}_{F,v}/\varpi_v^2$
		such that $0\in X \subset Y$. For each pair $(I,J)$ with $I\in\mathfrak{I}$ and
		$J\subset I$, we consider the set $\mathfrak{J}_1(I,J)$ of elements $b$ in
		$\mathbf{B}_v(F_v)$ satisfying the following conditions:
		\begin{enumerate}
			\item 
			For $1\leq i\leq n/2$,
			\[
			b_{i,i}
			=
			\begin{cases*}
				\varpi_v,
				&
				if
				$i\in J$;
				\\
				\varpi_v^{-1},
				&
				if
				$i\in I\backslash J$;
				\\
				1,
				&
				otherwise.
			\end{cases*}
			\]

			\item 
			For $1\leq i<l\leq n_\mathrm{s}$,
			\[
			b_{i,l}
			\begin{cases*}
				\in
				Y,
				&
				if
				$i\in J$
				and $l\notin J$;
				\\
				\in
				X,
				&
				if
				$i\notin I$
				and $l\in I$;
				\\
				=0,
				&
				otherwise.
			\end{cases*}
			\]
			
			\item 
			For
			$l=n_\mathrm{s}+1$
			(thus $n$ is odd),
			\[
			b_{i,l}
			\begin{cases*}
				\in X,
				&
				if
				$i\in I$,
				\\
				=0,
				&
				otherwise
			\end{cases*}
			\]
		\end{enumerate}
		Note that these conditions uniquely determine an element $b$ in $\mathbf{B}_v(F_v)$. We consider also the set $\mathfrak{J}_2(I,J)$
		of elements $b\in\mathbf{B}_v(F_v)$ satisfying the following conditions
		\[
		b_{i,l}
		\begin{cases*}
			=1,
			&
			if
			$1\leq i=l\leq n$;
			\\
			\in
			Y,
			&
			if
			$i-(n-n_\mathrm{s}),l\in J$;
			\\
			\in
			X,
			&
			if
			$1\leq i\leq n_\mathrm{s}$
			and
			$l\in I^c$;
			\\
			=0,
			&
			$1\leq i<l\leq n-n_\mathrm{s}$.
		\end{cases*}
		\]
		Again these conditions determine a unique element in
		$\mathbf{B}_v(F_v)$.
		Then we set
		\[
		\mathfrak{J}(I,J)
		=
		\{
		g_1g_2\mid 
		g_1\in\mathfrak{J}_1(I,J),
		\,
		g_2\in\mathfrak{J}_2(I,J)
		\}.
		\]
		One verifies the following decomposition
		\[
		G_v(\mathcal{O}_{F,v})
		\varsigma_{v,j}
		G_v(\mathcal{O}_{F,v})
		=
		\bigsqcup_{I\in\mathfrak{I},J\subset I}
		\bigsqcup_{b\in \mathfrak{J}(I,J)}
		b G_v(\mathcal{O}_{F,v}).
		\]

		For each pair
		$(I,J)$,
		we write
		$s:=\# (I)$
		and
		$t:=
		\# (J)\leq s$.
		We list the elements
		as follows
		\begin{align*}
			J
			&
			=\{i_1<i_2<\cdots<i_t\},
			\\
			I\backslash J
			&
			=
			\{
			m_1<m_2<\cdots<m_{s-t}
			\},
			\\
			\{1,2,\cdots,n_\mathrm{s}\}\backslash I
			&
			=
			\{
			j_1<j_2<\cdots<j_{n_\mathrm{s}-s}
			\}.
		\end{align*}
		Then we have
		\begin{align*}
			\log_{q_v}(\# \mathfrak{J}(I,J))
			=
			&
			2\cdot\frac{t(t-1)}{2}
			+
			(n_\mathrm{s}-s)s
			+
			\frac{(n_\mathrm{s}-s)(n_\mathrm{s}-s-1)}{2}
			+
			(n-2n_\mathrm{s})s
			\\
			&
			+
			2\cdot
			\sum_{k=1}^t
			(n_\mathrm{s}-i_k-(t-k))
			+
			\sum_{k=1}^{n_\mathrm{s}-s}
			(n_\mathrm{s}-j_k-(n_\mathrm{s}-s)).
		\end{align*}
		We can rewrite the last two summations
		$\sum_{k=1}^t$ and $\sum_{k=1}^{n_\mathrm{s}-s}$  
		as follows
		\begin{align*}
			&
			2\cdot
			\sum_{k=1}^t
			(n_\mathrm{s}-i_k-(t-k))
			+
			\sum_{k=1}^{n_\mathrm{s}-s}
			(n_\mathrm{s}-j_k-(n_\mathrm{s}-s))
			\\
			=
			&
			\bigg(  
			\sum_{k=1}^t
			(n_\mathrm{s}-i_k-(t-k))
			+
			\sum_{k=1}^{n_\mathrm{s}-s}
			(n_\mathrm{s}-j_k-(n_\mathrm{s}-s))
			+
			\sum_{k=1}^{s-t}
			(n_\mathrm{s}-m_k-(s-t))
			\bigg)
			\\
			&
			+
			\bigg(   
			\sum_{k=1}^t
			(n_\mathrm{s}-i_k-(t-k))
			-
			\sum_{k=1}^{n_\mathrm{s}-s}
			(n_\mathrm{s}-j_k-(n_\mathrm{s}-s))
			\bigg)
			\\
			=
			&
			\bigg(
			n_\mathrm{s}^2
			-\frac{n_\mathrm{s}(n_\mathrm{s}+1)}{2}
			-\frac{t(t-1)}{2}
			-\frac{(n_\mathrm{s}-s)(n_\mathrm{s}-s-1)}{2}
			-\frac{(s-t)(s-t-1)}{2}
			\bigg)
			\\
			&
			+
			\bigg(
			\sum_{k=1}^t
			(n_\mathrm{s}-i_k-(t-k))
			-
			\sum_{k=1}^{n_\mathrm{s}-s}
			(n_\mathrm{s}-j_k-(n_\mathrm{s}-s))
			\bigg).
		\end{align*}
		On the other hand,
		if we set
		$X_i=\chi_i(\varpi_v)$
		for
		$i=1,2,\cdots,n_\mathrm{s}$,
		then for each
		$b\in\mathfrak{J}(I,J)$,
		we have
		\[
		\phi(b)
		=
		\prod_{i\in J}X_iq_v^{-\frac{1}{2}(n_\mathrm{s}-2(i-1)+n_\mathrm{s}+1)}
		\prod_{i\in I\backslash J}
		X_i^{-1}
		q_v^{\frac{1}{2}(n_\mathrm{s}-2(i-1)+n_\mathrm{s}+1)}.
		\]
		From this one deduces that in the expression
		$(T_v^{(j)}\phi)(1)
		=
		\sum_{I\in\mathfrak{I},J\in 2^I,b\in\mathfrak{J}(I,J)}
		\phi(b)$,
		the power of $q_v$ in each
		$\phi(b)$
		depends only on the set
		$\mathfrak{J}(I,J)$
		and therefore
		the power of
		$q_v$ in each term
		$\sum_{b\in\mathfrak{J}(I,J)}
		\phi(b)$
		is equal to
		\begin{align*}
			&
			\bigg(
			n_\mathrm{s}^2
			-\frac{n_\mathrm{s}(n_\mathrm{s}+1)}{2}
			-\frac{t(t-1)}{2}
			-\frac{(n_\mathrm{s}-s)(n_\mathrm{s}-s-1)}{2}
			-\frac{(s-t)(s-t-1)}{2}
			\bigg)
			\\
			+
			&
			\bigg(
			\sum_{k=1}^t
			(n_\mathrm{s}-i_k-(t-k))
			-
			\sum_{k=1}^{n_\mathrm{s}-s}
			(n_\mathrm{s}-j_k-(n_\mathrm{s}-s))
			\bigg)
			\\
			+
			&
			\bigg(
			-\frac{1}{2}\sum_{l\in J}(n_\mathrm{s}-(2l-1)+n_\mathrm{s}+1)
			+
			\frac{1}{2}\sum_{l\in I\backslash J}(n_\mathrm{s}-(2l-1)+n_\mathrm{s}+1)
			\bigg)
		\end{align*}
		which is the same as
		\[
		\frac{n_\mathrm{s}(n_\mathrm{s}-1)}{2}
		+
		(n-n_\mathrm{s}-s)s
		+
		(s-2t)
		=:
		f(s)+
		(-t+s-t).
		\]
		If we write
		$Y_i=X_iq_v^{-1}$,
		then we have
		\begin{align*}
			(T_v^{(j)}\phi)(1)
			&
			=
			\sum_{I,J,b}
			\phi(b)
			=
			\sum_{I,J}
			q_v^{f(\# (I))}
			\prod_{i\in J}Y_i
			\prod_{i\in I\backslash J}Y_i^{-1}
			\\
			&
			=
			q_v^{\frac{n_\mathrm{s}(n_\mathrm{s}-1)}{2}
				+
				(n-n_\mathrm{s}-j)j}
			\sum_{I,J}
			\prod_{i\in J}Y_i
			\prod_{i\in I\backslash J}Y_i^{-1},
		\end{align*}
		which is the expression in the proposition for
		$T_v^{(j)}$
		(note that $s=\#  I=j$).
		The rest follows from our normalization of
		the Satake isomorphism
		(\ref{Satake isomorphism for quasi-split}).
	\end{proof}

	Write $\mathbf{P}$ for the set of polynomials $P(X)\in\mathbb{C}[X]$ of the form $P(X)=\prod_{i=1}^r(X-x_i)(X-x_i^{-1})$ and $\widetilde{\mathbf{P}}$ the set $\widetilde{\mathbf{P}}$ of the polynomials $Q(X)\in\mathbb{C}[X]$
	of the form $Q(X)=\prod_{i=1}^r(X-(x_i+x_i^{-1}))$ (the coefficient field $\mathbb{C}$ can be replaced by any other algebraically closed field in both $\mathbf{P}$ and $\widetilde{\mathbf{P}}$). Then we have a natural bijection between $\mathbf{P}$ and $\widetilde{\mathbf{P}}$:
	\begin{equation}\label{degree 1 and 2 polynomials}
		\mathbf{P}
		\rightarrow
		\widetilde{\mathbf{P}},
		\quad
		P(X)=\prod_{i=1}^r(X-x_i)(X-x_i^{-1})
		\mapsto
		\widetilde{P}(X)=\prod_{i=1}^r(X-(x_i+x_i^{-1})).
	\end{equation}

	\begin{lemma}
		For an unramified irreducible representation $W$ of $ G(F_v)$, write
		$t^{(j)}$ for the eigenvalue of $T_v^{(j)}$	on $W^{ G(\mathcal{O}_{F,v})}$
		and set $\underline{t}=(t^{(1)},\cdots,t^{(n_\mathrm{s})})$. Then
		\[
		W^{U_1^{(n_0)}(v)}
		=
		W^{U_0^{(n_0)}(v)}.
		\]
		Moreover the characteristic polynomial of the operator $V_{1/\varpi_v}$
		on $W^{U_0^{(n_0)}(v)}$ divides the polynomial		$(P_{\underline{t}}(X))^{n_0\frac{(n_\mathrm{s}-1)!}{(n_\mathrm{s}-n_0)!}}$
		where $P_{\underline{t}}(X)$ is the polynomial in $\mathbf{P}$	corresponding to the following polynomial in $\widetilde{\mathbf{P}}$
		under the correspondence (\ref{degree 1 and 2 polynomials})
		\begin{equation*}
			\widetilde{P}_{\underline{t}}(X)
			=
			\begin{cases*}
				X^{n_\mathrm{s}}
				+
				\sum_{j=1}^{n_\mathrm{s}}
				(-1)^j
				q_v^{-\frac{n_\mathrm{s}(n_\mathrm{s}-1)}{2}+(j+n_\mathrm{s}-n)j}
				t^{(j)}
				X^{n_\mathrm{s}-j},
				&
				$G^\ast=\mathrm{SO}_n^\eta$;
				\\
				X^{n_\mathrm{s}}
				+
				\sum_{j=1}^{n_\mathrm{s}}
				(-1)^j
				q_v^{-\frac{n_\mathrm{s}(n_\mathrm{s}+1)}{2}+(j+n_\mathrm{s}-n)j}
				t^{(j)}
				X^{n_\mathrm{s}-j},
				&
				$G^\ast=\mathrm{Sp}_n$.
			\end{cases*}
		\end{equation*}
	\end{lemma}
	\begin{proof}
		We prove the case $G^\ast=\mathrm{SO}_n^\eta$, the case $G^\ast=\mathrm{Sp}_n$ is similar.
		For the first part,
		since
		$U_1^{(n_0)}(v)
		\mathbf{T}_v(\mathcal{O}_{F,v})
		=
		U_0^{(n_0)}(v)$
		and
		$\mathbf{T}_v(\mathcal{O}_{F,v})$
		acts trivially on
		$W$,
		we get the equality.

		To ease notations,
		we write
		\begin{align*}
			\widetilde{t}^{(j)}
			&
			=
			q_v^{-\frac{n_\mathrm{s}(n_\mathrm{s}-1)}{2}+(j+n_\mathrm{s}-n)j}
			t^{(j)},
			\\
			V
			&
			=V_{1/\varpi_v},
			\\	  	
			t
			&
			=
			\mathrm{diag}
			(\varpi_v^{-1}1_{n_0},1_{n-2n_0},\varpi_v1_{n_0}).
		\end{align*}
		For the second part,
		choose unramified characters
		$\chi_i\colon
		F_v^\times
		\rightarrow
		\mathbb{C}^\times$
		($i=1,2,\cdots,n_\mathrm{s}$)
		such that
		$\chi_i(\varpi_v)
		q_v^{-1}
		+
		(\chi_i(\varpi_v)q_v^{-1})^{-1}$
		are the roots of the polynomial
		(with multiplicities)
		\[
		X^{n_\mathrm{s}}
		-
		\widetilde{t}^{(1)}
		X^{n_\mathrm{s}-1}
		+
		\cdots
		+
		(-1)^{j}
		\widetilde{t}^{(j)}
		X^{n_\mathrm{s}-j}
		+
		\cdots
		+
		(-1)^{n_s}
		\widetilde{t}^{(n_s)}.
		\]
		By Frobenius reciprocity,		
		$W$ is a subquotient of the representation
		\[
		\widetilde{W}
		:=
		\mathrm{n}\text{-}\mathrm{Ind}_{\mathbf{B}_v(F_v)}^{ G_v(F_v)}
		(\chi_1,\cdots,\chi_{n_\mathrm{s}}).
		\]

		We need to compute the eigenvalues of
		$V_t$ on
		$\widetilde{W}^{U_0^{(n_0)}(v)}$.
		Consider the double quotient
		\[
		\mathbf{B}_v(F_v)\backslash
		G(F_v)/U_0^{(n_0)}(v),
		\]
		which is in bijection with
		the right quotient
		$ G(\kappa_v)
		/U_0^{(n_0)}(\kappa_v)$
		where
		$U_0^{(n_0)}(\kappa_v)$
		denotes the image of
		$U_0^{(n_0)}(v)$
		under the reduction map
		$ G(\mathcal{O}_{F,v})
		\rightarrow
		G(\kappa_v)$.
		Note that
		$U_0^{(n_0)}(\kappa_v)$
		is a parabolic subgroup of
		$ G(\kappa_v)$.
		Write
		\begin{align*}
			B_r
			:=
			&
			\begin{cases*}
				\begin{pmatrix}
					0 & 0 & 1 \\
					0 & 1_{r-2} & 0 \\
					1 & 0 & 0
				\end{pmatrix},
				&
				$r\geq2$;
				\\
				1,
				&
				$r=1$
			\end{cases*}
			\\
			S_i'
			:=
			&
			\{
			\mathrm{diag}(1_{i-1},B_r,1_{n_\mathrm{s}-i+1-r})
			\mid r=1,2,\cdots,n_\mathrm{s}-i+1
			\},
			\quad
			i=1,2,\cdots,n_0,
			\\
			S_i
			:=
			&
			\{
			\mathrm{diag}(t,1,\hat{t}^{-1})\mid 
			t\in S_i'
			\},
			\\
			S
			:=
			&
			S_{n_0}\times S_{n_0-1}\times\cdots\times S_1
			\subset
			G(\kappa_v),
			G(F_v).
		\end{align*}
		Then it is easy to see that we have the following decompositions (the second follows from the first)
		\[
		G(\kappa_v)
		=
		\bigsqcup_{s\in S}
		sU_0^{(n_0)}(\kappa_v),
		\quad
		G(F_v)
		=
		\bigsqcup_{s\in S}
		\mathbf{B}_v(F_v)sU_0^{(n_0)}(v).
		\]
		We take the following basis for	$\widetilde{W}^{U_0^{(n_0)}(v)}$: for each $s\in S$, let $\phi_s\colon	G(F_v)\rightarrow\mathbb{C}$ be the characteristic function of	$\mathbf{B}_v(F_v)sU_0^{(n_0)}(v)$.	Thus
		$\widetilde{W}^{U_0^{(n_0)}(v)}=\mathbb{C}\langle\phi_s\rangle_{s\in S}$.
		Now for each $s\in S$, if $V\phi_s=\sum_{s'\in S}a_{s'}\phi_{s'}$
		for some $a_{s'}\in\mathbb{C}$,	then $(V\phi_s)(s'')=\sum_{s'\in S}a_{s'}\phi_{s'}(s'')=a_{s''}$. From this we get
		\[
		V\phi_s
		=
		\sum_{s'\in S}
		(V\phi_s)(s')\phi_{s'}.
		\]
		Let $Z_1$ be a set of representatives in $\mathrm{U}_{n_0}(\mathcal{O}_{F,v})$ for its quotient $\mathrm{U}_{n_0}(\kappa)$, $Z_2$ a set of representatives in
		$\mathrm{M}_{n-2n_0,n_0}(\mathcal{O}_{F,v})$ for its quotient
		$\mathrm{M}_{n-2n_0,n_0}(\kappa_v)$, $Z_3$
		a set of representatives in
		$\mathrm{M}_{n_0,n_0}^\mathrm{as}(\mathcal{O}_{F,v})$
		for its quotient
		$\mathrm{M}_{n_0,n_0}^\mathrm{as}(\kappa_v)$.
		Here
		the superscript
		$^\mathrm{as}$
		means the set of anti-symmetric matrices.
		We can take the following set of representatives for
		the double coset
		decomposition
		\[
		U_0^{(n_0)}(v)
		\mathrm{diag}(\varpi_v^{-1}1_{n_0},1_{n-2n_0},\varpi_v1_{n_0})
		U_0^{(n_0)}(v)
		=
		\bigsqcup_{x\in X}XU_0^{(n_0)}(v)
		\]
		with
		\[
		X=
		\left\{
		\begin{pmatrix}
			\varpi_v^{-1}a & 0 & \varpi_v^{-1}aA_{n_0}cA_{n_0}
			\\
			b & 1 & 0
			\\
			\varpi_v g & \varpi_v h & \varpi_v i 
		\end{pmatrix}\in G(F_v)
		\bigg|
		a\in Z_1,\,
		b\in Z_2,\,
		c\in Z_3
		\right\}.
		\]
		Note $h,i,g$ are determined by $a,b,c$.
		We next compute
		$(V\phi_s)(s')=
		\sum_{x\in X}
		\phi_s(s'x)$.
		Set
		\[
		X'=\{t^{-1}x\mid x\in X\}.
		\]
		Then for
		$x'\in X'$,
		\[
		\phi_s(s't^{-1}x')
		=
		\phi_s(s't^{-1}(s')^{-1}\cdot s'x')
		=
		\Delta(s't^{-1}(s')^{-1})
		(\prod_i\chi_i)(s't^{-1}(s')^{-1})
		\phi_s(s'x').
		\]

		We identify $S$ with a subset of the permutation group $P_{n_\mathrm{s}}$
		of $n_0$ elements $\{1,2,\cdots,n_\mathrm{s}\}$ as follows:	let
		$E:=(E_1,E_2,\cdots,E_{n_\mathrm{s}})$ be the standard basis for
		$\mathcal{O}_{F,v}^{n_\mathrm{s}}$ and $P_{n_\mathrm{s}}$ acts on $E$ by
		$\sigma E_i:=E_{\sigma^{-1}i}$ for any $\sigma\in P_{n_\mathrm{s}}$ and $1\leq i\leq n_0$, which gives an embedding $P_{n_\mathrm{s}}	\hookrightarrow\mathrm{GL}_{n_0}(\mathcal{O}_{F,v})$. Each set $S_i'$ can be identified with a subset of $P_{n_\mathrm{s}}$ in a natural way.	Since each element $s\in S$ is a product of matrices of the form $\mathrm{diag}(t,1,\hat{t}^{-1})$ with $t\in S_i'$, we get a natural inclusion of
		$S$ into $P_{n_\mathrm{s}}$. Note that if we identify $P_{n_\mathrm{s}-n_0}$ with a subgroup of $P_{n_\mathrm{s}}$
		by letting $P_{n_\mathrm{s}-n_0}$ act on the set $\{n_0+1,n_0+2,\cdots,n_\mathrm{s}\}$ in the obvious way,
		then $S$ can be seen as a set of representatives for the quotient
		$P_{n_\mathrm{s}}/P_{n_\mathrm{s}-n_0}$. One checks that an element
		$g\in G(\mathcal{O}_{F,v})$ lies in $\mathbf{B}_v(\mathcal{O}_{F,v})sU_0^{(n_0)}(v)$ (equivalently,
		the image $\overline{g}\in G(\kappa_v)$ lies in $\mathbf{B}_v(\kappa_v)sU_0^{(n_0)(\kappa_v)}$)	if and only if
		for each $1\leq i\leq n_0$, $s(i)$ (viewed as an element in $P_{n_\mathrm{s}}$ as above) is the smallest integer such that the row matrix $(\underbrace{0,\cdots,0}_\text{$(i-1)$ terms},1,\underbrace{0,\cdots,0}_\text{$(n-i)$ terms})$ lies in the	$\kappa_v$-linear span of the modulo $\varpi_v$ row vectors of the
		$(i+s(i)+1)$-th row to the $n_\mathrm{s}$-row of $g$ (\emph{cf.} \cite[p.85]{ClozelHarrisTaylor2008}).

		We define an order $\prec$ on $S$ as follows:
		\[
		s'\prec s
		\text{ if }
		s'(i)\leq s(i)\, \forall 
		1\leq i \leq n_0.
		\]
		Then for $x\in X$, we have $\phi_s(s'x)=0$ unless $s'\prec s$. Let the set
		$\mathbf{S}=\{1,2,\cdots,\# (S)\}$ be equipped with the total order
		$\leq$ from $\mathbb{Z}$. Choose a bijective map $\psi\colon	\mathbf{S}\rightarrow S$ such that $\psi(a)\prec\psi(a')$ implies
		$a\leq a'$ (any partial order can be extended to a total order). Write the basis of $\widetilde{W}$ in the following order
		\[
		(e_1:=\phi_{\psi(1)},
		e_2:=\phi_{\psi(2)},\cdots,e_{\# (S)}:=\phi_{\psi(\# (S))}).
		\]
		Then we deduce that the matrix of $V$ under this basis is upper triangular with diagonal entries $((V\phi_s)(s))_{s\in S}$. Since $\phi_s(st^{-1})=\Delta(st^{-1}s^{-1})
		(\prod_i\chi_i)(st^{-1}s^{-1})\phi_s(s)$, the eigenvalues of $V$ on
		$\widetilde{W}$ are
		\[
		\left\{
		\Delta(st^{-1}s^{-1})
		(\prod_i\chi_i)(st^{-1}s^{-1})
		\sum_{x'\in X'}\phi_s(sx')
		\right\}_{s\in S}.
		\]
		By definition of $t$, we have
		\begin{align*}
			\Delta(st^{-1}s^{-1})
			&
			=
			\prod_{i=1}^{n_0}
			q_v^{-\frac{n_\mathrm{s}+1-2s(i)+n_\mathrm{s}+1}{2}}
			=
			\prod_{i=1}^{n_0}
			q_v^{-(n_\mathrm{s}+1-s(i))},
			\\
			(\prod_i\chi_i)(st^{-1}s^{-1})
			&
			=
			\prod_{i=1}^{n_0}\chi_{s(i)}(\varpi_v).
		\end{align*}
		On the other hand, using the fact $\phi_s(s)=1$, we have
		\[
		\sum_{x'\in X'}\phi_s(sx')
		=
		\prod_{i=1}^{n_0}q_v^{n_\mathrm{s}-s(i)}
		\chi_{s(i)}(\varpi_v).
		\]
		Putting all these together, we get
		\[
		(V\phi_s)(s)
		=
		\prod_{i=1}^{n_0}
		\big(
		\chi_{s(i)}(\varpi_v)q_v^{-1}
		\big).
		\]
		Now counting the roots of characteristic polynomials with multiplicities,
		we see that the characteristic polynomial of $V$ on $\widetilde{W}^{U_0^{(n_0)}(v)}$ divides the polynomial
		$(P_{\underline{t}}(X))^{n_0\frac{(n_\mathrm{s}-1)!}{(n_\mathrm{s}-n_0)!}}$
		and so does the characteristic polynomial of $V$ on the subquotient
		$W^{U_0^{(n_0)}(v)}$ of $\widetilde{W}^{U_0^{(n_0)}(v)}$.	  	
	\end{proof}

	\section{Smooth representations of $ G(F_v)$}
	\label{smooth representations of G(F_v)}
	Let
	$v$ be a finite place of $F$
	where
	$G$ is split.
	In this section,
	we establish some more
	results on representations of
	$ G(F_v)$.
	These will be used in
	constructing the Taylor-Wiles systems.
	In the following
	we give details for the case
	$G(F_v)=\mathrm{SO}_n(F_v)$ with $n$ odd
	and put the necessary modifications for the cases $G(F_v)=\mathrm{SO}_n(F_v)$ with $n$ even or $G(F_v)=\mathrm{Sp}_n(F_v)$
	in remarks.

	Recall
	$W$
	is the Weyl group of the pair
	$( G_v,\mathbf{T}_v)$.
	For the triple
	$( G_v,\mathbf{B}_v,\mathbf{T}_v)$,
	we have the following set of simple positive roots
	\[
	\Delta=
	\Delta( G_v,\mathbf{B}_v,\mathbf{T}_v)
	=
	\{
	E^\ast_1,\cdots,E^\ast_{n_\mathrm{s}}
	\}
	\]
	where
	$E^\ast_i=e^\ast_i-e^\ast_{i+1}$
	for $i=1,\cdots,n_\mathrm{s}-1$
	and
	$E^\ast_{n_\mathrm{s}}=2e^\ast_{n_\mathrm{s}}$.
	Then
	$W$ can be identified with the group of all
	bijections $w$ of
	the set
	$\{\pm E_i^\ast\}_{i=1}^{n_\mathrm{s}}$
	onto itself
	such that
	$w(-E^\ast_i)=-w(E^\ast_i)$
	for all $i$.
	Thus $w$ is uniquely determined by
	its images on
	$E^\ast_1,\cdots,E^\ast_{n_\mathrm{s}}$
	and we write
	\[
	w=[0,a_1,\cdots,a_{n_\mathrm{s}}]
	\]
	where
	$w(E^\ast_i)=\mathrm{sgn}(a_i)E^\ast_{a_i}$
	(this is the \emph{window notation} of $w$ as in
	\cite[p.245]{BjonerBrenti2005}).
	For ease of notation,
	we will often identify
	$\pm E^\ast_i$
	with $\pm (n_\mathrm{s}+1-i)$
	and thus we can identify
	$W$ with the subgroup of
	bijections $w$ of
	$\{0,\pm1,\cdots,\pm n_\mathrm{s}\}$
	onto itself
	such that
	$w(-i)=-w(i)$
	for all $i$
	(therefore
	$[0,a_1,\cdots,a_{n_\mathrm{s}}]$
	is the bijection sending $i$ to $a_i$, $0$ to $0$).
	We fix the following set $\{s_0,\cdots,s_{n_\mathrm{s}}\}$ of generators for $W$ (\cite[p.246]{BjonerBrenti2005}):
	\[
	s_i=s_{E^\ast_i}=[1,2,\cdots,i-1,i+1,i,i+2,\cdots,n_\mathrm{s}]
	\]
	the reflection associated to $E^\ast_i$
	for $i=1,\cdots,n_\mathrm{s}-1$
	and
	\[
	s_0=[-1,2,3,\cdots,n_\mathrm{s}].
	\]
	\begin{remark}\rm
		\begin{enumerate}
			\item 
			For $G(F_v)=\mathrm{SO}_n(F_v)$ with $n$ even,
			$W$ can be identified with the group of bijections $w$ of
			the set
			$\{\pm E_i^\ast\}_{i=1}^{n_{\mathrm{s}}}$
			onto itself such that
			$w(-E_i^\ast)=-w(E_i^\ast)$
			for all $i$
			and the number of negative values in
			$(w(E_i^\ast))_{i=1}^{n_{\mathrm{s}}}$
			is \emph{even}.
			Moreover we put
			$s_0=[-2,-1,3,4,\cdots,n_{\mathrm{s}}]$ and the other $s_i$ are the same as above (\cite[p.253]{BjonerBrenti2005}).

			\item 
			For $G(F_v)=\mathrm{Sp}_n(F_v)$, $W$ is the group of all bijections $w$ of the set $\{\pm E_i^\ast\}_{i=1}^{n_\mathrm{s}}$ onto itself such that $w(-E_i^\ast)=w(E_i^\ast)$ for all $i$. Moreover, the generators $s_i$ of $W$ is the same as the case $G(F_v)=\mathrm{SO}_{n+1}(F_v)$.
		\end{enumerate}
		
	\end{remark}
	For any subset
	$\Theta
	=\{i_1,i_1+1,\cdots,i_1+t_1,i_2,i_2+1,\cdots,
	i_r,i_r+1,\cdots,i_r+t_r,\}
	\subset\{0,\cdots,n_\mathrm{s}-1\}
	$
	with
	$i_2>i_1+t_1+1,\cdots,i_r>i_{r-1}+t_{r-1}+1$,
	we write
	$W_\Theta$
	for the subgroup of $W$ generated by
	$s_i$ for all $i\in\Theta$
	and
	$P_\Theta
	\subset
	G_v$
	the parabolic subgroup
	associated to $\Theta$.
	We know that the coset
	$W_\Theta\backslash W$
	has a special set
	$W^\Theta$
	of representatives in
	$W$ given by the elements of minimal length in each coset
	(\emph{cf.}
	\cite[Lemma 1.1.2]{Casselman1995}).
	In our case,
	$W^\Theta$
	is explicitly given by the set of
	$w\in W$
	such that
	$w(i_r+t_r+1)>w(i_r+t_r)>w(i_r+t_r-1)>\cdots>w(i_r+1)>w(i_r)$,
	$w(i_{r-1}+t_{r-1}+1)>\cdots>w(i_{r-1}+1)>w(i_{r-1})$,
	...,
	$w(i_1+t_1+1)>\cdots>w(i_1)$.
	The quotient
	$W/W_\Theta$
	has also a special set
	$^\Theta W=(W^\Theta)^{-1}$
	of representatives in $W$
	consisting of elements of minimal length
	(\emph{cf.}
	\cite[Proposition 8.1.4\& (2.11)]{BjonerBrenti2005}).
	Now consider another subset
	$\Omega
	=\{j_1,\cdots,j_1+s_1,j_2,\cdots,j_2+s_2,\cdots,j_k,\cdots,j_k+s_k\}
	\subset
	\{0,1,\cdots,n_\mathrm{s}-1\}$,
	then
	the double coset
	$W_\Theta\backslash W/W_\Omega$
	has a set
	$^\Omega W^\Theta$
	of representatives consisting of those
	$w\in W$
	which is of minimal length in the double coset.
	For two integers
	$0\le
	q\le q'$,
	we write
	$I(q,q')=\{q,q+1,\cdots,q'\}$
	and put
	$I(q)=I(q,q)$.
	Then
	each $\Theta$
	corresponds to a unique partition of
	the set
	$\{0,1,\cdots,n_\mathrm{s}\}$
	given by
	\begin{align*}
		I(0,n_\mathrm{s})
		=
		&
		I(0)
		\sqcup
		\cdots
		\sqcup
		I(i_1-1)
		\sqcup
		I(i_1,i_1+t_1+1)
		\sqcup
		I(i_1+t_1+2)
		\sqcup
		\cdots
		\sqcup
		I(i_2-1)
		\sqcup
		\cdots
		\\
		&
		\sqcup
		I(i_2,i_2+t_2+1)
		\sqcup
		\cdots
		\sqcup
		I(i_r,i_r+t_r+1)
		\sqcup
		I(i_r+t_r+1)
		\sqcup
		\cdots
		\sqcup
		I(n_\mathrm{s}).
	\end{align*}

	Conversely, any such partition corresponds to a subset $\Theta$ of $I(0,n_\mathrm{s})$ (which is just the union of all these subsets with the maximal element removed in each subset). We label the above subsets as
	\[
	I(0,n_\mathrm{s})
	=
	I^\Theta_1
	\sqcup\cdots\sqcup
	I^\Theta_{q_\Theta}.
	\]
	where $q_\Theta=i_1+r+n_\mathrm{s}-(i_r+t_r+1)$. Similarly, we write
	$I(0,n_\mathrm{s})=I^\Omega_1\sqcup\cdots\sqcup	I^\Omega_{q_\Omega}$ and $^\Omega\widetilde{W}^\Theta$ for the set of matrices
	\begin{equation}\label{m}
		m=\begin{pmatrix}
			0 & a_{1,1} & \cdots & a_{1,q_\Theta} \\
			b_2 & a_{2,1} & \cdots & a_{2,q_\Theta} \\
			\vdots & \vdots & \ddots & \vdots \\
			b_{q_\Omega} & a_{q_\Omega,1} & \cdots & a_{q_\Omega,q_\Theta}
		\end{pmatrix}
	\end{equation}
	with non-negative integer entries such that
	$a_{i,1}+\cdots+a_{i,q_\Theta}=\# (I_i^\Omega)$,
	$a_{1,j}+\cdots+a_{q_\Omega,j}=\# (I_j^\Theta)$
	and
	$0\le b_j\le \# (I_j^\Omega)$
	for all $i,j$.
	For a subset
	$I\subset\{0,\pm1,\cdots,\pm n_\mathrm{s}\}$,
	we write
	\[
	I^+=\{x\in I\mid x>0\},
	\quad
	I^-=\{x\in I\mid x<0\},
	\quad
	|I|=
	\{|x|
	\big|
	x\in I\}.
	\]
	Then one can show
	\begin{lemma}
		There is a bijection between $^\Omega W^\Theta$ and	$^\Omega\widetilde{W}^\Theta$, which sends $w\in\,^\Omega W^\Theta$
		to the matrix $m\in\,^\Omega\widetilde{W}^\Theta$ whose entries are given by $a_{i,j}=\# (I_i^\Omega\cap|w(I_j^\Theta)|)$	and	$b_j=\# ((w^{-1}(I_j^\Omega))^-)$.
	\end{lemma}
	\begin{proof}
		Clearly the matrix $m$ given as in the lemma is indeed an element in
		$^\Omega\widetilde{W}^\Theta$.

		For injectivity, note that each
		$w\in\,^\Omega W^\Theta$ is determined by the intersections	$I_i^\Omega\cap w(I_j^\Theta)$	and $(I_i^\Omega)^-\cap w(I_j^\Theta)$.
		Moreover each intersection $I_i^\Omega\cap w(I_j^\Theta)$ consists of a sequence of consecutive integers: otherwise there are $k,k+1\in I_j^\Theta$
		such that $w(k)+1<w(k+1)$ and $w(k)+1,w(k+1)\in I_i^\Omega$. Say
		$w(k')=w(k)+1$ for some $k'\notin I_j^\Theta$. If $k'$ is smaller than the elements in $I_j^\Theta$, then $w^{-1}(w(k))>w^{-1}(w(k'))<w^{-1}(w(k+1))$, contradicting the characterization of $w$ being the element of minimal length in the coset
		$wW_\Omega$. The same result holds if we assume $k'$ greater than the elements in $I_j^\Theta$ or replace $I_i^\Omega$ by $(I_i^\Omega)^-$.

		The same argument shows that if $I_i^\Omega\cap w(I_j^\Theta)=\emptyset
		=I_i^\Omega\cap w(I_{j+2}^\Theta)$, then one necessarily has	$I_i^\Omega\cap w(I_{j+1}^\Theta)=\emptyset$. Again the same result holds if we exchange $I_i^\Omega$ and $(I_i^\Omega)^-$. This shows that
		$w$ is indeed determined by these entries $a_{i,j}$ and $b_j$, which shows the injectivity.

		It is also clear that from any such $m$, we can construct an element $w$, which gives the surjectivity of the natural map.
	\end{proof}
	\begin{remark}\rm
		\begin{enumerate}
			\item
			For $G(F_v)=\mathrm{SO}_n(F_v)$ with $n$ even,
			$^\Omega W^\Theta$
			is a subset of
			the $^\Omega W^\Theta$
			for the case $G(F_v)=\mathrm{SO}_{n+1}(F_v)$,
			consisting of those $w$
			such that
			$(w(E_i^\ast))_{i=1}^{n_{\mathrm{s}}}$
			is even.
			Correspondingly the
			$^\Omega\widetilde{W}^\Theta$
			for $n$ even is a subset of
			$^\Omega\widetilde{W}^\Theta$
			consisting of those $m$
			such that
			$b_2+\cdots+b_{q_{\Omega}}$
			is even.

			\item 
			For $G(F_v)=\mathrm{Sp}_n$, $^\Omega W^\Theta$ is the same as the case $\mathrm{SO}_{n+1}(F_v)$.
		\end{enumerate}
		
	\end{remark}

	For $\Theta$
	as above,
	we write
	$\mathfrak{p}_\Theta$\index{p@$\mathfrak{p}_\Theta$}
	for the preimage of
	$P_\Theta(\kappa_v)$
	under the projection
	$ G(\mathcal{O}_{F,v})
	\rightarrow
	G(\kappa_v)$,
	the parahoric subgroup associated to
	$\Theta$.
	We have the Levi decomposition
	\[
	P_\Theta
	=
	L_\Theta U_\Theta
	\]
	where
	$L_\Theta$
	is the Levi factor of
	$P_\Theta$,
	$U_\Theta$
	is the unipotent radical of
	$P_\Theta$.
	For $j=1,2,\cdots,q_\Theta$,
	we write
	$L_{\Theta,j}$
	for the $j$-th factor in
	$L_\Theta$.
	For an element
	$w\in\,^\Omega W^\Theta$
	(corresponding to $m\in\,^\Omega \widetilde{W}^\Theta$),
	we write
	$\mathfrak{p}^w_j$\index{p@$\mathfrak{p}^w_j$}
	for the parahoric subgroup of
	$L_{\Theta,j}(\mathcal{O}_{F_v})$
	corresponding to the partition
	$(I_j^\Theta)^\pm=
	\sqcup_j(w^{-1}(I_i^\Omega)\cap(I_j^\Theta)^\pm)$.
	We have
	\begin{proposition}
		For any
		$j=1,\cdots,q_\Theta$
		and
		$w\in\,^\Omega W^\Theta$,
		one has
		\[
		L_{\Theta,j}(F_v)\cap
		\left(
		w\mathfrak{p}_\Omega w^{-1}
		\right)
		=
		\mathfrak{p}^w_j.
		\]
		Let
		$\pi_j$ be a smooth representation of
		$L_{\Theta,j}(F_v)$ ($j=1,\cdots,q_{\Theta}$),
		then we have
		the following decomposition
		\[
		\left(
		\mathrm{n}\text{-}\mathrm{Ind}_{P_\Theta(F_v)}^{ G(F_v)}
		\pi_1\otimes\cdots\otimes\pi_{q_\Theta}
		\right)^{\mathfrak{p}_\Omega}
		=
		\bigoplus_{w\in\,^\Omega W^\Theta}
		\left(
		\pi_1^{\mathfrak{p}_1^w}
		\otimes
		\cdots
		\otimes
		\pi_{q_\Theta}^{\mathfrak{p}_{q_\Theta}^w}
		\right).
		\]
	\end{proposition}
	\begin{proof}
		The first part follows from the general result in
		\cite[Proposition 1.3.3]{Casselman1995}.
		The second part follows from the
		Bruhat decomposition
		$ G(F_v)
		=
		\sqcup_{w\in\,^\Omega W^\Theta}
		P_\Theta(F_v)
		w
		\mathfrak{p}_\Omega$,
		which gives
		\[
		\left(
		\mathrm{n}\text{-}\mathrm{Ind}_{P_\Theta(F_v)}^{ G(F_v)}
		\pi_1\otimes\cdots\otimes\pi_{q_\Theta}
		\right)^{\mathfrak{p}_\Omega}
		=
		\bigoplus_{w\in\,^\Omega W^\Theta}
		\left(
		\pi_1\otimes\cdots\otimes\pi_{q_\Theta}
		\right)^{P_\Theta(F_v)\cap(w\mathfrak{p}_\Omega w^{-1})}.
		\]
	\end{proof}
	For an unramified character
	$\chi$ of $F_v^\times$,
	we write
	$\mathrm{St}_m(\chi)$ for the Steinberg representation associated to $\chi$, which is
	the unique generic subquotient of
	the induced representation
	\[
	\mathrm{n}\text{-}\mathrm{Ind}_{\mathrm{B}_m(F_v)}^{\mathrm{GL}_m(F_v)}
	\left(
	\chi\otimes\chi|\cdot|\otimes\cdots\otimes\chi|\cdot|^{m-1}
	\right)
	\]
	where
	$\mathrm{B}_m\subset\mathrm{GL}_m$
	is the subgroup of upper triangular matrices.
	We write
	$\mathfrak{S}$\index{S@$\mathfrak{S}$}
	for the set of induced representations
	of
	$ G(F_v)$
	of the form
	$\mathrm{n}\text{-}
	\mathrm{Ind}_{P_\Gamma(F_v)}^{ G(F_v)}
	(\pi_1\otimes\pi_2\otimes\cdots\otimes\pi_r)$
	where the components of
	the Levi subgroup of $P_\Gamma(F_v)$
	is either $\mathrm{GL}_1(F_v)$,
	$\mathrm{GL}_2(F_v)$
	or
	the trivial group
	and moreover
	\begin{enumerate}
		\item
		each $\pi_i$
		($i=1,\cdots,r-1$)
		is either an unramified character
		or a Steinberg representation
		$\mathrm{St}_2(\chi)$
		with $\chi$ an unramified character and

		\item 
		$\pi_r$
		is either an unramified character,
		a Steinberg representation
		$\mathrm{St}_2(\chi)$
		with $\chi$ an unramified character, or the trivial representation of the trivial group.
	\end{enumerate}

	We will assume in the following that $\Omega$ is given by $\{2,3,\cdots,n_{\mathrm{s}}-1\}$. In particular, the Levi subgroup of $P_\Omega$ is of the form $\mathrm{GL}_{n_{\mathrm{s}}-1}\times\mathrm{GL}_1$.
	\begin{remark}\rm
		\begin{enumerate}
			\item
			For $G(F_v)=\mathrm{SO}_n(F_v)$ with $n$ even,
			$\mathfrak{S}$ is the set of induced representations
			$\mathrm{n}\text{-}\mathrm{Ind}_{P_\Gamma(F_v)}^{ G(F_v)}
			(\pi_1\otimes\cdots\otimes\pi_r)$ where the factors of the Levi subgroup of
			$P_\Gamma(F_v)$	is either $\mathrm{GL}_1(F_v)$,	$\mathrm{GL}_2(F_v)$
			and moreover each $\pi_i$ ($i=1,\cdots,r$) is either an unramified character or a Steinberg representation $\mathrm{St}_2(\chi)$ with $\chi$
			unramified character. We assume that $\Omega$ is given by $\{2,3,\cdots,n_{\mathrm{s}}-1\}$, so the Levi subgroup of $P_\Omega$ is of the form	$\mathrm{GL}_{n_{\mathrm{s}}-1}\times\mathrm{GL}_1$.

			\item 
			For $G(F_v)=\mathrm{Sp}_n(F_v)$, $\mathfrak{S}$ and $\Omega$ are the same as the case $\mathrm{SO}_n(F_v)$.
		\end{enumerate}
	\end{remark}


	\begin{lemma}\label{Iwahori spherical representations}
		Let $\pi$ be an irreducible representation of $ G(F_v)$. Suppose that
		$\pi^{\mathfrak{p}_\Omega}\ne0$, then we have an embedding
		\[
		\pi
		\hookrightarrow
		\Pi=\mathrm{n}\text{-}\mathrm{Ind}_{P_\Theta(F_v)}^{ G(F_v)}
		(\pi_1\otimes\cdots\otimes\pi_r)
		\]
		for some $\Pi\in\mathfrak{S}$. If moreover $\pi$ is ramified, we can choose $\Pi$ such that one of $\pi_i$ is $\mathrm{St}_2(\chi)$ with $\chi$
		unramified and all other $\pi_j$ are unramified characters.

	\end{lemma}
	\begin{proof}
		For the first part, note that $\mathrm{Iw}\subset\mathfrak{p}_{\Omega}$
		and thus Frobenius reciprocity implies that $\pi$ is a subrepresentation of unramified principal series of $ G(F_v)$.

		For the second part, assume $\pi$ ramified, then we have the following commutative diagram:
		\[
		\begin{tikzcd}
			\pi^{\mathfrak{p}_{\Omega}}
			\arrow[r,twoheadrightarrow]
			\arrow[d,hookrightarrow]
			&
			\pi_{U_{\Omega}(F_v)}^{L_{\Omega}(\mathcal{O}_{F,v})}
			\arrow[r,hookrightarrow]
			&
			\pi_{U_{\Omega}(F_v)}
			\arrow[d,twoheadrightarrow]
			\\
			\pi^{\mathrm{Iw}}
			\arrow[r,"\sim"]
			&
			\pi_{\mathbf{U}_v(F_v)}^{\mathbf{T}_v(\mathcal{O}_{F,v})}
			\arrow[r,hookrightarrow]
			&
			\pi_{\mathbf{U}_v(F_v)}
		\end{tikzcd}
		\]
		the isomorphism comes from \cite[Lemma 4.7]{Borel1976} and the horizontal surjection comes from \cite[Theorem 3.3.3]{Casselman1995}. This implies in particular that $\pi_{U_{\Omega}(F_v)}$ is an unramified admissible representation of $L_{\Omega}(F_v)$ (\cite[Theorem 3.3.1]{Casselman1995})
		and thus by Frobenius reciprocity, we can find an unramified admissible irreducible representation $\pi'\otimes\pi''$ of $L_{\Omega}(F_v)$ such that there is a quotient $\mathrm{n}\text{-}\mathrm{Ind}_{P_{\Omega}(F_v)}^{ G(F_v)}(\pi'\otimes\pi'')\twoheadrightarrow\pi$. Moreover we know that
		$\pi'$ and $\pi''$ are quotients of unramified principal series of the
		corresponding direct factors of $L_{\Omega}(F_v)$. So we can write
		\[
		\mathrm{n}\text{-}\mathrm{Ind}_{\mathbf{B}_v(F_v)}^{ G(F_v)}
		(\chi_1\otimes\cdots\otimes\chi_{n_{\mathrm{s}}})
		\twoheadrightarrow
		\pi
		\]
		with each $\chi_i$ an unramified character.

		We next show that the projection must factor through (up to permuting the order of these characters $\chi_1,\cdots,\chi_{n_{\mathrm{s}}-1}$):
		\[
		\mathrm{n}\text{-}\mathrm{Ind}_{\mathbf{B}_v(F_v)}^{ G(F_v)}
		(\chi_1\otimes\cdots\otimes\chi_{n_{\mathrm{s}}})
		\twoheadrightarrow
		\mathrm{n}\text{-}\mathrm{Ind}_{P_{\Omega}(F_v)}^{ G(F_v)}
		(\chi_1\otimes\cdots\otimes\chi_{n_{\mathrm{s}}-2}
		\otimes
		\mathrm{St}_2(\chi_{n_{\mathrm{s}}-1}))
		\]
		Take $P_{\Omega'}(F_v)$ to be the parabolic subgroup of $ G(F_v)$
		whose Levi factor is given by $\mathrm{GL}_1\times\cdots\mathrm{GL}_1\times\mathrm{GL}_2$. Then
		$P_{\Omega'}$ is contained in $P_{\Omega}$ and since
		$P_{\Omega}$ is a maximal parabolic, $\pi$ being ramified implies
		$\pi^{L_{\Omega'}(\mathcal{O}_{F,v})}=0$. Now again by Frobenius reciprocity, we have
		\begin{align*}
			&
			\mathrm{Hom}_{\mathrm{G}(F_v)}
			(
			\pi,
			\mathrm{n}\text{-}\mathrm{Ind}_{\mathbf{B}_v(F_v)}^{ G(F_v)}
			(\chi_1\otimes\cdots\otimes\chi_{n_{\mathrm{s}}})
			)
			\\
			=
			&
			\mathrm{Hom}_{P_{\Omega'}(F_v)}
			(
			\pi,
			\chi_1\otimes\cdots\otimes\chi_{n_{\mathrm{s}}-2}
			\otimes
			\mathrm{n}\text{-}\mathrm{Ind}_{B_2(F_v)}^{\mathrm{GL}_2(F_v)}
			(\chi_{n_{\mathrm{s}}-1}\otimes\chi_{n_{\mathrm{s}}})
			)
			\\
			=
			&
			\mathrm{Hom}_{L_{\Omega'}(F_v)}
			(\pi_{U_{\Omega'}(F_v)},
			\chi_1\times\cdots\times\chi_{n_{\mathrm{s}}-2}
			\times
			\mathrm{n}\text{-}\mathrm{Ind}_{B_2(F_v)}^{\mathrm{GL}_2(F_v)}
			(\chi_{n_{\mathrm{s}}-1}\otimes\chi_{n_{\mathrm{s}}})
			)
		\end{align*}	
		The last term is $0$
		if
		$\mathrm{n}\text{-}\mathrm{Ind}_{B_2(F_v)}^{\mathrm{GL}_2(F_v)}
		(\chi_{n_{\mathrm{s}}-1}\otimes\chi_{n_{\mathrm{s}}})$
		is irreducible
		(recall it is unramified).
		However the first term is by assumption
		non-zero.
		Thus
		the induction
		representation
		$\mathrm{n}\text{-}\mathrm{Ind}_{B_2(F_v)}^{\mathrm{GL}_2(F_v)}
		(\chi_{n_{\mathrm{s}}-1}\otimes\chi_{n_{\mathrm{s}}})$
		must be reducible and
		any $P_{\Omega'}(F_v)$-morphism
		\[
		\pi
		\rightarrow
		\chi_1\otimes\cdots\otimes\chi_{n_{\mathrm{s}}-2}
		\otimes
		\mathrm{n}\text{-}\mathrm{Ind}_{B_2(F_v)}^{\mathrm{GL}_2(F_v)}
		(\chi_{n_{\mathrm{s}}-1}\otimes\chi_{n_{\mathrm{s}}})
		\]
		must factor through
		\[
		\chi_1\otimes\cdots\otimes\chi_{n_{\mathrm{s}}-2}
		\otimes
		\mathrm{St}_2(\chi_{n_{\mathrm{s}}-1})
		\hookrightarrow
		\chi_1\otimes\cdots\otimes\chi_{n_{\mathrm{s}}-2}
		\otimes
		\mathrm{n}\text{-}\mathrm{Ind}_{B_2(F_v)}^{\mathrm{GL}_2(F_v)}
		(\chi_{n_{\mathrm{s}}-1}\otimes\chi_{n_{\mathrm{s}}}),
		\]
		again by the fact
		$\pi^{L_{\Omega'}(\mathcal{O}_{F_v})}=0$.		
	\end{proof}

	This implies in particular that
	we can choose the partition corresponding to
	$\Theta$
	to be given by
	$I(0,n_\mathrm{s})
	=
	I_1^\Theta\sqcup\cdots\sqcup I_{q_\Theta}^\Theta$
	with
	\[
	\# (I_1^\Theta)=
	\# (I_2^\Theta)
	=
	\cdots
	=
	\# (I_s^\Theta)
	=
	1,
	\quad
	\# (I_{s+1}^\Theta)
	=
	\cdots
	=
	\# (I_{q_\Theta}^\Theta)
	=
	2,
	\]
	for some
	$s<n_{\mathrm{s}}$
	in case $\pi$ ramified.

	For a smooth representation $\pi$ of
	$ G(F_v)$
	and a parabolic subgroup
	$P$ of $ G(F_v)$
	with Levi decomposition
	$P=LU$,
	we write
	$\pi_U$
	for the Jacquet module of $\pi$
	with respect to $U$.
	Moreover,
	we write
	$\pi^\mathrm{ss}$
	for the semisimplification of $\pi$.
	We have

	\begin{corollary}\label{semi-simplification of pi}
		Let
		$\pi
		\hookrightarrow
		\mathrm{n}
		\text{-}
		\mathrm{Ind}_{P_\Theta(F_v)}^{ G(F_v)}
		(
		\chi_1
		\otimes
		\cdots
		\chi_{n_{\mathrm{s}}-2}
		\otimes
		\mathrm{St}_2(\chi_{n_{\mathrm{s}}-1})
		)$
		be the embedding as in Lemma \ref{Iwahori spherical representations}.
		Then we have
		\[
		\pi_{U_\Omega(F_v)}^\mathrm{ss}
		\hookrightarrow
		\sigma
		\oplus
		\bigoplus_{w\in\mathcal{S}}
		\mathrm{n}
		\text{-}
		\mathrm{Ind}_{\mathbf{B}(F_v)\cap L_\Omega(F_v)}^{L_\Omega(F_v)}
		\left(
		w^{-1}
		(
		\chi_1
		\otimes
		\cdots
		\otimes
		\mathrm{St}_2(\chi_{n_{\mathrm{s}}-1})
		)_{L_\Theta(F_v)\cap wU_\Omega(F_v)w^{-1}}
		\right)
		\]
		where
		$\sigma$ is a representation of
		$L_\Omega(F_v)$
		such that
		$\sigma^{L_\Omega(\mathcal{O}_{F,v})}=0$,
		and
		$\mathcal{S}$
		is the set of
		$w\in\,^\Omega W^\Theta$
		such that in the corresponding matrix
		$m\in\,^\Omega W^\Theta$,
		we have
		$0\le a_{j,i}\le1$
		for all possible $i,j$
		(in the notation of (\ref{m})).
		Moreover
		at least one of the induced representations is contained in
		$\pi_{U_{\Omega}(F_v)}^{\mathrm{ss}}$.
	\end{corollary}
	\begin{proof}
		This follows from the fact that the Jacquet module of
		$\mathrm{St}_2(\psi)$
		is isomorphic to
		$\psi\otimes\psi|\cdot|$
		and the assumption that
		$\pi^{\mathfrak{p}_{\Omega}}
		\ne0$.
	\end{proof}

	We now recall some results from
	\cite{Vigneras1998}.
	Write
	\begin{align*}
		H_{\mathfrak{p}_\Omega}
		&
		=
		\mathcal{H}_\mathbb{Z}( G(F_v),\mathfrak{p}_\Omega)
		=
		\mathbb{Z}
		\langle
		\mathfrak{p}_\Omega\backslash
		G(F_v)/\mathfrak{p}_\Omega
		\rangle,
		\\
		H_{L_\Omega(\mathcal{O}_{F,v})}
		&
		=
		\mathcal{H}_\mathbb{Z}
		(L_\Omega(F_v),L_\Omega(\mathcal{O}_{F,v}))
		=
		\mathbb{Z}
		\langle
		L_\Omega(\mathcal{O}_{F,v})\backslash
		L_\Omega(F_v)/L_\Omega(\mathcal{O}_{F,v})
		\rangle
	\end{align*}
	for the Hecke algebras.
	Write
	$L_\Omega^+=L_\Omega\cap\mathbf{B}_v$
	for the standard Borel subgroup of
	$L_\Omega$ and
	$L_\Omega^-$
	for the opposite of
	$L_\Omega^+$.
	An element
	$m\in L_\Omega(\mathcal{O}_{F,v})$
	is called
	\emph{positive} if
	$mL_\Omega^+(\mathcal{O}_{F,v})m^{-1}\subset
	L_\Omega^+(\mathcal{O}_{F,v})$
	and
	$m^{-1}L_\Omega^-(\mathcal{O}_{F,v})m
	\subset
	L_\Omega^-(\mathcal{O}_{F,v})$.
	Then we put
	\[
	H_{L_\Omega(\mathcal{O}_{F,v})}^-
	=
	\mathbb{Z}
	\langle
	[
	L_\Omega(\mathcal{O}_{F,v})
	m
	L_\Omega(\mathcal{O}_{F,v})
	]
	|m^{-1}
	\text{ positive}
	\rangle
	\subset
	H_{L_\Omega(\mathcal{O}_{F,v})}.
	\]
	We define the following map
	\begin{align*}
		t_H
		\colon
		&
		H_{L_\Omega(\mathcal{O}_{F,v})}^-
		\otimes_\mathbb{Z}\mathbb{C}
		\rightarrow
		H_{\mathfrak{p}_\Omega}
		\otimes_\mathbb{Z}\mathbb{C},
		\\
		&
		[L_\Omega(\mathcal{O}_{F,v})g
		L_\Omega(\mathcal{O}_{F,v})]
		\mapsto
		\delta_{P_\Omega^-}
		(g)^{1/2}
		\cdot
		[\mathfrak{p}_\Omega
		g
		\mathfrak{p}_\Omega],
	\end{align*}
	where
	$P_\Omega^-$
	is the opposite of
	$P_\Omega$
	and
	$\delta_{P_\Omega^-}$
	is the modular character on
	$P_\Omega^-$.
	Write
	\[
	T_\Omega
	:=
	\text{connected component of }
	1
	\text{ in }
	\bigcap_{\alpha\in\Omega}
	\mathrm{Ker}(\alpha)
	\]
	for the torus in $P_\Omega$
	and
	\[
	T_\Omega^-
	:=
	\left\{
	a\in T_\Omega
	|
	|\alpha(a)|_p\le1,
	\forall
	\alpha\in\Delta\backslash\Omega
	\right\}.
	\]
	Fix the Iwahori decomposition
	\[
	\mathfrak{p}_\Omega
	=
	\mathfrak{u}_\Omega^-\mathfrak{l}_\Omega\mathfrak{u}_\Omega
	\]
	with
	$\mathfrak{u}_\Omega
	=
	\mathfrak{p}_\Omega
	\cap
	U_\Omega(F_v)$,
	$\mathfrak{l}_\Omega
	=
	\mathfrak{p}_\Omega\cap
	L_\Omega(F_v)
	=
	L_\Omega(\mathcal{O}_{F,v})$
	and
	$\mathfrak{u}_\Omega^-
	=
	\mathfrak{p}_\Omega\cap
	U_\Omega^-(F_v)$.
	Fix also an admissible representation
	$(\pi,V)$ of $G(F_v)$.
	For any
	$a\in A_\Omega^-$
	such that
	$a\mathfrak{u}_\Omega a^{-1}
	\subset
	\mathfrak{u}_\Omega$,
	we put
	\[
	V^{\mathfrak{p}_\Omega}_a
	:=
	\left\{
	\frac{\int_{\mathfrak{p}_\Omega}\pi(k)v\ d\mu(k)}{\mu(\mathfrak{p}_\Omega)}
	\Big|
	v\in V^{\mathfrak{p}_\Omega}
	\right\}.
	\]
	Here $d\mu$ is the Haar measure on $\mathfrak{p}_{\Omega}$.
	Then by
	\cite[Proposition 4.1.4]{Casselman1995},
	for any
	$a\in T_\Omega^-$
	with
	$a\mathfrak{u}_\Omega a^{-1}\subset
	\mathfrak{u}_\Omega$,
	the natural projection map
	$V\rightarrow
	V_{U_\Omega(F_v)}$
	induces an isomorphism	  
	\[
	\mathrm{pr}
	\colon
	V^{\mathfrak{p}_\Omega}_a
	\xrightarrow{\sim}
	V_{U_\Omega(F_v)}^{\mathfrak{l}_\Omega}.
	\]
	By
	\cite[Proposition 4.1.6]{Casselman1995},
	these
	$V^{\mathfrak{p}_\Omega}_a$
	are independent of $a$
	and we denote this subspace by
	$V^{\mathfrak{p}_\Omega}_{T_\Omega^-}$.
	In summary, we have
	\begin{lemma}\label{projection of parahoric invariant gives iso}
		Let
		$(\pi,V)$ be an admissible irreducible representation of
		$ G(F_v)$
		and
		\[
		\mathrm{pr}
		\colon
		V^{\mathfrak{p}_\Omega}_{T_\Omega^-}
		\rightarrow
		V_{U_\Omega(F_v)}^{\mathfrak{l}_\Omega}
		\]
		be the natural projection.
		Then
		$\mathrm{pr}$
		is an isomorphism
		and moreover
		\[
		\mathrm{pr}
		(t_H(h)\cdot v)
		=
		h\cdot\mathrm{pr}(v),
		\quad
		\forall
		v\in V^{\mathfrak{p}_\Omega}_{T_\Omega^-},
		\,
		h\in H_{L_\Omega(\mathcal{O}_{F,v})}^-.
		\]
	\end{lemma}
	\begin{proof}
		This is
		\cite[Lemma II.9]{Vigneras1998}.	
	\end{proof}

	Now we define some elements in
	the Hecke algebras
	$H_{L_\Omega(\mathcal{O}_{F,v})}$
	and
	$H_{\mathfrak{p}_\Omega}$.
	Recall the partition
	$I(0,n_\mathrm{s})
	=
	I_1^\Omega\sqcup\cdots
	\sqcup I_{q_\Omega}^\Omega$
	corresponding to the subset
	$\Omega\subset\Delta$.
	Set
	$i_j^\Omega=\# (I_j^\Omega)$.
	For $j=1,\cdots,q_\Omega$
	and
	$k=1,\cdots,i_j^\Omega$,
	we write the matrix
	$\alpha^j_k
	:=\mathrm{diag}
	(1_{i_1^\Omega},
	\cdots,
	1_{i_{j-1}^\Omega},
	\varpi_v1_k,
	1_{i_j^\Omega-k},
	1_{i_{j+1}^\Omega},
	\cdots,
	1_{i_{q_\Omega}^\Omega})$
	and set
	\[
	\widetilde{\alpha}_k^j
	=
	\mathrm{diag}
	(\alpha_k^j,1,
	A_{n_\mathrm{s}}
	(\alpha_k^j)^{-1}
	A_{n_\mathrm{s}})
	\in
	T_\Omega^-.
	\]
	Here
	$A_r$
	is the anti-diagonal matrix of size $r\times r$
	with $1$ on the anti-diagonal.
	Now we put
	\begin{align*}
		V_k^j
		&
		=
		[\mathfrak{p}_\Omega\widetilde{\alpha}_k^j\mathfrak{p}_\Omega]
		\in
		H_{\mathfrak{p}_\Omega},
		\\
		T_k^j
		&
		=
		[L_\Omega(\mathcal{O}_{F,v})\widetilde{\alpha}_k^j
		L_\Omega(\mathcal{O}_{F,v})]
		\in
		H^-_{L_\Omega(\mathcal{O}_{F,v})}.
	\end{align*}
	Thus $V_k^j=\delta_{P_\Omega^-}(\widetilde{\alpha}_k^j)^{1/2}t(T_k^j)$.
	We next look at the action of $V_k^j$ on $\pi^{\mathfrak{p}_\Omega}$.
	For a character $\phi=\phi_1\otimes\cdots\otimes\phi_{n_\mathrm{s}}
	\otimes1\otimes\phi_{n_\mathrm{s}}^{-1}\otimes\cdots\otimes\phi_1^{-1}$
	of $\mathbf{T}(F_v)$ (each $\phi_i$ is a character of $F_v^\times$)
	and an element $w\in W$, we write $w\phi$ for the conjugation of $\phi$ by $w$. We enumerate the blocks in the Levi component $L_\Omega$ of $P_\Omega$ as
	\[
	L_{\Omega}=
	L_{\Omega,1}\times L_{\Omega,2}\times\cdots\times
	L_{\Omega,q_\Omega},
	\]
	with each $L_{\Omega,i}$ corresponding to $I_i^\Omega$ in the partition of $I(0,n_\mathrm{s})$. We then write $\phi_1^\Omega=\phi_1\otimes\cdots\otimes\phi_{i_1^\Omega}$
	for the character of $L_{\Omega,1}$, and similarly
	$\phi_j^\Omega
	=
	\phi_{i_1^\Omega+i_2^\Omega+\cdots+i_{j-1}^\Omega+1}
	\otimes
	\cdots
	\otimes
	\phi_{i_1^\Omega+\cdots+i_j^\Omega}$
	for the character of $L_{\Omega,j}$.
	\begin{proposition}
		Let $\pi$ and $w\in\mathcal{S}$ be as in Corollary \ref{semi-simplification of pi}, then $T_k^j$ acts on the component
		\[
		\mathrm{n}
		\text{-}
		\mathrm{Ind}_{\mathbf{B}(F_v)\cap L_\Omega(F_v)}^{L_\Omega(F_v)}
		\left(
		w^{-1}
		(\chi_1\otimes\cdots\otimes\mathrm{St}_2(\chi_{n_{\mathrm{s}}-1}))_{
			L_\Theta(F_v)\cap wU_\Omega(F_v)w^{-1}}
		\right)
		\]
		inside the semi-simplification $(
		\pi^{\mathfrak{p}_\Omega}_{T_\Omega^-}
		)^\mathrm{ss}
		\simeq
		(\pi_{U_\Omega(F_v)}^{\mathfrak{l}_\Omega})^\mathrm{ss}$
		by the scalar
		\[
		\sum_{J\subset I_j^\Omega,\# (J)=k}
		\prod_{l\in J}
		(w\phi)_l(\varpi_v)
		\]
		where
		$\phi=\chi_1\otimes\cdots\otimes\chi_{\chi_{n_{\mathrm{s}}-2}}
		\otimes\chi_{n_{\mathrm{s}}-1}\otimes(\chi_{n_{\mathrm{s}}-1}|\cdot|)$.
	\end{proposition}
	\begin{proof}
		This is a direct computation.
	\end{proof}
	
	Let $\kappa'$ be an algebraically closed field.
	As in
	(\ref{degree 1 and 2 polynomials}),
	for each $j$ and $k$ as above,
	to each degree
	$2n_\mathrm{s}$ polynomial
	$P(X)
	\in\kappa'[X]$
	of the form
	$P(X)
	=\prod_{i=1}^{n_\mathrm{s}}
	(X-x_i)(X-x_{-i})$
	with $x_{-i}=x_i^{-1}\in(\kappa')^\times$,
	we can associate a monic polynomial
	$\widehat{P}_k^j(X)
	\in
	\kappa'[X]$
	whose roots
	(with multiplicities)
	are
	$\sum_{J\subset \widetilde{I}^\Omega_j,\# (J)=k}
	\prod_{l\in J}x_{w^{-1}(l)}$
	where
	$\widetilde{I}_j^\Omega$
	runs through subsets of
	$I(0,n_\mathrm{s})$
	of cardinal $i_j^\Omega=\# (I_j^\Omega)$.

	In the following, we assume that the cardinality $q_v$ of
	the residual field $\kappa_v$
	satisfies
	\[
	q_v\equiv1(\mathrm{mod}\,p).
	\]
	Using the preceding proposition, one can show
	\begin{proposition}\label{projector induces an isomorphism}
		Let $R$ be a local complete $\mathcal{O}$-algebra with residue field equal to the residue field $\kappa$ of $\mathcal{O}$.	For a smooth $R[ G(F_v)]$-module $\Pi$ such that	$\Pi^U$ is finite free over $\mathcal{O}$
		for any compact open subgroup $U$ of $ G(F_v)$, we write $R_U$ for the image of $R$ in $\mathrm{End}_\mathcal{O}(\Pi^U)$. We assume that
		$\Pi\otimes_\mathcal{O}\overline{\mathbb{Q}}_p$ is a semisimple
		$\overline{\mathbb{Q}}_p[ G(F_v)]$-module,	$R_U\otimes_\mathcal{O}\overline{\mathbb{Q}}_p$ is a semisimple algebra.
		Suppose furthermore that there is a Galois representation
		\[
		\rho
		\colon
		\Gamma_{F_v}
		\rightarrow
		\,^CG\left(R_{\mathfrak{p}_\Omega}\right)
		\]
		such that for any homomorphism $\phi\colon R_{\mathfrak{p}_\Omega}
		\rightarrow\overline{\mathbb{Q}}_p$	and any	irreducible component
		$\pi$ of $\Pi\otimes\overline{\mathbb{Q}}_p$ generated by
		$\Pi^{\mathfrak{p}_\Omega}\otimes_{R_{\mathfrak{p}_\Omega},\phi}
		\overline{\mathbb{Q}}_p$ which can be realized as the local component of
		an irreducible automorphic representation $\pi'$ of $G(\mathbb{A}_F)$
		satisfying the conditions of Theorem \ref{Galois representations attached to auto rep}, then one has an isomorphism
		\[
		r_{p,\iota,\pi'}|_{\Gamma_{F_v}}^\mathrm{ss}
		\simeq
		(
		\rho
		\otimes_{R_{\mathfrak{p}_\Omega},\phi}
		\overline{\mathbb{Q}}_p
		)^\mathrm{ss},
		\]
		where $r_{p,\iota,\pi'}$ is the $p$-adic Galois representation associated to
		$\pi'$ as in Theorem \ref{Galois representations attached to auto rep}.

		We fix a Frobenius lift $\sigma_v$ in $\Gamma_{F_v}$ corresponding to the uniformiser $\varpi_v$ in $F_v$ under the local Artin map. Write
		$P(X)$, resp. $P(X)(X-1)$ for the characteristic polynomial for	$\rho(\sigma_v)$ if $G^\ast=\mathrm{SO}_n^\eta$, resp. $G^\ast=\mathrm{Sp}_n$. Suppose that	$\overline{\rho}:=\rho(\mathrm{mod}\,\mathfrak{m}_R)$ is unramified
		and $\overline{\rho}(\sigma_v)$ satisfies one of the following two conditions:
		\begin{enumerate}
			\item 
			$\overline{\rho}(\sigma_v)$ has eigenvalues $\overline{\alpha}$ and $\overline{\alpha}^{-1}$ such that $\overline{\alpha}\ne\overline{\alpha}^{-1}$ and both of them have the same multiplicity $i_{j_0}^\Omega$ for some	$j_0=1,\cdots,q_\Omega$;

			\item 
			$\overline{\rho}(\sigma_v)$ has eigenvalue $\overline{\alpha}$ such that
			$\overline{\alpha}=\overline{\alpha}^{-1}$ and it has multiplicity	$2i_{j_0}^\Omega$ for some $j_0=1,\cdots,q_\Omega$. In case $G^\ast=\mathrm{Sp}_n$ and $\overline{\alpha}=1$, the multiplicity is $2i^\Omega_{j_0}+1$ for some $j_0=1,\cdots,q_\Omega$.
		\end{enumerate}		
		Then for each pair $(j,k)$ as above, we have a factorization into coprime factors $\widehat{P}^j_k(X)=R_k^j(X)Q_k^j(X)$ such that
		\[
		R_k^j(X)
		\equiv
		\left(
		\left(
		X-\binom{i_j^\Omega}{k}\overline{\alpha}^k
		\right)
		\left(
		X-\binom{i_j^\Omega}{k}\overline{\alpha}^{-k}
		\right)
		\right)^{r(j,k)}
		(\mathrm{mod}\,\mathfrak{m}_R),
		\quad
		\]
		\[
		Q_k^j
		\left(
		\binom{i_j^\Omega}{k}
		\overline{\alpha}^k
		\right),
		\,
		Q_k^j
		\left(
		\binom{i_j^\Omega}{k}
		\overline{\alpha}^{-k}
		\right)
		\not\equiv
		0
		(\mathrm{mod}\,\mathfrak{m}_R)
		\]
		for some non-negative integer $r(j,k)$
		(which is $0$ if $i^\Omega_j<i^\Omega_{j_0}$).

		We write
		\[
		\mathrm{pr}^j
		:=
		\prod_{k=1}^{i_j^\Omega}
		Q_k^j(V_k^j)
		\in
		\mathrm{End}_R(\Pi^{\mathfrak{p}_\Omega}),\index{p@$\mathrm{pr}^j$}
		\quad
		\mathrm{pr}_{(j_1)}:=
		\prod_{j=1,j\ne j_1}^{q_\Omega}
		\mathrm{pr}^j
		\,
		(j_1\ne j_0).\index{p@$\mathrm{pr}_{(j_1)}$}
		\]
		Then
		the operator
		$\mathrm{pr}_{(j_1)}$
		induces an isomorphism
		of $R$-modules
		\begin{equation}\label{G(O_v)-invariant isomorphic to pr_{(j_1)}}
			\Pi^{ G(\mathcal{O}_{F,v})}
			\xrightarrow{\sim}
			\mathrm{pr}_{(j_1)}(\Pi^{\mathfrak{p}_\Omega})
			(\subset
			\Pi^{\mathfrak{p}_\Omega}).
		\end{equation}
	\end{proposition}
	\begin{proof}
		Note that we have assumed $q_v\equiv1(\mathrm{mod}\,p)$, so $\overline{\rho}$ has image in the dual group $\widehat{G}(R_{\mathfrak{p}_\Omega}/\mathfrak{m}_{R_{\mathfrak{p}_\Omega}})$.
		We treat the case where $G^\ast=\mathrm{SO}_n^\eta$ and
		$\overline{\rho}(\sigma_v)$
		has eigenvalues
		$\overline{\alpha}\ne\overline{\alpha}^{-1}$
		(each of multiplicity
		$i_j^\Omega$),
		the other cases being similar.

		For ease of notation, for an element $x \in R$, we write $\overline{x}$
		for $x(\mathrm{mod}\,\mathfrak{m}_R)$. We claim that for any irreducible component $\pi$ of $\Pi\otimes\overline{\mathbb{Q}}_p$ as in the proposition such that $\mathrm{pr}^j\pi^{\mathfrak{p}_\Omega}\ne0$,
		$\pi$ is an unramified representation of $ G(F_v)$.	We argue by contradiction: assume that $\pi$ is ramified. Then by Lemma
		\ref{Iwahori spherical representations}, we have an embedding
		$\pi\hookrightarrow\mathrm{n}\text{-}\mathrm{Ind}_{P_\Theta(F_v)}^{ G(F_v)}(\chi_1\otimes\cdots\otimes\chi_{\chi_{n_{\mathrm{s}}-2}}
		\otimes\mathrm{St}_2(\chi_{n_{\mathrm{s}}-1}))$ with $\chi_i$ unramified characters. The eigenvalues $\overline{\alpha}$ of the Frobenius
		$(\rho\otimes_{R_{\mathfrak{p}_\Omega},\phi}\overline{\mathbb{Q}}_p)(\sigma_v)$ are then the following elements in $R_{\mathfrak{p}_\Omega}$:
		\[
		\overline{\alpha}
		\in
		\left\{
		\chi_1(\omega_v)^{\pm1},\cdots,
		\chi_{n_{\mathrm{s}}-2}(\omega_v)^{\pm1},
		\chi_{n_{\mathrm{s}}-1}(\omega_v)^{\pm1},
		(\chi_{n_{\mathrm{s}}-1}(\omega_v)|\omega_v|)^{\pm1},1
		\right\}.
		\]
		\begin{enumerate}
			\item
			First assume that
			$\overline{\alpha}=\chi_{n_{\mathrm{s}}-1}(\omega_v)^{\pm1}$.
			Since
			$q_v\equiv1
			(\mathrm{mod}\,p)$,
			we have
			\[
			\chi_{n_{\mathrm{s}}-1}(\omega_v)
			\equiv
			\chi_{n_{\mathrm{s}}-1}(\omega_v)|\omega_v|
			(\mathrm{mod}\,\mathfrak{m}_R).
			\]
			Therefore
			$\mathrm{pr}^j$
			projects $\Pi^{\mathfrak{p}_\Omega}$
			onto a subspace where
			the Hecke operators
			$V_k^j$
			($k=1,\cdots,i_j^\Omega$)
			act as the roots of
			$\widehat{P}_k^j(X)$
			corresponding to the subset
			$\widetilde{I}_j^\Omega$
			of
			$I(0,n_\mathrm{s})$
			whose corresponding eigenvalues include
			both
			$\chi_{n_{\mathrm{s}}-1}(\omega_v)$ and
			$\chi_{n_{\mathrm{s}}-1}(\omega_v)|\omega_v|$
			or both
			$\chi_{n_{\mathrm{s}}-1}(\omega_v)^{-1}$ and
			$(\chi_{n_{\mathrm{s}}-1}(\omega_v)|\omega_v|)^{-1}$.
			However we have seen in
			Corollary \ref{semi-simplification of pi}
			only one of
			$\chi_{n_{\mathrm{s}}-1}(\omega_v)$
			and
			$\chi_{n_{\mathrm{s}}-1}(\omega_v)|\omega_v|$
			can occur in such
			$\widetilde{I}_j^\Omega$
			(similarly for
			$\chi_{n_{\mathrm{s}}-1}(\omega_v)^{-1}$ and
			$(\chi_{n_{\mathrm{s}}-1}(\omega_v)|\omega_v|)^{-1}$).
			So this case is not possible.

			\item 
			Next we assume that
			$\overline{\alpha}\neq\chi_{n_{\mathrm{s}}-1}(\omega_v)^{\pm1}$,
			then again we know that
			$\mathrm{pr}^j$
			projects
			$\Pi^{\mathfrak{p}_\Omega}$
			onto the subspace
			where
			$V_k^j$
			act as roots of
			$\widehat{P}_k^j(X)$
			corresponding to the subset
			$\widetilde{I}_j^\Omega$ of $I(0,n_\mathrm{s})$
			whose corresponding eigenvalues do not contain
			any of these
			$\chi_{n_{\mathrm{s}}-1}(\omega_v)$,
			$\chi_{n_{\mathrm{s}}-1}(\omega_v)|\omega_v|$.
			This again contradicts
			Corollary
			\ref{semi-simplification of pi}
			as either
			$\chi_{n_{\mathrm{s}}-1}(\omega_v)$
			or
			$\chi_{n_{\mathrm{s}}-1}(\omega_v)|\omega_v|$
			appears in such a subset $\widetilde{I}_j^\Omega$
			for some $j=2,\cdots,q_\Omega$.
			Thus this case is not possible, either.
		\end{enumerate}

		From this we deduce that
		\[
		\mathrm{rk}_\mathcal{O}
		\Pi^{ G(\mathcal{O}_{F,v})}
		\ge
		\mathrm{rk}_\mathcal{O}
		\mathrm{pr}_{(j_1)}
		(\Pi^{\mathfrak{p}_\Omega}).
		\]
		Thus by Nakayama's lemma,
		it remains to show that
		the induced residual map
		\[
		\overline{\mathrm{pr}}_{(j_1)}
		\colon
		\Pi^{ G(\mathcal{O}_{F,v})}
		\otimes_{\mathcal{O}}\overline{\kappa}
		\rightarrow
		\Pi^{\mathfrak{p}_\Omega}
		\otimes_{\mathcal{O}}\overline{\kappa}
		\]
		is in fact injective.
		Here
		$\overline{\kappa}$
		is an algebraic closure of $\kappa$.

		Again we argue by contradiction:
		we suppose that there is some vector
		$0\ne u\in
		\Pi^{ G(\mathcal{O}_{F,v})}
		\otimes\overline{\kappa}$
		such that
		$\overline{\mathrm{pr}}_{(j_1)}(u)=0$.
		Write
		$N$
		for an irreducible quotient of
		the submodule
		$\overline{\kappa}[ G(F_v)]u$
		of $\Pi^{ G(\mathcal{O}_{F,v})}
		\otimes\overline{\kappa}$
		generated by $u$.
		By Lemma
		\ref{Lemma 1.30},
		we know that
		$\mathrm{dim}_{\overline{\kappa}}
		(N^{ G(\mathcal{O}_{F,v})})=1$
		and
		$\overline{\mathrm{pr}}_{(j_1)}
		(N^{ G(\mathcal{O}_{F,v})})\ne0$.
		This contradiction finishes the proof of 
		the proposition.
		
	\end{proof}

	\begin{lemma}\label{Lemma 1.30}
		Retain the notations from the Proposition \ref{projector induces an isomorphism}. Let $\pi$ be an unramified irreducible admissible module of
		$\overline{\kappa}[ G(F_v)]$. Then one has
		\begin{enumerate}
			\item 
			There is
			a subset
			$\Theta\subset
			\{0,1,\cdots,n_\mathrm{s}-1\}$
			as above
			and
			distinct unramified characters
			$\chi_i$
			such that we have an embedding
			\[
			\pi
			\hookrightarrow
			\mathrm{n}\text{-}\mathrm{Ind}_{P_\Theta(F_v)}^{ G(F_v)}
			\left(
			\chi_1\circ\mathrm{det}\otimes
			\cdots
			\otimes
			\chi_{q_\Theta}\circ\mathrm{det}
			\right),
			\]
			and
			both sides have $ G(\mathcal{O}_{F,v})$-invariants of
			$\overline{\kappa}$-dimension
			equal to $1$.

			\item
			Let $\Omega\subset
			\{0,\cdots,n_\mathrm{s}-1\}$
			be another subset with
			$i_{j_0}^\Omega=i_{l_0}^\Theta$
			for some $j_0=2,\cdots,q_\Omega$
			and $l_0=1,\cdots,q_\Theta$.
			Write
			$P(X)
			=
			\prod_{j=1}^{q_\Theta}
			(X-\chi_j(\omega_v))^{i^\Theta_j}
			(X-\chi_j(\omega_v)^{-1})^{i^\Theta_j}$
			and for each pair
			$(j,k)$,
			we have a factorization
			$\widehat{P}_k^j
			(X)
			=
			R_k^j(X)
			Q_k^j(X)$
			with
			\[
			R_k^j(X)
			=
			\left(
			X-\binom{i_j^\Omega}{k}\chi_{l_0}(\omega_v)^k
			\right)^{r(j,k)}
			\left(
			X-\binom{i_j^\Omega}{k}\chi_{l_0}(\omega_v)^{-k}
			\right)^{r(j,k)}
			\]
			for some non-negative integer
			$r(j,k)\ge0$
			and
			$R_k^j(X)$,
			$Q_k^j(X)$
			coprime.
			Then
			for any
			$j_1\ne j_0$,
			the operator
			$\overline{\mathrm{pr}}_{(j_1)}$
			induces an injection
			\[
			\pi^{ G(\mathcal{O}_{F,v})}
			\hookrightarrow
			\overline{\mathrm{pr}}_{(j_1)}(\pi^{\mathfrak{p}_\Omega}).
			\]
		\end{enumerate}
		
	\end{lemma}
	\begin{proof}
		The first part follows from
		\cite[VI.1]{Vigneras1998}.

		For the second part,
		we follow the strategy as in the proof of
		\cite[Lemma 5.10]{Thorne2012}.
		Recall we have the Iwahori subgroup
		$\mathrm{Iw}=\mathrm{Iw}(v)$
		of $ G(\mathcal{O}_{F,v})$.
		Write $\Omega_0$
		for the subset of
		$I(0,n_\mathrm{s})$
		whose associated parabolic subgroup is
		the standard Borel subgroup
		$\mathbf{B}_v(F_v)$
		of $ G_v(F_v)= G(F_v)$.	
		Write
		\[
		\psi
		=
		\psi_1\otimes\cdots\otimes\psi_{n_\mathrm{s}}
		=
		\underbrace{\chi_1\otimes\cdots\otimes\chi_1}_{i_1^\Theta}
		\otimes
		\cdots
		\otimes
		\underbrace{\chi_j\otimes\cdots\otimes\chi_j}_{i_j^\Theta}
		\otimes
		\cdots
		\otimes
		\chi_{i_{q_\Theta}^\Theta}.
		\]	
		Then we put
		(note that $P_{\Omega_0}=\mathbf{B}_v$)
		\[
		\Pi=
		\mathrm{n}\text{-}\mathrm{Ind}_{P_\Theta(F_v)}^{ G(F_v)}
		\left(
		\chi_1\circ\mathrm{det}\otimes
		\cdots
		\otimes
		\chi_{q_\Theta}\circ\mathrm{det}
		\right),
		\quad
		\widetilde{\Pi}
		=
		\mathrm{n}\text{-}\mathrm{Ind}_{\mathbf{B}_v(F_v)}^{ G(F_v)}
		(\psi).
		\]
		We have an embedding $\pi\hookrightarrow\Pi$. It is easy to see that
		$\Pi^\mathrm{Iw}$ has a basis $\{\varphi_w|w\in\,^{\Omega_0}W^\Theta\}$
		with $\varphi_w$ the characteristic function of $P_\Theta(F_v)w\mathrm{Iw}$. Similarly, $\widetilde{\Pi}^\mathrm{Iw}$
		has a basis $\{\phi_w|w\in\,^{\Omega_0}W^{\Omega_0}\}$ with $\phi_w$ the characteristic function of $\mathbf{B}_v(F_v)w\mathrm{Iw}$.

		Now let's look at the action of certain Iwahori Hecke operators on these functions.
		We write $\delta_l=\mathrm{diag}(\omega_v\cdot1_l,1,\omega_v^{-1}\cdot1_l)
		\in	G(F_v)$ and put	
		\[
		X_l
		=
		[\mathrm{Iw}\delta_l\mathrm{Iw}]
		\cdot
		[\mathrm{Iw}\delta_{l-1}\mathrm{Iw}]^{-1}
		\in
		H_\mathrm{Iw}\otimes_\mathbb{Z}\kappa_v.
		\]
		(\emph{cf.}
		\cite[II.7]{Vigneras1998} and recall that
		$q_v\equiv1(\mathrm{mod}\,p)$).
		By
		\cite[Corollary 3.2]{Lansky2002},
		we know that
		$X_l\phi_w
		=
		\psi_{w(l)}(\omega_v)\phi_w$.
		Since
		$\varphi_w=\sum_{w'\in W_\Theta w}\phi_{w'}$,
		we see
		\[
		X_l\varphi_w
		=
		\sum_{w'}X_l\phi_{w'}
		=
		\sum_{w'}\psi_{w'(l)}(\omega_v)
		\phi_{w'}
		=
		\psi_{w(l)}(\omega_v)
		\varphi_w.
		\]

		On the other hand,
		we have a commutative diagram
		\[
		\begin{tikzcd}
			\Pi^{\mathfrak{p}_\Omega}
			\arrow[r,hook]
			\arrow[d,"q_1"]
			&
			\Pi^{\mathrm{Iw}}
			\arrow[d,"q_2"]
			\\
			(\Pi_{U_\Omega(F_v)})^{\mathfrak{l}_\Omega}
			\arrow[r,"q_3"]
			&
			(\Pi_{\mathbf{U}_v(F_v)})^{\mathfrak{l}_{\Omega_0}}
		\end{tikzcd}
		\]
		where
		$\mathfrak{l}_{\Omega_0}$,
		the Iwahori subgroup,
		is just the maximal compact open subgroup
		$L_{\Omega_0}(\mathcal{O}_{F,v})$
		of
		the torus
		$L_{\Omega_0}(F_v)$.
		Here
		$q_2$ is an isomorphism
		since
		$T_{\Omega_0}^-=T_{\Omega_0}$
		and
		$\Pi_{T_{\Omega_0}}^{\mathfrak{p}_{\Omega_0}}
		=
		\Pi^\mathrm{Iw}$
		(\emph{cf.} Lemma
		\ref{projection of parahoric invariant gives iso}).
		We then write
		$Y_l=
		[\mathfrak{l}_{\Omega_0}
		(\delta_l\cdot\delta_{l-1})
		\mathfrak{I}_{\Omega_0}]$.
		Note that we have the following coset decomposition  	
		\[
		L_\Omega(\mathcal{O}_{F,v})
		\widetilde{\alpha}_k^j
		L_\Omega(\mathcal{O}_{F,v})
		=
		\bigsqcup_{\sigma}L_\Omega(\mathcal{O}_{F,v})
		\sigma(\widetilde{\alpha}^j_k)
		\]
		where
		$\sigma$
		runs through the quotient set
		$W_k^j:=
		S_k\backslash
		S_{i^\Omega_j}\rtimes(\mathbb{Z}/2)^{i^\Omega_j}$
		(we view $S_{i^\Omega_j}\rtimes(\mathbb{Z}/2)^{i^\Omega_j}$
		as a subgroup of
		the Weyl group $W$).	
		Now for any
		$v\in\Pi^{\mathfrak{p}_\Omega}$,
		\begin{align*}
			q_2V_k^j(v)
			&
			=
			q_3q_1V_k^j(v)
			=
			q_3T_k^jq_1(v)
			\\
			&
			=
			\sum_{\sigma\in W_k^j}
			Y_\sigma q_3q_1(v)
			=
			\sum_{\sigma\in W_k^j}
			Y_\sigma q_2(v)
			\\
			&
			=\sum_{\sigma\in W_k^j}
			q_2X_\sigma(v)
		\end{align*}
		where
		$X_\sigma
		=
		\prod_{
			r=i^\Omega_1+\cdots+i^\Omega_{j-1}+1}^{
			i^\Omega_1+\cdots+i^\Omega_{j-1}+k}
		X_{|\sigma(r)|}^{\mathrm{sgn}(\sigma(r))}$
		and
		$Y_\sigma
		=
		\prod_{
			r=i^\Omega_1+\cdots+i^\Omega_{j-1}+1}^{
			i^\Omega_1+\cdots+i^\Omega_{j-1}+k}
		Y_{|\sigma(r)|}^{\mathrm{sgn}(\sigma(r))}$.
		Since
		$q_2$ is an isomorphism,
		$V^j_k=\sum_{\sigma\in W_k^j}
		X_\sigma$
		acting on the subspace
		$\Pi^\mathrm{Iw}$.
		This shows that
		for any
		$w\in\,^{\Omega_0}W^\Theta$,
		$\varphi_w$ is an eigenvector for
		$V_k^j$
		with eigenvalue
		$\sum_{\sigma\in W_k^j}
		\psi_{w,\sigma}(\omega_v)$
		where
		$\psi_{w,\sigma}(\omega_v)=\prod_{
			r=i^\Omega_1+\cdots+i^\Omega_{j-1}+1}^{
			i^\Omega_1+\cdots+i^\Omega_{j-1}+k}
		\psi_{w(|\sigma(r)|)}(\omega_v)^{\mathrm{sgn}(\sigma(r))}
		$.
		It is easy to see that
		$\varphi=
		\sum_{w\in\,^{\Omega_0}W^\Theta}
		\varphi_w$
		generates the $1$-dimensional subspace
		$\Pi^{ G(\mathcal{O}_{F,v})}$
		and
		moreover
		\[
		\overline{\mathrm{pr}}_{(j_1)}(\varphi)
		=
		\sum_{w\in W'}\varphi_w
		\]
		where
		$W'$ is the subset of
		$^{\Omega_0}W^\Theta$
		consisting of
		$w$ such that
		in the corresponding matrix	
		$m\in\,^{\Omega_0}\widetilde{W}^\Theta$,
		there is some
		$j'\ne j_1$ with
		$a_{j',l_0}=I^\Theta_{l_0}$
		(\emph{cf.}
		(\ref{m})).
		Clearly $W'$ is non-empty
		and thus
		$\overline{\mathrm{pr}}_{(j_1)}(\varphi)\ne0$.
		Since
		$\Pi^{ G(\mathcal{O}_{F,v})}
		\hookrightarrow
		\Pi^{\mathfrak{p}_\Omega}$,
		we see that
		$\overline{\kappa}\varphi
		=\Pi^{ G(\mathcal{O}_{F,v})}
		\rightarrow
		\overline{\mathrm{pr}}_{(j_1)}
		(\Pi^{\mathfrak{p}_\Omega})$
		is injective.	
		Since
		$\pi^{ G(\mathcal{O}_{F,v})}
		=
		\Pi^{ G(\mathcal{O}_{F,v})}$
		is $1$-dimensional,
		we deduce that
		$\pi^{ G(\mathcal{O}_{F,v})}
		\rightarrow
		\overline{\mathrm{pr}}_{(j_1)}
		(\pi^{\mathfrak{p}_\Omega})$
		is also injective.

	\end{proof}

	\begin{remark}\rm 
		The proof of
		Proposition
		\ref{projector induces an isomorphism}
		shows that
		the injection in
		Lemma
		\ref{Lemma 1.30}
		is in fact an isomorphism.
	\end{remark}

	Write
	$\kappa_v^\times(p)$
	for the maximal
	$p$-power quotient of $\kappa_v^\times$.
	For each
	$j=1,\cdots,q_\Omega$,
	we consider the following subgroup
	$\mathfrak{p}_{\Omega,j}$
	of
	$\mathfrak{p}_\Omega$
	\[
	\mathfrak{p}_{\Omega,j}
	:=
	\mathrm{Ker}
	\left(
	\mathfrak{p}_\Omega
	\rightarrow
	P_\Omega(\kappa_v)
	\rightarrow
	\mathrm{GL}_{i_j^\Omega}(\kappa_v)
	\xrightarrow{\mathrm{det}}
	\kappa_v^\times
	\rightarrow
	\kappa_v^\times(p)
	\right).
	\]
	Therefore
	$\mathfrak{p}_\Omega/\mathfrak{p}_{\Omega,j}
	\simeq
	\kappa_v^\times(p)$.
	Similar to
	$V_k^j$,
	we put
	$\widetilde{V}_k^j
	=
	[\mathfrak{p}_{\Omega,j}
	\widetilde{\alpha}_k^j
	\mathfrak{p}_{\Omega,j}]$.
	Then
	the inclusion
	$\pi^{\mathfrak{p}_\Omega}
	\hookrightarrow
	\pi^{\mathfrak{p}_{\Omega,j}}$
	takes
	$V_k^j$ to
	$\widetilde{V}_k^j$.
	We have:
	
	\begin{proposition}\label{parahoric invariant subspace is one-dimensional}
		Retain the notations from
		Proposition
		\ref{projector induces an isomorphism}.
		Fix some
		$j_0=1,\cdots,q_\Omega$.
		Let
		$\pi$ be an irreducible admissible representation of
		$ G(F_v)$.
		Suppose that
		$\pi^{\mathfrak{p}_{\Omega,j_0}}\ne0$
		and
		$\pi^{\mathfrak{p}_\Omega}=0$, then
		\[
		\mathrm{dim}(\pi^{\mathfrak{p}_{\Omega,j_0}})=1.
		\]
		Moreover,
		$r_{p,\iota,\pi'}^\mathrm{ss}$
		is a direct sum of $1$-dimensional representations
		and
		$\pi$ is a subquotient of a principal series
		$\mathrm{n}\text{-}\mathrm{Ind}_{\mathbf{B}_v(F_v)}^{ G(F_v)}
		(\chi_1\otimes\cdots\otimes\chi_{n_\mathrm{s}})$
		where
		for
		$i=i^\Omega_1+\cdots+i^\Omega_{j_0-1}+1,
		\cdots,
		i^\Omega_1+\cdots+i^\Omega_{j_0}$,
		the characters $\chi_i$
		are tamely ramified 
		and are equal to each other when restricted to
		$\mathcal{O}_{F_v}^\times$,
		while for $i$ not in the above range,
		$\chi_i$ is unramified.
	\end{proposition}
	\begin{proof}
		Write
		$U_{n,1}\subset U_{n,0}
		\subset
		\mathrm{GL}_n(\mathcal{O}_{F,v})$
		where
		$U_{n,1}$
		consists of
		$g\equiv
		\begin{pmatrix}
			A & B \\
			0 & 1
		\end{pmatrix}
		(\mathrm{mod}\,\omega_v)$,	
		$U_{n,0}$
		consists of
		$g\equiv
		\begin{pmatrix}
			A & B \\
			0 & D
		\end{pmatrix}
		(\mathrm{mod}\,\omega_v)$
		with
		$A\in\mathrm{GL}_{n-1}(\kappa_v)$.
		Since
		$\pi^{\mathfrak{p}_{\Omega,j_0}}\ne0$,
		$\pi$
		can be embedded into an induced representation
		$\mathrm{n}\text{-}\mathrm{Ind}_{P_\Omega(F_v)}^{ G(F_v)}
		(
		\chi_1\otimes\cdots\otimes\chi_s\otimes
		\mathrm{St}_2(\psi_1)
		\otimes
		\cdots
		\otimes
		\mathrm{St}_2(\psi_t)
		)$
		such that
		$\mathrm{St}_2(\psi_j)^{U_{2,1}}\ne0$
		for all
		$j=1,\cdots,t$
		($t$ may be $0$).

		If $t\neq0$, then by
		\cite[p.90]{ClozelHarrisTaylor2008},
		$\psi_j$ is unramified and
		thus
		$\mathrm{St}_2(\psi_j)^{U_{2,0}}=
		\mathrm{St}_2(\psi_j)^{U_{2,1}}\ne0$.
		We deduce that
		$\pi^{\mathfrak{p}_\Omega}\ne0$,
		contradicting the assumption on
		$\pi$.
		This shows that
		$t=0$.
		Besides,
		$\mathfrak{p}_{\Omega,j}$
		contains the subgroup
		$\mathrm{Iw}_j\subset\mathrm{Iw}$
		consisting of
		$g=(g_{i,k})$
		such that
		$\prod_{
			r=i^\Omega_1+\cdots+i^\Omega_{j-1}+1
		}^{i^\Omega_1+\cdots+i^\Omega_{j}}g_{r,r}
		\equiv
		1(\mathrm{mod}\,\omega_v)$.

		By applying
		\cite[Theorem 7.7]{Roche1998}
		(noting that $F_v$ has odd residual characteristic,
		taking into account of \cite[Remark 4.14]{Roche1998}),
		we see that
		$\pi$
		is indeed a subquotient of a principal series of the given form.

		The rest of the proposition
		is clear.
	\end{proof}

	We need some auxiliary Hecke operators
	in the following:
	for each
	$j=1,\cdots,q_\Omega$
	and
	$\alpha\in
	\mathcal{O}_{F,v}^\times$,
	we write
	\[
	\widetilde{\alpha}^j=(a_{i,k})\in
	\mathbf{T}_v(\mathcal{O}_{F,v})
	\]
	for the diagonal matrix
	with $a_{i,i}=\alpha$
	if $i=i^\Omega_1+\cdots+i^\Omega_{j-1}+1$
	and
	$=1$ otherwise.
	Then we put
	\[
	V^j_\alpha=
	[\mathfrak{p}_{\Omega,j}\widetilde{\alpha}^j
	\mathfrak{p}_{\Omega,j}].
	\]
	We have the following (see also Theorem \ref{Theorem 4}):

	\begin{theorem}\label{decomposition of the Galois representation}
		Retain the notations from
		Proposition
		\ref{projector induces an isomorphism}.
		Write
		$R_{j_0,j_1}$ for the image of
		$R$ in
		$\mathrm{End}_{\mathcal{O}}(\mathrm{pr}_{(j_1)}
		(\Pi^{\mathfrak{p}_{\Omega,j_0}}))$.
		Suppose that
		$\overline{\rho}(\sigma_v)$
		has eigenvalues $\overline{\alpha}$ of one
		of the two types listed in 
		Proposition
		\ref{projector induces an isomorphism}. Write
		$\rho_{j_0,j_1}
		:=
		\rho\otimes_{R_{\mathfrak{p}_{\Omega,j_0}}}
		R_{j_0,j_1}$, equipped with a non-degenerate symmetric/symplectic bilinear form $(-,-)$.
		Then we have a decomposition
		into non-degenerate quadratic, resp. symplectic subspaces:
		\[
		(\rho_{j_0,j_1},(-,-))
		=
		(s,(-,-)|_s)\bigoplus(\psi,(-,-)|_{\psi})
		\]
		where
		$s$ is unramified,
		$\psi$ is tamely ramified and its
		restriction to the inertial subgroup $I_v$ is given by a pair of scalar characters
		$(\psi',(\psi')^{-1})$.
		Moreover,
		for each
		$\alpha\in
		\mathcal{O}_{F,v}^\times$,
		\[
		V^{j_0}_\alpha
		=
		\psi'(\mathrm{Art}_{F_v}(\alpha))
		\in
		R_{j_0,j_1}.
		\]
	\end{theorem}
	\begin{proof}
		We treat the case $G^\ast=\mathrm{SO}_n^\eta$.
		Recall that
		$\rho$
		acts on a symplectic module over $R_{\mathfrak{p}_\Omega}$ free of rank
		$N=2n_{\mathrm{s}}=2\lfloor n/2\rfloor$
		over
		$R_{\mathfrak{p}_{\Omega,j_0}}$.
		Since
		$\pi^{\mathfrak{p}_{\Omega,j_0}}
		\ne0$,
		we know that
		$\rho_{j_0,j_1}$
		has abelian image.
		Write
		$P(X)$
		for the characteristic polynomial of
		$\rho_{j_0,j_1}(\sigma_v)$
		such that there is a factorization
		$P(X)=A(X)B(X)$
		with
		$B(\overline{\alpha})=0=B(\overline{\alpha}^{-1})$
		and
		$A(X),B(X)$
		coprime.
		Set
		\[
		s=
		B(\rho_{j_0,j_1}(\sigma_v))
		R_{j_0,j_1}^N,
		\quad
		\psi
		=
		A(\rho_{j_0,j_1}(\sigma_v))R_{j_0,j_1}^N.
		\]
		So
		$R_{j_0,j_1}^N
		=
		s\oplus\psi$,
		the decomposition
		is
		invariant under
		$\rho_{j_0,j_1}$.
		Suppose that
		$B(X)$
		is of degree
		$d_B$.
		Since the roots of
		$B(X)$
		appear in pairs
		$\beta,\beta^{-1}$,
		we have
		$B(X^{-1})=(-X)^{-d_B}B(X)$.
		Suppose that under the standard basis of
		$R_{j_0,j_1}^N$,
		the quadratic form
		$q$
		(or the symmetric bilinear form
		$(-,-)$)
		corresponds to the matrix
		$\Lambda$.
		Thus for any
		$v,w\in R_{j_0,j_1}^N$,
		one has
		$B(\rho_{j_0,j_1}(\sigma_v))v\in s$
		and
		$A(\rho_{j_0,j_1}(\sigma_v))w\in\psi$
		and therefore
		\begin{align*}
			(B(\sigma)v,A(\sigma)w)
			&
			=
			v^\mathrm{t}
			B(\rho_{j_0,j_1}(\sigma_v))^\mathrm{t}
			\Lambda
			A(\rho_{j_0,j_1}(\sigma_v))
			w
			\\
			&
			=
			v^\mathrm{t}
			\Lambda
			B(\rho_{j_0,j_1}(\sigma_v)^{-1})
			A(\rho_{j_0,j_1}(\sigma_v))
			w
			\\
			&
			=
			v^\mathrm{t}
			\Lambda
			(\rho_{j_0,j_1}(\sigma_v)^\mathrm{t})^{-d_B}
			B(\rho_{j_0,j_1}(\sigma_v))
			A(\rho_{j_0,j_1}(\sigma_v))
			w
			\\
			&
			=
			0.
		\end{align*}
		This shows that
		$(\rho_{j_0,j_1},(-,-))
		=
		(s,(-,-)|_s)\oplus(\psi,(-,-)|_\psi)$.
		Next we verify the properties of
		$s$ and $\psi$.

		If $\pi$ is unramified,
		then
		$s$ and $\psi$ are unramified
		and therefore
		$V^{j_0}_\alpha$
		acts trivially on
		$\mathrm{pr}_{(j_1)}
		(\pi^{\mathfrak{p}_{\Omega,j_0}})$.

		If $\pi$ is ramified,
		then by the isomorphism (\ref{G(O_v)-invariant isomorphic to pr_{(j_1)}}) in Proposition \ref{projector induces an isomorphism} and Proposition \ref{parahoric invariant subspace is one-dimensional}, we have
		$\pi^{\mathfrak{p}_\Omega}=0$.
		Thus by Proposition \ref{parahoric invariant subspace is one-dimensional},
		$\pi$ is a subquotient of a principal series
		$\mathrm{n}
		\text{-}
		\mathrm{Ind}_{\mathbf{B}_v(F_v)}^{ G(F_v)}
		(\chi_1\otimes\cdots\otimes\chi_{n_\mathrm{s}})$
		with
		(in the notations of the Proposition \ref{parahoric invariant subspace is one-dimensional})
		\begin{enumerate}
			\item 
			$\chi_i$
			tamely ramified for
			$i=i^\Omega_1+\cdots+i^\Omega_{j_0-1}+1,\cdots,
			i^\Omega_1+\cdots+i^\Omega_{j_0}$
			and being equal when restricted to the subgroup
			$\mathcal{O}_{F,v}^\times$ (we denote this common character by $\psi'$),

			\item 
			the remaining characters being unramified.
		\end{enumerate}		
		We have
		$\mathrm{pr}_{(j_1)}(\pi^{\mathfrak{p}_{\Omega,j_0}})
		=
		\pi^{\mathfrak{p}_{\Omega,j_0}}$ again by Proposition \ref{parahoric invariant subspace is one-dimensional}.
		Note that
		$V_k^{j_0}$
		acts on
		$\pi^{\mathfrak{p}_{\Omega,j_0}}$
		by the scalar
		$\sum_{\sigma\in W_k^{j_0}}\chi_\sigma(\omega_v)$
		where
		(as in the proof of Lemma \ref{Lemma 1.30})
		\[
		\chi_\sigma
		=
		\prod_{
			r=i^\Omega_1+\cdots+i^\Omega_{j_0-1}+1
		}^{
			i^\Omega_1+\cdots+i^\Omega_{j_0}}
		\chi_{|\sigma(r)|}^{\mathrm{sng}(\sigma(r))},
		\]
		and thus
		$V^{j_0}_k$
		acts
		on
		$\mathrm{pr}_{(j_0)}
		(\pi^{\mathfrak{p}_{\Omega,j_0}})$
		by the same scalar.
		Write
		$\widehat{i}^\Omega_j=i^\Omega_1+\cdots+i^\Omega_j$.
		This gives
		\begin{align*}
			s^\mathrm{ss}
			&
			\simeq
			\left(
			\bigoplus_{r=1}^{\widehat{i}^\Omega_{j_0-1}}(\chi_r\oplus\chi_r^{-1})
			\oplus
			\bigoplus_{r=\widehat{i}^\Omega_{j_0}+1}^{n_\mathrm{s}}
			(\chi_r\oplus\chi_r^{-1})
			\right)
			|\cdot|^{(1-n)/2}
			\circ
			\mathrm{Art}_{F_v}^{-1},
			\\
			\psi^\mathrm{ss}
			&
			\simeq
			\left(
			\bigoplus_{r=\widehat{i}^\Omega_{j_0-1}+1}^{\widehat{i}^\Omega_{j_0}}
			(\chi_r\oplus\chi_r^{-1})
			\right)
			|\cdot|^{(1-n)/2}
			\circ
			\mathrm{Art}_{F_v}^{-1}.
		\end{align*}

		It follows from the definition of
		$V^j_\alpha$
		that
		$V^{j_0}_\alpha$
		acts as
		the scalar
		$\psi'(\mathrm{Art}_{F_v}(\alpha))$.

		For the case $G^\ast=\mathrm{Sp}_n$, the proof is the same as above, except that if $\overline{\alpha}=1$, $\psi^\mathrm{ss}$ has an extra direct summand $|\cdot|^{(1-n)/2}\circ\mathrm{Art}_{F_v}^{-1}$.
	\end{proof}

	\section{Deformation theory}\label{deformation theory}
	In this section, we study the deformations of residual Galois representations.
	
	We write $\widehat{\mathfrak{g}}$ for the Lie algebra of the dual group $\widehat{G}$ of $G$. We view $\widehat{\mathfrak{g}}$ as a Lie algebra over $\Z$, more explicitly, for any ring $A$,
	\[
	\widehat{\mathfrak{g}}(A)
	=
	\begin{cases*}
		\left\{
		M\in\mathrm{M}_{N,N}(A)|M^\mathrm{t}A_N'+A_N'M=0
		\right\},
		&
		$G^\ast=\mathrm{SO}_n$ with $n$ odd;
		\\
		\left\{
		M\in\mathrm{M}_{N,N}(A)|M^\mathrm{t}A_N+A_NM=0
		\right\},
		&
		$G^\ast=\mathrm{Sp}_n$ or $G^\ast=\mathrm{SO}_n^\eta$ with $n$ even.
	\end{cases*}
	\]

	For any local ring $A$ (denote $\mathfrak{m}_A$ for its maximal ideal),	write
	\[
	\widehat{G}^1(A)
	:=
	\mathrm{Ker}
	\big(
	\widehat{G}(A)
	\rightarrow
	\widehat{G}(A/\mathfrak{m}_A)
	\big).
	\]

	\begin{remark}\label{dual group is smooth}\rm
		We see easily that $^CG$ and its center $Z(\,^CG)\simeq\mathbb{G}_m$
		are smooth over $\mathbb{Z}$. For a local ring $A$ and an ideal $I$ of $A$ such that $\mathfrak{m}_AI=0$ (thus $I$ is a vector space over
		$\kappa_A=A/\mathfrak{m}_A$), we have an injective exponential map
		\begin{align*}
			\exp\colon
			\,
			\widehat{\mathfrak{g}}(\kappa_A)\otimes_{\kappa_A}I
			&
			\rightarrow
			\widehat{G}^1(A),
			\\
			Y
			&
			\mapsto
			1+Y.
		\end{align*}
		Here the sum
		$1+Y$ is taken in the embedding
		$\widehat{G}^1(A)
		\hookrightarrow
		\mathrm{M}_{N,N}(A)$.
		If $I=\mathfrak{m}_A$,
		then
		$\exp$ is a bijection.
		
	\end{remark}

	We fix the following notations in
	this section:
	\begin{enumerate}
		\item 
		We fix a finite field extension $\kappa$ of $\mathbb{F}_p$, $\kappa[\epsilon]$, for the dual numbers of $\kappa$
		(\emph{i.e.} $\epsilon^2=0$), $W(\kappa)$ for the ring of Witt vectors over
		$\kappa$ and $L(\kappa)$ the fraction field of $W(\kappa)$. We choose $\kappa$ such that $L(\kappa)$ embeds into $E$ and $E$ is totally ramified over $L(\kappa)$.

		\item 
		Write $\mathrm{Art}_\mathcal{O}$ for the category of Artinian local $\mathcal{O}$-algebras $\mathcal{O}\rightarrow R$ which induces the identity map on the residual fields $\mathcal{O}/\mathfrak{p}
		=R/\mathfrak{m}_R$.	Morphisms in $\mathrm{Art}_\mathcal{O}$
		are continuous morphisms $R\rightarrow R'$ of $\mathcal{O}$-algebras
		inducing the identity map on residual fields $R/\mathfrak{m}_R=R'/\mathfrak{m}_{R'}$. Similarly, write $\mathrm{CNL}_\mathcal{O}$ to be the category of complete noetherian local $\mathcal{O}$-algebras $\mathcal{O}\rightarrow R$	inducing the identity map on residual fields and morphisms between these objects are continuous $\mathcal{O}$-algebra morphisms inducing the identity map on residual fields. Thus each object $R$ in $\mathrm{CNL}_\mathcal{O}$
		is a projective limit of objects $R/\mathfrak{m}_R^n$ in $\mathrm{Art}_\mathcal{O}$ and we can (and will) view $\mathrm{Art}_\mathcal{O}$ as a full subcategory of $\mathrm{CNL}_\mathcal{O}$ in a natural way.

		\item 
		All morphisms of topological groups
		considered in this section
		will be assumed to be continuous.
		For $E=F$ or $F_v$ and an
		$\mathcal{O}[\Gamma_E]$-module
		$M$,
		we write
		$M[n]$
		for the twist of $M$ by the
		$n$-th power of the
		$p$-adic cyclotomic character of
		$\Gamma_E$
		(whenever it is defined).

		\item 
		If $G^\ast=\mathrm{SO}_n$ with $n$ odd, for any $g\in\,^CG$, we write $\mu(g)$ for the similitude factor of $g$. If $G^\ast=\mathrm{Sp}_n$ or $G^\ast=\mathrm{SO}_n^\eta$ with $n$ even, for any $g\in\,^CG$, we write $\mu(g)$ for the image of $g$ under the projection map $^CG=^LG\times\G_m\to\G_m$. Then to any group homomorphism $r\colon\Gamma\to\,^CG(\kappa)$, we associate a character
		\begin{equation}\label{mu_r}
			\mu_r\colon\Gamma\to\G_m(\kappa)=\kappa^\times,
			\quad
			\gamma\mapsto\mu(r(\gamma)).
		\end{equation}
	\end{enumerate}

	\subsection{Deformation problem}

	Fix a profinite group
	$\Gamma$,
	which eventually will be the global Galois group
	$\Gamma_F$ of $F$ or
	the local Galois group
	$\Gamma_{F_v}$
	of $F_v$.
	We consider a continuous group homomorphism:
	\[
	\overline{r}
	\colon
	\Gamma
	\rightarrow
	\,^CG(\kappa)
	\]
	and a continuous character $\nu\colon\Gamma\to\,^CG(\mathcal{O})$ such that
	\[
	\overline{\nu}=\mu_{\overline{r}},
	\]
	where $\overline{\nu}$ is the reduction mod $\mathfrak{p}$ of $\nu$. We will put more conditions on $\overline{r}$ later on.
	
	\begin{definition}
		A \emph{lifting} of
		$\overline{r}$ to an object
		$R$ in $\mathrm{CNL}_\mathcal{O}$
		is a group morphism
		$r\colon
		\Gamma
		\rightarrow\,
		^CG(R)$
		such that
		\[
		r(\mathrm{mod}\, \mathfrak{m}_R)
		=
		\overline{r}
		\text{ and }
		\mu_r=\nu.
		\]
		Two such liftings
		$r_1,r_2
		\colon
		\Gamma
		\rightarrow
		\,
		^CG(R)$
		are equivalent if there is an element
		$g\in
		\widehat{G}^1(R)$
		such that
		$r_1=gr_2g^{-1}$.
		A \emph{deformation}
		of $\overline{r}$ is an equivalence class of
		liftings of
		$\overline{r}$.

		Consider the functor
		\[
		\mathcal{F}
		\colon
		\mathrm{Art}_\mathcal{O}
		\rightarrow
		\mathrm{Sets},
		\]
		which takes an object
		$R$
		to the set of
		liftings,
		resp.
		deformations
		of $\overline{r}$.
		We say that $\mathcal{F}$
		is
		\emph{pro-representable}
		if there are an object
		$R^\mathrm{univ}$
		in
		$\mathrm{CNL}_\mathcal{O}$
		and a lifting,
		resp.
		deformation,
		$r^\mathrm{univ}
		\colon
		\Gamma
		\rightarrow
		\,^CG(R^\mathrm{univ})$
		such that 
		any
		lifting,
		resp.
		deformation,
		$r\colon
		\Gamma\rightarrow
		\,^CG(R)$
		to an object
		$R$ in $\mathrm{Art}_\mathcal{O}$
		corresponds bijectively to
		a unique morphism
		$R^\mathrm{univ}
		\rightarrow
		R$
		in
		$\mathrm{CNL}_\mathcal{O}$
		such that its composition with
		$r^\mathrm{univ}$
		gives rise to $r$.
		In this case,
		we will say that
		$\mathcal{F}$
		is pro-representable by the pair
		\[
		(R^\mathrm{univ},r^\mathrm{univ})
		\]
		or simply
		pro-representable by
		$R^\mathrm{univ}$
		whenever the context is clear.
		$R^\mathrm{univ}$
		is called the universal
		lifting,
		resp.
		deformation ring and
		$r^\mathrm{univ}$
		the universal lifting,
		resp.
		deformation
		of $\overline{r}$.
	\end{definition}

	Now let $\Gamma=\Gamma_F$.
	Since the local Galois group
	$\Gamma_{F_v}$
	is topologically finitely generated,
	for each $v$
	there is a universal lifting
	\[
	(R^\mathrm{univ}_v,r^\mathrm{univ}_v)
	\]
	of
	$\overline{r}|_{\Gamma_{F_v}}$
	with
	$R^\mathrm{univ}_v\in
	\mathrm{CNL}_\mathcal{O}$.

	\begin{definition}\label{local deformation problem}
		For a finite place $v$ of $F$, a \emph{local deformation problem at $v$}
		is a functor
		\[
		\mathcal{D}_v
		\colon
		\mathrm{CNL}_\mathcal{O}
		\rightarrow
		\mathrm{Sets},
		\]
		which takes $R$ to a set of
		liftings
		of $\overline{r}$ to
		$R$ and
		satisfies the following conditions
		\begin{enumerate}
			\item
			$\overline{r}
			\in
			\mathcal{D}_v(\kappa)$;
			
			\item 
			If
			$r\in
			\mathcal{D}_v(R)$
			and
			$\phi\colon
			R\rightarrow
			R'$
			is a morphism in
			$\mathrm{CNL}_\mathcal{O}$,
			then
			$\phi\circ r
			\in
			\mathcal{D}_v(R')$;
			
			\item 
			For any two morphisms
			$\phi_1\colon
			R_1\rightarrow
			R_0$ and
			$\phi_2\colon
			R_2
			\rightarrow
			R_0$ in
			$\mathrm{CNL}_\mathcal{O}$,
			write the fiber product
			$R_3=R_1\times_{\phi_1,R_0,\phi_2}R_2$
			(that is,
			$R_3=\{(x_1,x_2)\in R_1\times R_2
			|\phi_1(x_1)=\phi_2(x_2)\}$).
			If
			$r_i\in
			\mathcal{D}_v(R_i)$
			($i=1,2$)
			and
			$\phi_1\circ r_1=\phi_2\circ r_2$,
			then
			$r_1\oplus r_2
			\in
			\mathcal{D}_v(R_3)$;
			
			\item 
			For a projective system
			$\{r_i\in\mathcal{D}_v(R_i)\}_{i\in I}$,
			one has
			$\lim\limits_{\overleftarrow{i}}
			r_i
			\in
			\mathcal{D}_v
			(\lim\limits_{\overleftarrow{i}}R_i)$;
			
			\item 
			If $r'$ is equivalent to $r$ and
			$r\in\mathcal{D}_v(R)$,
			one has
			$r'\in\mathcal{D}_v(R)$;
			
			\item 
			For an injective morphism
			$\phi
			\colon
			R
			\rightarrow
			R'$
			in
			$\mathrm{CNL}_\mathcal{O}$
			and
			$r$ a lifting of
			$\overline{r}$ to $R$
			such that
			$\phi\circ r\in
			\mathcal{D}_v(R')$,
			then
			$r\in\mathcal{D}_v(R)$.
		\end{enumerate}
	\end{definition}

	\begin{remark}\label{group action on universal ring}\rm
		We have a natural
		$\mathcal{O}$-linear
		action of
		the group
		$\widehat{G}^1(R^\mathrm{univ}_v)$
		on $R^\mathrm{univ}_v$:
		for each
		$g\in
		\widehat{G}^1(R^\mathrm{univ}_v)$,
		we have another lifting
		$gr^\mathrm{univ}_vg^{-1}$
		of $\overline{r}$
		and this induces a morphism
		$R^\mathrm{univ}_v
		\rightarrow
		R^\mathrm{univ}_v$
		(depending on $g$)
		of
		$\mathcal{O}$-algebras.
		By the universality of
		$(R^\mathrm{univ}_v,r^\mathrm{univ}_v)$,
		one verifies that this is indeed a group action.
		Similarly,
		for each object
		$R$ in
		$\mathrm{CNL}_\mathcal{O}$,
		there is a natural
		action of
		$\widehat{G}^1(R)$
		on the set
		$\mathrm{Hom}_{\mathrm{CNL}_\mathcal{O}}
		(R^\mathrm{univ}_v,R)$.
	\end{remark}

	\begin{lemma}\label{isomorphisms of tangent spaces}
		We have natural isomorphisms of
		$\kappa$-vector spaces:
		\[
		\mathrm{Hom}_\kappa
		(\mathfrak{m}_{R^\mathrm{univ}_v}/
		(\mathfrak{m}_{R^\mathrm{univ}_v}^2,\mathfrak{p}),
		\kappa
		)
		\simeq
		\mathrm{Hom}_{\mathrm{CNL}_{\mathcal{O}}}
		(R^\mathrm{univ}_v,
		\kappa[\epsilon])
		\simeq
		Z^1(\Gamma_{F_v},
		\widehat{\mathfrak{g}}(\kappa)).
		\]
	\end{lemma}

	\begin{proof}
		The first isomorphism is valid for any
		object
		$R$ in
		$\mathrm{CNL}_\mathcal{O}$
		instead of $R^\mathrm{univ}_v$.	
		The second follows from
		\cite[Lemma 3.1 \& Proposition 3.2]{Tilouine2002}
	\end{proof}
	\begin{remark}\label{isomorphism of general tangent spaces}\rm 
		One can generalize the above isomorphisms
		as follows:
		for any $R\in\mathrm{CNL}_\mathcal{O}$ and
		any ideal $I$ of $R$ with
		$\mathfrak{m}_RI=0$
		(thus $I$ is a $\kappa$-vector space),
		consider two liftings
		$r_1,r_2$ of
		$\overline{r}$ to $R$
		such that
		$r_1=r_2
		(\mathrm{mod}\,I)$.
		Then one can write
		$r_2(\gamma)
		=
		(1+r(\gamma))
		r_1(\gamma)$
		with
		\[
		r(\gamma)
		=
		r_2(\gamma)r_1(\gamma)^{-1}-1
		\in
		\,
		\widehat{\mathfrak{g}}(\kappa)\otimes_\kappa I
		\]
		for all
		$\gamma\in\Gamma_{F_v}$.
		One verifies immediately that
		$r(\gamma_1\gamma_2)r_1(\gamma_1)
		=
		r(\gamma_1)r_1(\gamma_1)
		+
		r_1(\gamma_1)r(\gamma_2)$.
		Let
		$\Gamma_{F_v}$
		act on
		$\widehat{\mathfrak{g}}(\kappa)\otimes I$
		by the conjugate action
		$r_1(\gamma)(\ast)r_1(\gamma)^{-1}$.
		Then we get an element
		\[
		r\in
		Z^1
		(
		\Gamma_{F_v},
		\widehat{\mathfrak{g}}
		(\kappa)\otimes I
		).
		\]
		Moreover,
		if $r\in B^1(\Gamma_{F_v},
		\widehat{\mathfrak{g}}(\kappa)\otimes I)$,
		that is,
		there exists an element
		$Y\in\widehat{\mathfrak{g}}(\kappa)\otimes I$
		such that
		$r(\gamma)
		=r_1(\gamma)Yr_1(\gamma)^{-1}-Y$,
		then one has
		$r_2(\gamma)=r_1(\gamma)+
		r_1(\gamma)Y-Yr_1(\gamma)$.
		Since
		$\mathfrak{m}_RI=0$,
		by
		Remark \ref{dual group is smooth},
		we have
		$1+Y=\exp(Y)
		\in
		\widehat{G}^1(R)$
		and
		$(1+Y)^{-1}=1-Y$.
		Therefore
		\[
		r_2(\gamma)=(1+Y)r_1(\gamma)(1-Y).
		\]
		The converse is also true:
		if $r_2$ is conjugate to $r_1$ by an element in
		$\exp(\widehat{\mathfrak{g}}(\kappa)\otimes I)$,
		then
		$r\in
		B^1(\Gamma_{F_v},
		\widehat{\mathfrak{g}}(\kappa)\otimes I)$.

		Suppose that $r_1$ corresponds to a morphism
		$\phi_1
		\colon
		R^\mathrm{univ}_v
		\rightarrow
		R$.
		Thus one has bijections
		\begin{align*}
			H^1(\Gamma_{F_v},\,
			\widehat{\mathfrak{g}}(\kappa)\otimes I)
			\leftrightarrow
			&
			\{
			r'\colon
			\Gamma_{F_v}
			\rightarrow
			\,
			^CG(R)\mid 
			r'=r_1(\mathrm{mod}\,I)
			\}/
			\exp(\widehat{\mathfrak{g}}(\kappa)\otimes I)
			\\
			\leftrightarrow
			&
			\{
			\phi'
			\colon
			R^\mathrm{univ}_v
			\rightarrow
			R\mid 
			\phi'=\phi_1(\mathrm{mod}\,I)
			\}/
			\exp(\widehat{\mathfrak{g}}(\kappa)\otimes I).
		\end{align*}
		By the definition of the group cohomology
		$H^\bullet(\Gamma_{F_v},
		\widehat{\mathfrak{g}}(\kappa))$,
		we have the following long exact sequence
		\begin{equation}\label{isomorphism of general tangent spaces-exact sequence}
			0
			\rightarrow
			H^0(\Gamma_{F_v},\widehat{\mathfrak{g}}(\kappa))
			\rightarrow
			\widehat{\mathfrak{g}}(\kappa)
			\rightarrow
			Z^1(\Gamma_{F_v},\widehat{\mathfrak{g}}(\kappa))
			\rightarrow
			H^1(\Gamma_{F_v},\widehat{\mathfrak{g}}(\kappa))
			\rightarrow
			0.
		\end{equation}
	\end{remark}

	Next we give some characterizations of
	local deformation problems.
	\begin{proposition}\label{characterization of local deformation problem}
		Let $I$ be an ideal of
		$R^\mathrm{univ}_v$
		which is invariant under
		$\widehat{G}^1(R^\mathrm{univ}_v)$
		(group action given in
		Remark \ref{group action on universal ring}).
		Consider the functor
		$\mathcal{D}_v
		\colon
		\mathrm{CNL}_\mathcal{O}
		\rightarrow
		\mathrm{Sets}$:
		for each object
		$R$,
		$\mathcal{D}_v(R)$
		is the set of liftings
		$r$ to $R$
		such that the kernel of the induced map
		$R^\mathrm{univ}_v
		\rightarrow
		R$
		contains $I$.
		Then $\mathcal{D}_v$
		is a local deformation problem.
		Conversely,
		any local deformation problem
		is of the above form for some
		ideal $I$.
	\end{proposition}
	\begin{proof}
		For the first part, (1)-(4) and (6) in 
		Definition \ref{local deformation problem}
		are clearly satisfied. It remains to verify (5).
		Suppose that there is an element
		$g\in
		\widehat{G}^1(R)$
		such that
		$r'=grg^{-1}$.
		Write
		the morphism
		$\phi
		\colon
		R^\mathrm{univ}_v
		\rightarrow
		R$
		corresponding to $r$
		and
		$R'=\mathrm{Im}(\phi)$.
		So $r$ takes values in
		$^CG(R')$.
		By (6),
		$r\in\mathcal{D}_v(R')$.
		Since $R$ is noetherian,
		we can write
		$R=R'[[x_1,x_2,\cdots,x_m]]$
		for some
		$x_1,\cdots,x_m\in\mathfrak{m}_R$
		(using the fact that
		$R'/\mathfrak{m}_{R'}=R/\mathfrak{m}_R$).
		For any integer
		$k>0$,
		set
		\begin{align*}
			R'_k&
			:=R'[[X_1,\cdots,X_m]]/(X_1,\cdots,X_m)^k,
			\\
			R_k&
			:=R'[[x_1,\cdots,x_m]]/(x_1,\cdots,x_m)^k,
		\end{align*}
		where
		$X_1,\cdots,X_m$
		are variables
		and we have a map
		$R_k'\rightarrow
		R_k$
		sending $X_i$ to 
		$x_i$.
		Thus
		$R_k',R_k$ are objects in
		$\mathrm{Art}_\mathcal{O}$
		and
		$R=\lim\limits_{\overleftarrow{k}}R_k$.
		Using the projection
		$R\rightarrow R_k$,
		we write
		$g_k\in
		\widehat{G}^1(R_r)$
		for the image of $g$
		and then using the smoothness of
		$^CG$ over $\mathbb{Z}$,
		one can find a preimage
		$g_k'\in
		\widehat{G}^1(R_k')$
		of $g_k$ under the reduction map
		$\widehat{G}^1(R_k')\to\widehat{G}^1(\kappa)$.
		We then set
		$r_k'=g'_kr(g_k')^{-1}$
		where
		we view $r\in\mathcal{D}_v(R'_k)$
		by the natural map
		$R'\rightarrow R'_k$.
		Clearly these
		$\{r'_k\}_k$
		form a projective system.
		Next we show that
		$r'_k\in\mathcal{D}_v(R'_k)$
		for each $k$.
		Write
		\[
		g'_k=(1+\sum_i\epsilon_{k,i}\gamma_{k,i})g'
		\]
		for
		$g'\in\widehat{G}^1(R')$,
		$\epsilon_{k,i}\in
		\mathrm{M}_{N,N}
		(R'\langle X_1,\cdots,X_m\rangle)$
		each entry being monomials in
		$X_1,\cdots,X_m$
		and
		$\gamma_{k,i}
		\in
		\mathrm{M}_{N,N}
		(R')$.
		Let
		$\phi\colon
		R^\mathrm{univ}_v
		\rightarrow
		R'$
		correspond to the lifting
		$r$.
		Now
		\[
		r_k'
		=
		(1+\sum_i\epsilon_{k,i}\gamma_{k,i})
		g'\gamma'r(\gamma')^{-1}
		(1+\sum_i\epsilon_{k,i}\gamma_{k,i})^{-1}
		=
		(1+\epsilon_k)r'(1+\epsilon_k'),
		\]
		where
		$1+\epsilon_k
		=\sum_i\epsilon_{k,i}\gamma_{k,i}$,
		$1+\epsilon_k'
		=(1+\sum_i\epsilon_{k,i}\gamma_{k,i})^{-1}$
		and
		$r'=g'r(g')^{-1}$.
		We can find a preimage
		$\widetilde{g}\in
		\widehat{G}^1(R^\mathrm{univ}_v)$
		of $g'$.
		Thus the morphism
		$\phi'\colon
		R^\mathrm{univ}_v
		\rightarrow
		R'$
		corresponding to
		$r'$
		is the same as the
		composition
		$R^\mathrm{univ}_v
		\rightarrow
		R^\mathrm{univ}_v
		\xrightarrow{\phi}
		R'$
		where the first morphism corresponds to
		$\widetilde{g}r^\mathrm{univ}_v\widetilde{g}^{-1}$.
		One then verifies easily that
		the kernel of $\phi'$
		is the result of the action of
		$\widetilde{g}$ on
		$\mathrm{Ker}(\phi)$,
		thus contains also the ideal $I$
		and thus
		$r'\in
		\mathcal{D}_v(R')$.
		Suppose
		$r_k'=r'+(\epsilon_kr'+r'\epsilon_k'+
		\epsilon_kr'\epsilon_k')$
		corresponds to a morphism
		$\phi_k'
		\colon
		R^\mathrm{univ}_v
		\rightarrow
		R_k'$.
		We can write
		$\phi_k'=\alpha_k+\beta_k$
		with
		$\alpha_k
		\colon
		R^\mathrm{univ}_v
		\rightarrow
		R'$
		a morphism in $\mathrm{CNL}_\mathcal{O}$
		and
		$\beta_k
		\colon
		R^\mathrm{univ}_v
		\rightarrow
		R'\langle X_1,\cdots,X_m\rangle$
		a map valued in the ideal generated by
		$X_1,\cdots,X_m$.
		Since we have a natural morphism
		$R_k'\rightarrow R'$
		sending all $X_i$ to $0$,
		one sees that $\alpha_k=\phi'$.
		Now let
		$\widetilde{R}_k'$
		be the $R'$-subalgebra of
		$R_k'$ generated by the entries of
		$\epsilon_k$ and $\epsilon_k'$.
		Thus
		$r', r_k'$ are liftings of $\overline{r}$ to
		$\widetilde{R}_k'$ and moreover
		$r'\in\mathcal{D}_v(\widetilde{R}_k')$.
		So by the same argument as above in the case of
		$R'$,
		we see that
		$r_k'\in\mathcal{D}_v(\widetilde{R}_k')$
		and therefore
		$r_k'\in\mathcal{D}_v(R_k')$
		by (2) in Definition \ref{local deformation problem}.
		Again by (2) and (4),
		we have
		\[
		r_k'\in\mathcal{D}_v(R_k),
		\quad
		r'=\lim\limits_{\overleftarrow{k}}r_k'
		\in
		\mathcal{D}_v(R).
		\]
		This gives the first part of the proposition.

		For the second part,
		consider the set
		$\mathcal{I}$
		of ideals $I$ of
		$R^\mathrm{univ}_v$
		such that
		the reduction
		$r^\mathrm{univ}_v(\mathrm{mod}\,I)
		\in
		\mathcal{D}_v(R^\mathrm{univ}/I)$.
		Clearly
		$\mathfrak{m}_{R^\mathrm{univ}_v}
		\in
		\mathcal{I}$ by (1) in Definition \ref{local deformation problem},
		thus
		$\mathcal{I}$ is non-empty.
		For any nested system
		$(I_i)_i$ in $\mathcal{I}$ under inclusion,
		by (4),
		we see that
		$\cap_iI_i\in\mathcal{I}$.
		Note that
		\[
		R_v^{\mathrm{univ}}/(I_1\cap I_2)
		\simeq
		R_v^{\mathrm{univ}}/I_1
		\times_{R_v^{\mathrm{univ}}/(I_1+I_2)}
		R_v^{\mathrm{univ}}/I_2.
		\]
		By (3), for any
		$I_1,I_2\in\mathcal{I}$,
		one has $I_1\cap I_2\in\mathcal{I}$.
		This shows that $\mathcal{I}$
		is in fact a nested system under inclusion.
		Therefore $\mathcal{I}$
		has a minimal element
		$I_\mathrm{min}$.
		By (5), this minimal element
		$I_\mathrm{min}$
		is in fact invariant under
		$\widehat{G}^1(R^\mathrm{univ}_v)$.
		Moreover by (6),
		for any
		$r\in\mathcal{D}_v(R)$,
		the corresponding morphism
		$R^\mathrm{univ}_v
		\rightarrow
		R$ has kernel belonging to
		$\mathcal{I}$
		and thus contains $I_\mathrm{min}$.
		This proves the second part of the proposition.
	\end{proof}

	\begin{definition}\label{L_v}
		Let $\mathcal{D}_v$ be a local deformation problem
		with the corresponding ideal
		$I_v$ of $R^\mathrm{univ}_v$ as in
		Proposition
		\ref{characterization of local deformation problem}.
		Write
		$L_v^1$
		for the subspace of
		$\mathrm{Hom}_\kappa
		(\mathfrak{m}_{R^\mathrm{univ}_v}
		/(\mathfrak{m}_{R^\mathrm{univ}_v}^2,\mathfrak{p}),\kappa)$
		consisting of those morphisms annihilating
		the image of $I_v$ in $\mathfrak{m}_{R^\mathrm{univ}_v}
		/(\mathfrak{m}_{R^\mathrm{univ}_v}^2,\mathfrak{p})$.
		Then under the isomorphisms in
		Lemma \ref{isomorphisms of tangent spaces},
		one writes
		\[
		L_v=L_v(\mathcal{D}_v)
		\]
		for the subspace of
		$H^1(\Gamma_{F_v},\,
		\widehat{\mathfrak{g}}(\kappa))$
		corresponding to $L_v^1$.
	\end{definition}
	
	\begin{remark}\label{dimension formula for L_v1}\rm 
		It follows from the definition that
		there is an isomorphism
		\[
		\mathrm{Hom}_\kappa
		(\mathfrak{m}_{R^\mathrm{univ}_v}/
		(\mathfrak{m}_{R^\mathrm{univ}_v}^2,I_v,\mathfrak{p}),\kappa
		)
		\simeq
		L_v^1.
		\]
		Moreover from (\ref{isomorphism of general tangent spaces-exact sequence}),
		we obtain
		\[
		\mathrm{dim}_\kappa L_v^1
		=
		\mathrm{dim}_\kappa
		\widehat{\mathfrak{g}}(\kappa)
		+
		\mathrm{dim}_\kappa L_v
		-
		\mathrm{dim}_\kappa
		H^0
		(\Gamma_{F_v},\widehat{\mathfrak{g}}(\kappa)).
		\]
		
	\end{remark}

	\begin{lemma}
		Let $\mathcal{D}_v$
		be a local deformation problem
		corresponding to an ideal
		$I_v$ of
		$R^\mathrm{univ}_v$.
		Fix an object
		$R$ in 
		$\mathrm{CNL}_\mathcal{O}$
		and an ideal
		$I$ of $R$
		such that
		$\mathfrak{m}_RI=0$.
		Let
		$r_1,r_2$ be two liftings to $R$ such that
		$r_1\cong r_2(\mathrm{mod}\,I)$
		and
		$r_1\in\mathcal{D}_v(R)$.
		Then
		$r_2\in\mathcal{D}_v(R)$
		if and only if the cocycle
		$(r_2r_1^{-1}-1)
		\in
		L_v\otimes_\kappa I$.
	\end{lemma}
	\begin{proof}
		Suppose that
		$\phi_i
		\colon
		R^\mathrm{univ}_v
		\rightarrow
		R$
		corresponds to the lifting
		$r_i$
		($i=1,2$).
		Write
		\[
		\phi=\phi_2-\phi_1
		\colon
		R^\mathrm{univ}_v
		\rightarrow
		I.
		\]
		One checks that
		$\phi(x+y)=\phi(x)+\phi(y)$,
		$\phi(xy)=\phi(x)\phi_1(y)
		+\phi(y)\phi_1(x)
		+
		\phi(x)\phi(y)$
		for all
		$x,y\in I$
		and
		$\phi(z)=0$
		for all
		$z\in\mathcal{O}$.
		Note that any element
		$x\in
		(R^\mathrm{univ}_v)^\times$
		can be written in the form
		$x=x'(1+y)$
		with
		$x'\in\mathcal{O}^\times$ and
		$y\in
		\mathfrak{m}_{R^\mathrm{univ}_v}$.
		Thus we see,
		by the above conditions,
		that
		for a fixed $\phi_1$,
		$\phi$ depends only on
		$\phi|_{\mathfrak{m}_{R^\mathrm{univ}_v}}$.
		Moreover one verifies easily that
		$\phi$ is trivial on
		$(\mathfrak{m}_{R^\mathrm{univ}_v}^2,\mathfrak{p})$
		and that
		$\phi$ is
		$\mathcal{O}$-linear.
		Thus each
		$r_2$
		corresponds to a $\kappa$-linear map
		\[
		\phi
		\colon
		\mathfrak{m}_{R^\mathrm{univ}_v}/
		(\mathfrak{m}_{R^\mathrm{univ}_v}^2,\mathfrak{p})
		\rightarrow
		I.
		\]
		By Remark
		\ref{isomorphism of general tangent spaces},
		the set of
		liftings
		$r_2\colon
		\Gamma_{F_v}
		\rightarrow
		\,
		^CG(R)$
		such that
		$r_2=r_1(\mathrm{mod}\,I)$
		is in bijection with
		the set of $\mathcal{O}$-linear morphisms
		$\phi_2
		\colon
		R^\mathrm{univ}_v
		\rightarrow
		R$
		such that
		$\phi_2=\phi_1
		(\mathrm{mod}\,I)$,
		thus also in bijection with the sets
		$\mathrm{Hom}_\kappa
		(\mathfrak{m}_{R^\mathrm{univ}_v}/
		(\mathfrak{m}_{R^\mathrm{univ}_v}^2,\mathfrak{p}),I)
		\simeq
		H^1(\Gamma_{F_v},
		\,
		\widehat{\mathfrak{g}}(\kappa))\otimes I$. One checks that
		$r_2\in
		\mathcal{D}_v(R)$
		if and only if the kernel of $\phi_2$ contains $I_v$,
		if and only if
		$\phi$
		factors through
		the quotient
		$\mathfrak{m}_{R^\mathrm{univ}_v}/
		(\mathfrak{m}_{R^\mathrm{univ}_v}^2,\mathfrak{q},I_v)$,
		which is true exactly when
		$\phi\in
		L_v^1\otimes I$,
		\emph{i.e.}
		$(r_2r_1^{-1}-1)
		\in
		L_v\otimes I$.
	\end{proof}

	\begin{definition}\label{T-deformation or lifting of type S}
		We fix a finite set $S$ of finite places of $F$.
		For each $v\in S$,
		we fix also a local deformation problem
		$\mathcal{D}_v$ and write
		\[
		\mathcal{S}=(\mathcal{D}_v)_{v\in S}.
		\]
		A lifting
		$r$ to
		an object $R$ in
		$\mathrm{Art}_\mathcal{O}$
		is said
		\emph{of type $\mathcal{S}$}
		if
		for each $v\in S$,
		$r|_{\Gamma_{F_v}}
		\in
		\mathcal{D}_v(R)$.
		A deformation to $R$ is said
		\emph{of type $\mathcal{S}$}
		if one (hence every) lifting in the equivalence class
		is of type $\mathcal{S}$.		
		We define a functor
		\[
		\mathrm{Def}_\mathcal{S}
		\colon
		\mathrm{Art}_\mathcal{O}
		\rightarrow
		\mathrm{Sets},\quad
		R
		\mapsto
		\{
		\text{deformation to }R \text{ of type }\mathcal{S}
		\}.
		\]

		Write
		$C^i(\Gamma_F,\widehat{\mathfrak{g}}(\kappa))$
		for the space of maps
		$\Gamma_F^i
		\rightarrow
		\widehat{\mathfrak{g}}(\kappa)$
		(which form a cochain complex and the cohomology
		groups are
		$H^i(\Gamma_F,\widehat{\mathfrak{g}}(\kappa))$)
		and similarly for
		$C^i(\Gamma_{F_v},\widehat{\mathfrak{g}}(\kappa))$.
		For each $v\in S$,
		define a subspace
		$M_v^i$ of
		$C^i(\Gamma_{F_v},\widehat{\mathfrak{g}}(\kappa))$ as follows:
		\[
		M_v^i
		:=
		\begin{cases*}
			C^0(\Gamma_{F_v},\widehat{\mathfrak{g}}(\kappa)),
			&
			$v\in S$ and $i=0$;
			\\
			\text{preimage of }L_v
			\text{ in }
			C^1(\Gamma_{F_v}\widehat{\mathfrak{g}}(\kappa)),
			&
			$v\in S$ and $i=1$;
			\\
			0,
			&
			otherwise.
		\end{cases*}
		\]
		Define a new cochain complex $C_{\mathcal{S}}^\bullet(\Gamma_{F},\widehat{\mathfrak{g}}(\kappa))$
		as follows:
		each term is given as
		\[
		C_{\mathcal{S}}^i(\Gamma_F,\widehat{\mathfrak{g}}(\kappa))
		:=
		C^i(\Gamma_F,\widehat{\mathfrak{g}}(\kappa))
		\oplus
		\big(
		\bigoplus_{v\in S}
		C^{i-1}(\Gamma_{F_v},\widehat{\mathfrak{g}}(\kappa))
		/M_v^{i-1}
		\big).
		\]
		The differential map is given  by
		\[
		C^i_{\mathcal{S}}
		(\Gamma_F,\widehat{\mathfrak{g}}(\kappa))
		\rightarrow
		C^{i+1}_{\mathcal{S}}
		(\Gamma_F,\widehat{\mathfrak{g}}(\kappa)),
		\quad
		(\phi;(\phi_v)_{v\in S})
		\mapsto
		(\partial\phi;\phi|_{\Gamma_{F_v}}-\partial\phi_v)_{v\in S}.
		\]
		
		Write
		$H^i_{\mathcal{S}}(\Gamma_F,
		\widehat{\mathfrak{g}}(\kappa))$
		for the cohomology group of this cochain complex
		$C_{\mathcal{S}}^\bullet(\Gamma_{F},\widehat{\mathfrak{g}}(\kappa))$.
	\end{definition}

	It is easy to see
	\begin{lemma}\label{long exact sequence}
		$(C^i_{\mathcal{S}}(\Gamma_F,
		\widehat{\mathfrak{g}}(\kappa)))_i$
		is indeed a cochain complex and
		we have the following long exact sequence
		(write
		$H^i_?$
		for
		$H^i_?(\Gamma_F,
		\widehat{\mathfrak{g}}(\kappa))$
		($?=\emptyset$ or $\mathcal{S}$) and
		similarly
		$H^i_v$ for
		$H^i(\Gamma_{F_v},
		\widehat{\mathfrak{g}}(\kappa))$):
		\begin{equation}\label{long exact sequence for cohomology groups}
			\begin{tikzcd}
				0
				\arrow[r]
				&
				H^0_{\mathcal{S}}
				\arrow[r]
				&
				H^0
				\arrow[r]
				&
				\bigoplus_{v\in S}
				H^0_v
				\arrow[r]
				\arrow[d, phantom, ""{coordinate, name=Z}]
				&
				H^1_{\mathcal{S}}
				\arrow[r]
				&
				H^1
				\arrow[llllld, rounded corners, to path={ -- ([xshift=2ex]\tikztostart.east) |- (Z) [near end]\tikztonodes			
					-| ([xshift=-2ex]\tikztotarget.west)			
					-- (\tikztotarget)}]
				\\
				(\bigoplus_{v\in S}
				H^1_v/L_v)
				\arrow[r]
				&
				H^2_{\mathcal{S}}
				\arrow[r]
				&
				H^2
				\arrow[r]
				&
				\bigoplus_{v\in S}H^2_v
				\arrow[r]
				&
				H_{\mathcal{S}}^3
				\arrow[r]
				&
				H^3=0
			\end{tikzcd}
		\end{equation}

		Suppose moreover that the dimensions of
		$H^i$ and
		$H^i_v$
		($v\in S$)
		are all finite and vanish for $i\gg0$.
		Set
		\begin{equation}\label{Euler characteristic}
			\chi(?)
			=\sum_i(-1)^i\mathrm{dim}_\kappa
			(?^i),
			\quad
			?=H^\bullet,H^\bullet_{\mathcal{S}},H_v^\bullet.
		\end{equation}
		Then one has
		\[
		\chi(H^\bullet_{\mathcal{S}})
		=
		\chi(H^\bullet)
		-
		\sum_{v\in S}
		\chi(H_v^\bullet)
		-
		\sum_{v\in S}
		(
		\mathrm{dim}_\kappa
		L_v
		-
		\mathrm{dim}_\kappa
		H^0_v).
		\]

	\end{lemma}
	\begin{proof}
		We have a cochain complex map
		\[
		C^\bullet
		(\Gamma_F,
		\widehat{\mathfrak{g}}(\kappa))
		\rightarrow
		\oplus_{v\in S}
		C^\bullet
		(\Gamma_{F_v},
		\widehat{\mathfrak{g}}(\kappa))
		/M^\bullet_v,
		\]
		sending
		$\alpha$ to
		$(-\alpha|_{\Gamma_{F_v}})_{v\in S}$.
		Then
		$E^\bullet=C^\bullet_{\mathcal{S}}
		(\Gamma_F,
		\widehat{\mathfrak{g}}(\kappa))$
		is the mapping cone of this cochain map.
		The above long exact sequence
		is then the long exact sequence associated to
		this mapping cone.
		We have $H^3=0$
		because the $p$-cohomological dimension of
		a number field is $2$.
	\end{proof}
	
	\begin{remark}\rm 
		Since each
		$H^i(\Gamma_{F_v},
		\widehat{\mathfrak{g}}(\kappa))$
		is of finite dimension,
		the cohomology
		group
		$H^i(\Gamma_F,
		\widehat{\mathfrak{g}}(\kappa))$
		is of finite dimension if and only if so is
		$H^i_{\mathcal{S}}(\Gamma_F,
		\widehat{\mathfrak{g}}(\kappa))$.
		In the following,
		when we put finite-dimensionality condition
		on
		$H^i(\Gamma_F,
		\widehat{\mathfrak{g}}(\kappa))$,
		this should also be understood as
		put on the space
		$H^i_{\mathcal{S}}
		(\Gamma_F,
		\widehat{\mathfrak{g}}(\kappa))$.
	\end{remark}

	Concerning the representability of the functor
	$\mathrm{Def}_\mathcal{S}$,
	one has
	\begin{proposition}
		Assume that
		$\overline{r}
		\colon
		\Gamma_F
		\rightarrow
		\,^CG(\kappa)$
		is absolutely irreducible and
		$H^1(\Gamma_F,
		\widehat{\mathfrak{g}}(\kappa))$
		is of finite dimension.
		Then
		$\mathrm{Def}_\mathcal{S}$
		is representable
		(by an object
		$R_\mathcal{S}\in
		\mathrm{CNL}_\mathcal{O}$
		and
		$r_\mathcal{S}
		\colon
		\Gamma_F
		\rightarrow
		\,^CG(R_\mathcal{S})$).
		Moreover,
		we have a natural isomorphism
		of finite dimensional $\kappa$-vector spaces
		\[
		\mathrm{Hom}_\kappa
		(
		\mathfrak{m}_{R_\mathcal{S}}/
		(\mathfrak{m}_{R_\mathcal{S}}^2,\mathfrak{p}),\kappa
		)
		\simeq
		H^1_{\mathcal{S}}(\Gamma_F,
		\widehat{\mathfrak{g}}(\kappa)).
		\]
	\end{proposition}
	\begin{proof}
		For the representability of
		$\mathrm{Def}_\mathcal{S}$,
		we apply
		Schlessinger's criterion
		as given in
		\cite[Theorem 2.11]{Schlessinger1968}.
		For the isomorphism of
		$\kappa$-vector spaces,
		we argue as follows:
		we have a natural isomorphism
		\[
		V:=
		\mathrm{Hom}_\kappa
		(\mathfrak{m}_{R_\mathcal{S}}/
		(\mathfrak{m}_{R_\mathcal{S}}^2,\mathfrak{p}),
		\kappa)
		\simeq
		\mathrm{Hom}_{\mathrm{CNL}_\mathcal{O}}
		(R_\mathcal{S},
		\kappa[\epsilon]).
		\]
		By the argument in
		Lemma
		\ref{isomorphisms of tangent spaces},
		$V$ consists of equivalence classes of liftings
		$r$ of type $\mathcal{S}$
		to
		$\kappa[\epsilon]$.
		Any such
		$r$
		gives rise to an element
		$(r\overline{r}^{-1}-1)
		\in
		C^1(\Gamma_F,
		\widehat{\mathfrak{g}}(\kappa))$.
		By definition of the differentials on
		$C^\bullet_{\mathcal{S}}
		(\Gamma_F,
		\widehat{\mathfrak{g}}(\kappa))$,
		the above elements in
		$C^1_{\mathcal{S}}
		(\Gamma_F,
		\widehat{\mathfrak{g}}(\kappa))$
		are in fact cocycles.
		Moreover coboundaries in
		$C^1_{\mathcal{S}}
		(\Gamma_F,
		\widehat{\mathfrak{g}}(\kappa))$
		are of the form
		$\partial g$
		with
		$g\in
		\widehat{\mathfrak{g}}(\kappa)$.
		Thus two elements
		$r$ and
		$r'$ in $V$
		differ by a coboundary
		$\partial g$
		if and only if
		$r'$ is conjugate to $r$ by
		the element
		$\exp(g)
		=1+g$,
		if and only if they lie in the same
		equivalence class in
		$\mathrm{Def}_\mathcal{S}(\kappa[\epsilon])$.
		Therefore we get the isomorphism
		\[
		\mathrm{Hom}_\kappa
		(\mathfrak{m}_{R_\mathcal{S}}/
		(\mathfrak{m}_{R_\mathcal{S}}^2,\mathfrak{p}),
		\kappa)
		\simeq
		H^1_{\mathcal{S}}
		(\Gamma_F,
		\widehat{\mathfrak{g}}(\kappa)).
		\]
	\end{proof}

	\begin{corollary}\label{relative dimension of deformation rings}
		Retain the assumption as in the preceding
		proposition,
		then
		the relative dimension of
		$R_\mathcal{S}$
		over
		$\mathcal{O}$
		is
		$\mathrm{dim}_\kappa
		H^1_{\mathcal{S}}
		(\Gamma_{F},
		\widehat{\mathfrak{g}}(\kappa))$.
	\end{corollary}

	We have a natural perfect pairing of
	$\kappa[\Gamma_{F_v}]$-modules:
	$\widehat{\mathfrak{g}}(\kappa)
	\times
	\widehat{\mathfrak{g}}(\kappa)[1]
	\rightarrow
	\kappa[1]$,
	given by
	sending
	$(x,y)$
	to
	$\mathrm{Tr}(xy)$.
	This induces a perfect pairing on
	cohomology groups
	\begin{equation}\label{pairing on cohomology groups}
		H^1
		\big(
		\Gamma_{F_v},\widehat{\mathfrak{g}}(\kappa)
		\big)
		\times
		H^1
		\big(
		\Gamma_{F_v},\widehat{\mathfrak{g}}(\kappa)[1]
		\big)
		\rightarrow
		H^2(\Gamma_{F_v},\kappa[1])
		\simeq
		\kappa.
	\end{equation}

	\begin{definition}
		Let
		$\Gamma_{F,S}$
		be the Galois group of the maximal extension of
		$F$ unramified outside $S\cup\{\infty\}$.
		Using the pairing (\ref{pairing on cohomology groups}), we write
		\[
		L_v^\perp
		\subset
		H^1(\Gamma_{F_v},
		\widehat{\mathfrak{g}}(\kappa)[1])
		\]
		for the the subspace which annihilates
		$L_v$.
		Then we set
		\[
		H^1_{\mathcal{S}^\perp}
		(\Gamma_{F,S},
		\widehat{\mathfrak{g}}(\kappa)[1])
		:=
		\mathrm{Ker}
		\left(
		H^1(\Gamma_{F,S},
		\widehat{\mathfrak{g}}(\kappa)[1])
		\rightarrow
		\bigoplus_{v\in S}
		H^1
		(\Gamma_{F_v},
		\widehat{\mathfrak{g}}(\kappa)[1])/L_v^\perp
		\right).
		\]
	\end{definition}

	From now on we assume that the residual Galois representation $\overline{r}$ is \emph{unramified outside $S$}.	
	Concerning the dimensions of the
	Galois cohomology groups,
	we have
	\begin{proposition}
		Retain the above notations.
		Moreover, we write
		$H_?^i$
		for the space
		$H_?^i(\Gamma_{F,S},
		\widehat{\mathfrak{g}}(\kappa))$,
		$H_?^i[1]$
		for
		$H_?^i(\Gamma_{F,S},
		\widehat{\mathfrak{g}}(\kappa)[1])$
		and
		$H^i_v$ for
		$H^i(\Gamma_{F_v},
		\widehat{\mathfrak{g}}(\kappa))$,
		and write
		$h^i_?$,
		$h^i_?[1]$,
		$h^i_v$, $l_v$ and
		$l_v^\perp$
		for the dimensions of
		$H^i_?$,
		$H^i_?[1]$,
		$H^i_v$, $L_v$
		and
		$L_v^\perp$
		over $\kappa$. Then we have
		\begin{enumerate}
			\item 
			$H^0_{\mathcal{S}}
			=
			H^0$,
			thus
			$h^0_\mathcal{S}=h^0$.

			\item
			$h^2_{\mathcal{S}}
			=
			h^1_{\mathcal{S}^\perp}[1]$.

			\item 
			$
			h^3_{\mathcal{S}}
			=
			h^0[1]$.
			
			\item 
			For $i>3$,
			$h^i_{\mathcal{S}}=0$.
			
			\item 
			$\chi
			(H_{\mathcal{S}}^\bullet)
			=
			\sum_{v |\infty}h_v^0
			+
			\sum_{v\in S}
			(h_v^0-l_v)$.
			(\emph{cf.}
			(\ref{Euler characteristic})).
		\end{enumerate}
	\end{proposition}

	\begin{proof}
		Note that the long exact sequence
		(\ref{long exact sequence for cohomology groups})
		is still valid if we replace the group
		$\Gamma_F$
		by
		$\Gamma_{F,S}$
		by our assumption on
		$\overline{r}$.
		\begin{enumerate}
			\item 
			Use the first three terms in
			(\ref{long exact sequence for cohomology groups}).

			\item 
			We have the Poitou-Tate exact sequence
			\[
			\begin{tikzcd}
				0
				\arrow[r]
				&
				H^0
				\arrow[r]
				&
				\oplus_{v\in S}
				H^0_v
				\arrow[r]
				\arrow[d, phantom, ""{coordinate, name=Z}]
				&
				H^2[1]^\vee
				\arrow[lld, rounded corners, to path={ -- ([xshift=2ex]\tikztostart.east) |- (Z) [near end]\tikztonodes			
					-| ([xshift=-2ex]\tikztotarget.west)			
					-- (\tikztotarget)}]
				&
				\\
				&
				H^1
				\arrow[r]
				&
				\oplus_{v\in S}
				H^1_v
				\arrow[r]
				\arrow[d, phantom, ""{coordinate, name=Z}]
				&
				H^1[1]^\vee
				\arrow[lld, rounded corners, to path={ -- ([xshift=2ex]\tikztostart.east) |- (Z) [near end]\tikztonodes			
					-| ([xshift=-2ex]\tikztotarget.west)			
					-- (\tikztotarget)}]
				&
				\\
				&
				H^2
				\arrow[r]
				&
				\oplus_{v\in S}
				H^2_v
				\arrow[r]
				&
				H^0[1]^\vee
				\arrow[r]
				&
				0
			\end{tikzcd}
			\]
			where $M^\vee$ is the $\kappa$-dual of $M$.
			Write
			$H$ for
			$\oplus_{v\in S} H^1/L_v$.
			By
			(\ref{long exact sequence for cohomology groups}),
			we have
			\begin{align*}
				h^2_{\mathcal{S}}
				&
				=
				\mathrm{dim}_\kappa
				\mathrm{Ker}
				(H^2_{\mathcal{S}}
				\rightarrow
				H^2)
				+
				\mathrm{dim}_\kappa
				\mathrm{Im}
				(H^2_{\mathcal{S}}
				\rightarrow
				H^2)
				\\
				&
				=
				\mathrm{dim}_\kappa
				\mathrm{Im}
				(
				H
				\rightarrow
				H^2_{\mathcal{S}}
				)
				+
				\mathrm{dim}_\kappa
				\mathrm{Im}
				(H^2_{\mathcal{S}}
				\rightarrow
				H^2)
				\\
				&
				=
				\mathrm{dim}_\kappa
				H
				-
				\mathrm{dim}_\kappa
				\mathrm{Ker}
				(
				H
				\rightarrow
				H^2_{\mathcal{S}}
				)
				+
				\mathrm{dim}_\kappa
				\mathrm{Im}(H^2_{\mathcal{S}}
				\rightarrow
				H^2)
				\\
				&
				=
				\mathrm{dim}_\kappa
				H
				-
				\mathrm{dim}_\kappa
				\mathrm{Im}
				(
				H^1
				\rightarrow
				H
				)
				+
				\mathrm{dim}_\kappa
				\mathrm{Ker}
				(
				H^2
				\rightarrow
				\oplus_{v\in S}H^2_v
				).
			\end{align*}
			On the other hand,
			\begin{align*}
				h^1_{\mathcal{S}^\perp}[1]
				&
				=
				\mathrm{dim}_\kappa
				\mathrm{Ker}
				\left(
				H^1[1]
				\rightarrow
				\oplus_{v\in S}
				H^1_v[1]/L_v^\perp
				\right)
				\\
				&
				=
				\mathrm{dim}_\kappa
				\mathrm{Ker}
				\left(
				H^1[1]
				\rightarrow
				\oplus_{v\in S}
				H^1_v[1]
				\right)
				+
				\sum_{v\in S}
				l_v^\perp
				\\
				&
				=
				\mathrm{dim}_\kappa H
			\end{align*}

			Another way to prove this is
			as in the proof of
			\cite[Lemma 2.3.4]{ClozelHarrisTaylor2008}.
			We have the following natural commutative diagram
			with exact rows
			\[
			\begin{tikzcd}
				H^1
				\arrow[r]
				\arrow[d,"="]
				&
				\oplus_{v\in S}
				H^1_v
				\arrow[r]
				\arrow[d,"\mathrm{pr}"]
				&
				H^1[1]^\vee
				\arrow[r]
				\arrow[d,"\mathrm{Ker}^\vee"]
				&
				H^2
				\arrow[d,"="]
				\\
				H^1
				\arrow[r]
				&
				(\oplus_{v\in S}
				H^1_v/L_v)
				\arrow[r]
				&
				H^1_{\mathcal{S}^\perp}[1]^\vee
				\arrow[r]
				&
				H^2
			\end{tikzcd}
			\]
			Therefore one obtains
			another exact sequence
			from the Poitou-Tate long exact sequence:
			\[
			H^1
			\rightarrow
			(\oplus_{v\in S}
			H^1_v/L_v)
			\rightarrow
			H^1_{\mathcal{S}}[1]^\vee
			\rightarrow
			H^2
			\rightarrow
			\oplus_{v\in S}
			H^2
			\rightarrow
			H^0[1]^\vee
			\rightarrow
			0
			\]
			Compare the first three terms of
			this sequence with the second row of
			(\ref{long exact sequence for cohomology groups})
			and use (3),
			then we get that
			$h_{\mathcal{S}}^2
			=
			h_{\mathcal{S}^\perp}^1$.

			\item 
			Compare the last four terms of the above exact
			sequence with the
			last four terms of
			(\ref{long exact sequence for cohomology groups}).
			One verifies that
			the maps
			$H^2
			\rightarrow
			\oplus_{v\in S}
			H^2_v$
			in these two long exact sequences
			coincide
			(both are restriction maps).
			Therefore we have
			$h^3_{\mathcal{S}}
			=
			h^0[1]$.

			\item 
			This follows from
			the last terms of
			(\ref{long exact sequence for cohomology groups}).
			
			\item
			By the local and global Euler-Poincar\'{e}
			characteristic formula
			(\cite[Chapter I, Theorems 2.8 \& 5.1]{Milne2006}),
			we know that
			\begin{align*}
				\sum_{v\in S}\chi(H_v^\bullet)
				&
				=
				\sum_{v\in S}
				(h_v^0-h_v^1+h_v^2)
				=
				\sum_{v\in S}
				\# 
				(
				\mathcal{O}_{F,v}/
				\# (\widehat{\mathfrak{g}}(\kappa))
				\mathcal{O}_{F,v}
				)
				\\
				&
				=
				\sum_{v |p}
				\# 
				(
				\mathcal{O}_{F,v}/
				\# (\widehat{\mathfrak{g}}(\kappa))
				\mathcal{O}_{F,v}
				)
				=
				\# 
				(
				\mathcal{O}_{F,p}/
				\# (\widehat{\mathfrak{g}}(\kappa))
				\mathcal{O}_{F,p}
				)
				\\
				&
				=
				[F:\mathbb{Q}]
				\# 
				(
				\widehat{\mathfrak{g}}(\kappa)
				),
				\\
				\chi(H^\bullet)
				&
				=
				h^0-h^1+h^2
				=
				\sum_{v |\infty}
				(
				h_v^0-
				|\# 
				(
				\widehat{\mathfrak{g}}(\kappa)
				)|_v
				)
				\\
				&
				=
				\sum_{v |\infty}h_v^0
				-
				[F:\mathbb{Q}]
				\# 
				(
				\widehat{\mathfrak{g}}(\kappa)
				).
			\end{align*}
			This gives us the expression for
			$\chi(H_{\mathcal{S}}^\bullet)$
			as in the proposition
			using the formula given in
			Lemma
			\ref{long exact sequence}.
			
		\end{enumerate}
		
	\end{proof}
	
	\begin{remark}\label{dimension for infinite H_v^0}\rm 
		By Theorem \ref{Galois representations attached to auto rep}, for each place $v|\infty$, we have
		
		\[
		h_v^0
		=
		\begin{cases*}
			N(N+1)/2,
			&
			$G^\ast=\mathrm{SO}_n$ with $n$ odd;
			\\
			N(N-1)/2,
			&
			$G^\ast=\mathrm{Sp}_n$ or $G^\ast=\mathrm{SO}_n^\eta$ with $n$ even.
		\end{cases*}
		\]
	\end{remark}
	
	Putting together the results in the preceding
	proposition,
	\begin{corollary}\label{dimension of H^1}
		We have
		\[
		h^1_{\mathcal{S}}
		=
		h^1_{\mathcal{S}^\perp}[1]
		-
		h^0[1]
		-
		\sum_{v |\infty}h_v^0
		+
		\sum_{v\in S}
		(l_v
		-
		h_v^0
		)
		+
		h^0.
		\]
		Moreover,
		the universal deformation ring
		$R_\mathcal{S}$
		is of relative dimension
		$\leq h_{\mathcal{S}}^1$
		over
		$\mathcal{O}$.
	\end{corollary}

	\subsection{Local deformation problems}
	In this subsection,
	we study some local deformation problems.

	\subsubsection{Fontaine-Lafaille deformations}\label{Fontaine-Lafaille deformations}

	We consider crystalline deformations in the Fontaine-Lafaille range (also called \emph{Fontaine-Lafaille deformations}) as in \cite[Definition 5.1]{Booher2019} (for $n$ even).\footnote{Note that in \cite{Booher2019}, the Galois representations take values in $H=\mathrm{GSp}_N$ or $H=\mathrm{GO}_N$, instead of our groups $^CG$. However, in \emph{loc.cit}, a lifting of the similitude factor of $H$ is fixed while in our case, a lifting of the character $\mu_{\overline{r}}$ is fixed. Thus it is easy to see that the deformation theory of \emph{loc.cit} and our case are equivalent.} Recall our convention is that the $p$-adic cyclotomic character has Hodge-Tate weight $-1$.

	We briefly recall the set-up in \cite{Booher2019}: we fix a place $v$ of $F$ unramified over $p$ and write $[F]_v$ for the set of embeddings $\tau\colon F_v\to\overline{\Q}_p$. We set $\sigma_v\colon
	\mathcal{O}_v
	\rightarrow
	\mathcal{O}_v$
	to be the Frobenius automorphism. We assume that
	
	\begin{equation}\label{(C3)}
		\makebox[0pt][l]{\hspace{-2cm}(C3)}
		\begin{cases*}
			\text{(1) the character $\nu_v\colon\Gamma_{F_v}\to\Ocal^\times$ is crystalline with Hodge-Tate weights $(w_\tau)_{\tau\in[F]_v}$}
			\\
			\text{in an interval of length $p-2$; (2)
				$\overline{r}$ is torsion-crystalline with Hodge-Tate weights}
			\\
			\text{in an interval $[a_v,b_v]$ of length $(p-2)/2$}
			\index{C@(C3)}
		\end{cases*}
		\notag
	\end{equation}
	The last condition means that there is a crystalline representation $V$ of $\Gamma_{F_v}$ with Hodge-Tate weights in $[a_v,b_v]$ and $\Gamma_{F_v}$-stable lattices $\Lambda\subset\Lambda'$ in $V$ such that $\Lambda'/\Lambda$ is isomorphic to the representation $\overline{r}$ (\emph{cf.} \cite[§4.1]{Booher2019}).

	\begin{remark}\label{remark on (C3) and (C6)}\rm
		Let $r=r_\pi$ be the Galois representation as in Theorem \ref{Galois representations attached to auto rep}, then for $v$ unramified over $p$, the second part of (C3) is satisfied for $\overline{r}$ if and only if $2a_{v,1}\leq(p-2)/2$ (here $a_{v,1}$ is part of the infinitesimal character of $\pi$ as in (\ref{a_1>a_2...})); the first part of (C3) is always satisfied for $\mu_{\overline{r}}$ (whose Hodge-Tate weight is $-1$).
	\end{remark}

	\begin{definition}\label{crystalline local deformation problem}
		We define the \emph{Fontaine-Laffaille deformation} functor
		\[
		\mathcal{D}_v^\mathrm{FL}
		=
		\mathcal{D}_{v,\overline{r}_v}^\mathrm{FL}
		\colon\mathrm{CNL}_\Ocal\to\mathrm{Sets}
		\]
		as follows: for each object $R$ in $\mathrm{CNL}_\Ocal$, $\mathcal{D}_v^\mathrm{FL}(R)$ consists of liftings $r\colon\Gamma_{F_v}\to\,^CG(R)$ of $\overline{r}_v$ such that $\mu_r=\nu$ and
		the Hodge-Tate weights of $r$ lie in the interval $[a_v,b_v]$ (\emph{cf.}\cite[Definition 5.1]{Booher2019}).
	\end{definition}

	Then we have
	\begin{theorem}\label{crystalline deformation-dimension defaut}
		The Fontaine-Laffaille deformation functor $\mathcal{D}_v^\mathrm{FL}$ is a local deformation problem. Suppose for each $\tau\in[F]_v$, the Hodge-Tate weights of $\overline{r}_v$ at $\tau$ is \emph{multiplicity-free}, then
		\[
		\mathrm{dim}_{\kappa}(L_v)
		-
		\mathrm{dim}_{\kappa}H^0(\Gamma_{F_v},\widehat{\mathfrak{g}}(\kappa))
		=
		[F_v:\Q_p]
		(\mathrm{dim}(\widehat{G}_{\kappa_v})-\mathrm{dim}(\widehat{B}_{\kappa_v})).
		\]
		Here $\widehat{B}_{\kappa_v}$ is a Borel subgroup of $\widehat{G}_{\kappa_v}$.
	\end{theorem}
	\begin{proof}
		In Definition \ref{local deformation problem}, it is clear that (1) is satisfied for $\mathcal{D}_v^\mathrm{FL}$. The conditions (2), (3) and (5) are given in \cite[Corollary 5.6]{Booher2019}. For (4), we can assume $R_i$ are all Artinian and the $\Gamma_{F_v}$-module $R_i^N$ associated to $r_i$ is of the form $T_{\mathrm{cris}}(M_i)$ for some $M_i$ in $\mathrm{MF}_{\mathcal{O},\mathrm{tor}}^{f,[a,b]}$ (for the undefined notations $T_{\mathrm{cris}}$ and $\mathrm{MF}_{\mathcal{O},\mathrm{tor}}^{f,[a,b]}$, we refer to \cite{Booher2019}). For $i>j$, we have
		\[
		R_j\otimes_{R_i}T_{\mathrm{cris}}(M_i)
		=
		T_{\mathrm{cris}}(R_j\otimes_{R_i}M_i).
		\]
		Thus $M_j=R_j\otimes_{R_i}M_i$ and we write $M:=\lim\limits_{\overleftarrow{n}}M_n$ and $R=\lim\limits_{\overleftarrow{n}}R_n$. Then $R^N=T_{\mathrm{cris}}(M)$ and one checks that $\lim\limits_{\overleftarrow{n}}r_n$ gives rise to an action of $\Gamma_{F_v}$ on $R^N$ which is an element in $\mathcal{D}_v^\mathrm{FL}(R)$ (see \cite[§5.1]{Booher2019} and \cite{Ramakrishna1993}). Similarly, one proves (6) using the characterization of the essential image of $T_{\mathrm{cris}}$ \cite[Corollary 5.5]{Booher2019}.

		The dimension formula is given in \cite[Theorem 5.2]{Booher2019}. 
	\end{proof}

	\begin{remark}\rm
		Let $r=r_\pi$ be as in Theorem \ref{Galois representations attached to auto rep}. Since $\pi$ is assumed to satisfy (C2), the Hodge-Tate weight of $r|_{\Gamma_{F_v}}$ at each $\tau\in[F]_v$ is multiplicity-free (\emph{cf.} Lemma \ref{regular and std-regular, characterisation}).
	\end{remark}

	\subsubsection{Taylor-Wiles deformations}
	In this subsection we prepare some results for 
	the next subsection on Taylor-Wiles primes.
	Take a finite place
	$v$ of $F$ such that
	$q_v\equiv1
	(\mathrm{mod}\,p)$.
	We assume that $v$ satisfies the following
	\begin{assumption}\label{assumption on Taylor-Wiles deformations-Thorne}
		The residual Galois representation
		$\overline{r}_v$
		of
		$\Gamma_{F_v}$
		is unramified
		and
		there is a decomposition into non-degenerate quadratic/symplectic spaces
		\[
		(\overline{r}_v,(-,-))
		=
		(\overline{\psi}_v,(-,-)|_{\overline{\psi}_v})
		\oplus
		(\overline{s}_v,(-,-)|_{\overline{s}_v})
		\]
		where
		$\overline{\psi}_v$
		corresponds
		to the eigenspace of
		Frobenius $\sigma_v$
		of certain eigenvalue,
		which is one of the two types listed in
		Proposition
		\ref{projector induces an isomorphism},
		and
		$\sigma_v$
		acts semi-simply on
		$\overline{\psi}_v$.
		Moreover
		$\overline{\psi}_v$
		is not isomorphic to a sub-quotient of
		$\overline{s}_v$.
	\end{assumption}

	\begin{definition}\label{Taylor-Wiles deformations-Thorne}
		Maintain Assumption \ref{assumption on Taylor-Wiles deformations-Thorne}. We define a local deformation functor
		$\mathcal{D}_v$
		as follows:
		for
		$R$ in
		$\mathrm{CNL}_\mathcal{O}$,
		$\mathcal{D}_v(R)$
		is the set of liftings
		$r$ of
		$\overline{r}_v$
		to
		$R$
		such that
		$r$
		is equivalent to
		a lifting
		$\psi_v\oplus s_v$
		of
		$\overline{\psi}_v\oplus\overline{s}_v$
		with
		$s_v$
		unramified
		and
		$\psi_v | _{I_{F_v}}$
		is a pair of
		characters
		$(\psi,\psi^{-1})$
		in the form of
		Theorem
		\ref{decomposition of the Galois representation}.
		Moreover there is a decomposition into non-degenerate quadratic/symplectic spaces
		\[
		(r,(-,-))=(\psi_v,(-,-)|_{\psi_v})\oplus(s_v,(-,-)|_{s_v}).
		\]
	\end{definition}

	It follows from the definition that
	$\psi_v$ and
	$s_v$
	are uniquely determined by
	$r$
	and we will write them as
	$\psi_v(r)$ and $s_v(r)$.
	We have
	\begin{lemma}\label{Taylor-Wiles deformations a la Thorne}
		$\mathcal{D}_v$
		is a local deformation problem.
	\end{lemma}
	\begin{proof}
		(1)-(2) and (5)-(6) in 
		Definition
		\ref{local deformation problem}
		are easily verified. It remains to verify (3) and (4).		
		We use the notations from
		Definition
		\ref{local deformation problem}.
		Suppose that
		$\overline{s}_v$
		is of dimension $d_s$.

		We first consider (3).
		Fix a basis
		$(e_1,\cdots,e_N)$
		for
		$R_3^N$
		and 
		denote by the same symbol their images in
		$R_1^N,R_2^N,R_0^N,\kappa^N$.
		We assume that
		\[
		\overline{s}_v
		=
		\kappa(e_1,\cdots,e_{d_s+1}),
		\quad
		\overline{\psi}_v
		=
		\kappa(e_{d_s},\cdots,e_n).
		\]
		Suppose now that
		$(r_?,(-,-))=(\psi_?,(-,-)|_{\psi_?})\oplus(s_?,(-,-)|_{s_?})$ ($?=1,2$).
		We then fix a basis
		$(e_1^{(?)},\cdots,e_{d_s}^{(?)},
		e_{d_s+1}^{(?)},\cdots,e_n^{(?)})$
		of
		$R_?^n$
		and write
		\[
		s_?=R_?(e_1^{(?)},\cdots,e_{d_s}^{(?)}),
		\quad
		\psi_?=R_?(e_{d_s+1}^{(?)},\cdots,e_n^{(?)}).
		\]
		Then
		write
		$\overline{e}_j^{(i)}$
		for the image of
		$e_j^{(i)}$
		in
		$R_0^n$.
		So one has
		$\overline{e}_j^{(1)}=\overline{e}_j^{(2)}$
		for $j=1,\cdots,n$.
		Write
		$A_?\in\,^CG(R_?)$
		for the basis change matrix from
		$(e_j)_j$
		to
		$(e_j^{(?)})_j$
		and
		$\overline{A}_?$
		its image in
		$^CG(R_0)$
		(thus
		$\overline{A}_1=\overline{A}_2
		\in
		\,^CG(R_0)$).
		Since
		$\psi_?\oplus s_?$
		lifts
		$\overline{\psi}_v\oplus\overline{s}_v$,
		one has
		$A_?\in\widehat{G}^1(R_?)$
		($?=1,2$).
		Put
		$A=(A_1,A_2)
		\in\widehat{G}^1(R_3)$,
		then we get
		\[
		A(r_1\times r_2)A^{-1}=(\psi_1\times\psi_2)\oplus(s_1\times s_2)
		\]
		and thus
		$r_1\times r_2\in\mathcal{D}_v(R_3)$.

		For (4),
		suppose that we have
		$r_i=g_i(\psi_i\oplus s_i)g_i^{-1}$
		with
		$g_i\in\widehat{G}^1(R_i)$
		for any
		$i\in I$.
		Write the partial order in
		$I$ as $\ge$.
		For any
		$i\ge j\in I$,
		we write
		$r_{i,j}
		\colon
		\Gamma_{F_v}
		\xrightarrow{r_i}
		\,^CG(R_i)
		\rightarrow
		\,^CG(R_j)$.
		Then put
		\[
		r_j'=\lim\limits_{j\le i\in I}
		r_{i,j}
		\in
		\mathcal{D}_v(R_j).
		\]
		Similarly, one defines
		$g_j',\psi_j',s_j'$.
		Thus one has
		$r_i'=g'_i(\psi'_i\oplus s'_i)(g'_i)^{-1}$.
		Then
		$(r_i')_{i\in I}$
		is again a projective system
		and we put
		\[
		r=\lim\limits_{\overleftarrow{i\in I}}r_i',
		\]
		which,
		by construction,
		is equal to
		$\lim\limits_{\overleftarrow{i\in I}}r_i$.
		Moreover,
		for any
		$i\ge j\in I$,
		the image of
		$g_i'$
		in
		$^CG(R_j)$
		is equal to
		$g_j'$
		and similarly the image of
		$s_i'$ in
		$R_j$
		is equal to
		$s_j'$,
		the image of
		$\psi_i'$ in $R_j$ is equal to
		$\psi_j'$.
		From these projective systems,
		one puts
		$g=\lim\limits_{\overleftarrow{i\in I}}g_i'$,
		$s=\lim\limits_{\overleftarrow{i\in I}}s_i'$
		and
		$\psi=\lim\limits_{\overleftarrow{i\in I}}\psi_i'$.
		Moreover,
		by construction
		\[
		r=g(\psi\oplus s)g^{-1}.
		\]
		Thus we have
		$r\in\mathcal{D}_v'(R)$.
	\end{proof}

	Next we determine
	the subspace
	$L_v$
	of
	$H^1(\Gamma_{F_v},\widehat{\mathfrak{g}}(\kappa))$.
	We fix a basis
	$(e_1,e_2,\cdots,e_{2r})$
	of
	$\overline{\psi}_v$ such that
	the bilinear form $(-,-)|_{\overline{\psi}_v}$
	is given by
	\[
	(x_1e_1+\cdots+x_{2r}e_{2r},y_1e_1+\cdots+y_{2r}e_{2r})|_{\overline{\psi}_v}
	=
	(x_1y_{2r}+\cdots+x_ry_{r+1})+\varepsilon(x_{r+1}y_r+\cdots+x_{2r}y_1).
	\]
	Here $\varepsilon$ is equal to $1$ or $-1$, depending on whether the bilinear form $(-,-)$ is symmetric or symplectic. We write
	$\widetilde{Z}(\mathrm{Ad}\,\overline{\psi}_v)$
	to be the
	$1$-dimensional
	subspace of
	$\mathrm{Ad}\,\overline{\psi}_v$
	consisting of
	diagonal matrices of the form
	$\mathrm{diag}(x\cdot1_r,-x\cdot1_r)$
	with
	$x\in\kappa$.
	\begin{lemma}
		$L_v$
		is the subspace of
		$H^1(\Gamma_{F_v},\mathrm{Ad}\,\overline{s}_v)
		\oplus
		H^1(\Gamma_{F_v},\mathrm{Ad}\,\overline{\psi}_v)$
		consisting of those $1$-cocycles
		$(\phi_1,\phi_2)$
		such that
		$\phi_1|_{I_{F_v}}$
		is trivial,
		$\phi_2|_{I_{F_v}}$
		lies in
		$\widetilde{Z}(\mathrm{Ad}\,\overline{\psi}_v)$.
		Moreover,
		we have
		\[
		\mathrm{dim}_\kappa
		L_v
		-
		\mathrm{dim}_\kappa
		H^0(\Gamma_{F_v},\widehat{\mathfrak{g}}(\kappa))
		=
		1.
		\]
	\end{lemma}
	\begin{proof}
		The first one follows from the definition of
		$\mathcal{D}_v(\kappa[\epsilon])$.
		For the second,
		note that we have the following long exact sequence
		for any finite
		$\kappa[\Gamma_{F_v}]$-module $M$
		such that
		$I_{F_v}$
		acts trivially on $M$,
		\[
		0
		\rightarrow
		H^0(\Gamma_{F_v},
		M)
		\rightarrow
		H^0(I_{F_v},M)
		\xrightarrow{\sigma_v-1}
		H^0(I_{F_v},M)
		\rightarrow
		H^1(\Gamma_{F_v}/I_{F_v},
		M)
		\rightarrow
		0.
		\]
		Therefore
		\begin{equation}\label{comparison H0 and H1}
			\mathrm{dim}_\kappa
			H^0(\Gamma_{F_v},M)
			=
			\mathrm{dim}_\kappa
			H^1(\Gamma_{F_v}/I_{F_v},M).
		\end{equation}
		Moreover,
		since
		$\overline{\psi}_v$ is not a subquotient of
		$\overline{s}_v$,
		we have a decomposition
		\[
		H^0
		(\Gamma_{F_v},
		\widehat{\mathfrak{g}}(\kappa))
		=
		H^0(\Gamma_{F_v},\mathrm{Ad}\,\overline{\psi}_v)
		\oplus
		H^0(\Gamma_{F_v},\mathrm{Ad}\,\overline{s}_v),
		\]
		(the terms $\mathrm{Hom}_{\kappa[\Gamma_{F_v}]}
		(\overline{s}_v,\overline{\psi}_v)$ and
		$\mathrm{Hom}_{\kappa[\Gamma_{F_v}]}
		(\overline{\psi}_v,\overline{s}_v)$
		vanish).
		Now compare
		the second factor of
		$L_v$
		with
		$H^0(\Gamma_{F_v},\mathrm{Ad}\,\overline{\psi}_v)
		=
		H^1(\Gamma_{F_v}/I_{F_v},\mathrm{Ad}\,\overline{\psi}_v)$:
		since
		$\widetilde{Z}
		(\mathrm{Ad}\,\overline{\psi}_v)$
		is of dimension $1$,
		we see immediately
		\[
		\mathrm{dim}_\kappa L_v
		-
		\mathrm{dim}_\kappa
		H^0(\Gamma_{F_v},
		\widehat{\mathfrak{g}}(\kappa))
		=
		1.
		\]
		
	\end{proof}

	We write
	$\Delta_v$
	for the maximal $p$-power quotient of the
	inertial subgroup of the abelienization
	$\Gamma_{F_v}^\mathrm{ab}$ of
	$\Gamma_{F_v}$.
	So $\Delta_v$ is a cyclic finite group of order the maximal $p$-power dividing
	$q_v-1$.
	For any $r\in\mathcal{D}_v(R)$,
	the associated character
	$\psi_v(r)
	\colon
	\Gamma_{F_v}
	\rightarrow
	R^\times$
	takes $I_{F_v}$ into
	$1+\mathfrak{m}_R$,
	thus
	$\psi_v(r)|_{I_{F_v}}$
	factors through $\Delta_v$.
	From this we get a morphism of
	$\mathcal{O}$-algebras
	\begin{equation}\label{a character of the universal deformation ring}
		\psi_v(r)
		\colon
		\mathcal{O}[\Delta_v]
		\rightarrow
		R.
	\end{equation}
	Write
	$\mathfrak{a}_v
	=
	\mathrm{Ker}
	(\mathcal{O}[\Delta_v]
	\rightarrow
	\mathcal{O})$
	for the augmentation ideal,
	then
	$R/\mathfrak{a}_vR$
	is the maximal quotient of $R$ such that
	the image of $r$ in this quotient is unramified.

	\subsubsection{Existence of Taylor-Wiles primes}
	In this subsection we show the existence
	of Taylor-Wiles primes.
	We need to put some conditions on the image
	of the residual Galois representation
	$\overline{r}$.
	These conditions appear naturally when we try to
	estimate the dimension of
	the universal deformation ring
	$R_\mathcal{S}^T$.
	This is analogous to
	the notions of ``big" in
	\cite[Definition 2.5.1]{ClozelHarrisTaylor2008}
	and
	``adequate" in
	\cite[Definition 2.3]{Thorne2012}:
	
	\begin{definition}\label{big}
		A subgroup
		$H$ of
		$^CG(\kappa)$ is
		\emph{adequate} if the following two conditions hold
		\begin{enumerate}
			\item 
			$H^0(H,\widehat{\mathfrak{g}}(\kappa))=0$,
			
			\item 
			$H^1(H,
			\widehat{\mathfrak{g}}(\kappa))=0$,

			\item 
			$\mathrm{Hom}(H,\kappa)=0$.

			\item 
			for any nonzero
			$\kappa[H]$-submodule
			$W$ of
			$\widehat{\mathfrak{g}}(\kappa)$,
			there is an element $\gamma\in H$
			such that the action of $\gamma$ on $\kappa^N$ (via $\overline{r}$)
			has an eigenvalue $a$ with
			the corresponding generalized eigenspace
			$V_{\gamma,a}$ satisfying the following:
			let
			\[
			\pi_{\gamma,a}
			\colon
			\kappa^N
			\rightarrow
			V_{\gamma,a}
			\]
			denote the $\gamma$-equivariant projection onto
			$V_{\gamma,a}$ and
			\[
			i_{\gamma,a}
			\colon
			V_{\gamma,a}
			\rightarrow
			\kappa^N
			\]
			denote the
			$\gamma$-equivariant injection into $\kappa^N$.
			Then
			we require that the trace
			$
			\mathrm{Tr}
			\left(
			\pi_{\gamma,a}\circ w\circ i_{\gamma,a}
			\right)
			\ne
			0$
			for some $w\in W$.
		\end{enumerate}		
	\end{definition}

	In \cite[Theorem 9 of Appendix]{Thorne2012}, the author
	gives
	some sufficient conditions for
	a subgroup $H$ of
	$^CG(\kappa)$
	to be adequate:
	\begin{lemma}\label{when a subgroup is adequate}
		Let $H$ be a subgroup of $^CG(\kappa)$. Assume that the following hold:
		\begin{enumerate}
			\item
			$p\ge 2(N+1)$,
			
			\item 
			$\kappa$
			contains all the eigenvalues of
			all elements of $\overline{r}(H)$,

			\item 
			$\overline{r}$ is absolutely irreducible,
		\end{enumerate}
		then
		$H$
		is adequate.
	\end{lemma}

	Recall we have fixed
	$\mathcal{S}=(\mathcal{D}_v)_{v\in S}$
	from
	Definition
	\ref{T-deformation or lifting of type S}. Let $Q$ be a finite set of finite primes $v\notin S$ of $F$
	such that
	$q_v\equiv1
	(\mathrm{mod}\,p)$
	and
	$\overline{r}_v$ satisfies
	Assumption
	\ref{assumption on Taylor-Wiles deformations-Thorne}.
	Then we write
	\[
	S(Q)=S\sqcup Q,
	\quad
	\mathcal{S}(Q)
	=
	(\mathcal{D}_v)_{v\in S(Q)}.
	\index{S@$S(Q),\mathcal{S}(Q)$}
	\]
	We also put
	$\Delta_Q=\prod_{v\in Q}\Delta_v$.

	\begin{proposition}\label{existence of Twylor-Wiles primes}
		We assume
		$\overline{r}$ is
		absolutely irreducible,
		$\overline{r}(\Gamma_{F(\zeta_p)})$
		is adequate, $p\nmid[F^\circ:F]$,
		and for each
		$v\in S$,
		\[
		\mathrm{dim}_\kappa
		L_v
		-
		\mathrm{dim}_\kappa
		H^0(\Gamma_{F_v},
		\widehat{\mathfrak{g}}(\kappa))
		=
		\begin{cases*}
			a_v,
			&
			if $v |p$;
			\\
			0
			&
			otherwise.
		\end{cases*}
		\]
		Here $a_v$ is an integer depending only on $p$.
		Fix a positive integer
		$q_0$ and write
		\[
		q=\max
		\left(
		q_0,
		\mathrm{dim}_\kappa
		H^1_{\mathcal{S}^\perp}
		\left(
		\Gamma_{F_v},
		\widehat{\mathfrak{g}}(\kappa)[1]
		\right)
		\right).
		\]
		Then there is a finite set $Q$ as above with $\#  Q=q$ and the relative dimension of $R_{\mathcal{S}(Q)}$ over $\mathcal{O}$ is
		\[
		\leq
		\#  Q
		+
		\sum_{v |p}a_v
		-\sum_{v |\infty}
		\mathrm{dim}_\kappa
		H^0(\Gamma_{F_v},
		\widehat{\mathfrak{g}}(\kappa)).
		\]
		
	\end{proposition}
	\begin{proof}
		Since $\overline{r}(\Gamma_{F(\zeta_p)})$ is adequate, Definition \ref{big}(1) and (3) imply that $H^0(\Gamma_{F,S(Q)},\mathfrak{g}[1])=0$ for any finite set $Q$.

		By Corollaries
		\ref{relative dimension of deformation rings}
		and
		\ref{dimension of H^1},
		it is enough to show that there is such a set $Q$ such that
		\[
		H^1_{\mathcal{S}(Q)^\perp}
		(\Gamma_{F,S(Q)},
		\widehat{\mathfrak{g}}(\kappa)[1])=0.
		\]

		Recall the inflation-restriction exact sequence 
		\[
		0
		\rightarrow
		H^1(\Gamma_{F,S(Q)},
		\widehat{\mathfrak{g}}(\kappa)[1])
		\rightarrow
		H^1(\Gamma_{F,S},
		\widehat{\mathfrak{g}}(\kappa)[1])
		\rightarrow
		\oplus_{v\in Q}
		H^1(I_{F_v},
		\widehat{\mathfrak{g}}(\kappa)[1])^{\Gamma_{F_v}}
		\]
		We have the following observation:
		let $v$ be a finite place of $F$ over a prime $\ell$
		such that
		$q_v\equiv1(\mathrm{mod}\,p)$.
		For a finite unramified
		$\kappa[\Gamma_{F_v}]$-module $M$,
		it is easy to see that
		\[
		H^1(I_{F_v},M[1])
		=
		H^1(I_{F_v}^p,M[1])
		\]
		where
		$I_{F_v}^p$ is the maximal
		$p$-profinite quotient of $I_{F_v}$.
		Note that $\Gamma_{F_v}$
		acts on
		$I_{F_v}^p$ by the cyclotomic character
		$\epsilon_v$.
		An element
		$\phi\in
		H^1(I_{F_v}^p,M[1])^{\Gamma_{F_v}}
		$
		satisfies
		$(\sigma_v\phi)(\gamma)=\phi(\gamma)$
		for all $\gamma\in I_{F_v}^p$.
		In other words,
		$\sigma_v(\phi(\sigma_v^{-1}(\gamma)))=\phi(\gamma)$.
		Changing $\sigma_v^{-1}\gamma$ to $\gamma$,
		we have
		\[
		\sigma_v(\phi(\gamma))
		=
		\phi(\sigma_v(\gamma))
		=
		\phi(\gamma^{\mathbf{N}(v)}).
		\]
		Since
		$\phi(\gamma^{\mathbf{N}(v)})
		=
		(1+\epsilon_v(\gamma)+\cdots+\epsilon_v(\gamma)^{\mathbf{N}(v)})
		\phi(\gamma)
		$
		and
		$p|(q_v-1)$,
		we have
		$\sigma_v(\phi(\gamma))
		=\phi(\gamma)$
		for all $\gamma\in I_{F_v}$.
		Therefore
		$H^1(I_{F_v},M[1])^{\Gamma_{F_v}}
		\subset
		H^0(\Gamma_{F_v},M)
		=
		M^{\Gamma_{F_v}}$.

		Now return to our proposition.
		Note that $\overline{r}$ is unramified at $v\in Q$
		and
		$\overline{r}_v$ has a decomposition
		$\overline{\psi}_v\oplus\overline{s}_v$,
		we deduce that
		\begin{align*}
			H^1(I_{F_v},
			\mathrm{Hom}_\kappa
			(\overline{\psi}_v,\overline{s}_v))^{\Gamma_{F_v}}
			&
			\subset
			\mathrm{Hom}_{\kappa[\Gamma_{F_v}]}
			(\overline{\psi}_v,\overline{s}_v),
			\\
			H^1(I_{F_v},
			\mathrm{Hom}_\kappa
			(\overline{s}_v,\overline{\psi}_v))^{\Gamma_{F_v}}
			&
			\subset
			\mathrm{Hom}_{\kappa[\Gamma_{F_v}]}
			(\overline{s}_v,\overline{\psi}_v).
		\end{align*}
		Both spaces on the RHS are zero by our assumptions on $\overline{\psi}_v$ and $\overline{s}_v$.
		For $g\in\Gamma_{F_v}$
		and an eigenvalue $a$ of $\overline{r}(g)$,
		write
		$\pi_{g,a}$
		for the projection of
		$\overline{r}$
		to the eigenspace of $\overline{r}(g)$
		corresponding to $a$ and
		$i_{g,a}$
		for the injection of the
		$a$-eigenspace into $\overline{r}$.
		We have an exact sequence
		\[
		0
		\rightarrow
		H^1
		(\Gamma_{F,S(Q)},
		\widehat{\mathfrak{g}}(\kappa)
		)
		\rightarrow
		H^1
		(\Gamma_{F,S},
		\widehat{\mathfrak{g}}(\kappa)
		)
		\rightarrow
		\bigoplus_{v\in Q}
		\left(
		H^1(I_{F_v},
		\mathrm{Ad}\overline{\psi}_v)^{\Gamma_{F_v}}
		\oplus
		H^1(I_{F_v},
		\mathrm{Ad}\overline{s}_v)^{\Gamma_{F_v}}
		\right)
		\]
		and by the properties of
		Taylor-Wiles deformations
		(\emph{cf.}
		Definition \ref{Taylor-Wiles deformations-Thorne}),
		we get
		the exact sequence
		\[
		0
		\rightarrow
		H^1_{\mathcal{S}(Q)^\perp,T}
		\left(
		\Gamma_{F,S(Q)},
		\widehat{\mathfrak{g}}(\kappa)[1]
		\right)
		\rightarrow
		H^1_{\mathcal{S}^\perp,T}
		\left(
		\Gamma_{F,S},
		\widehat{\mathfrak{g}}(\kappa)[1]
		\right)
		\xrightarrow{\alpha}
		\bigoplus_{v\in Q}
		H^1
		\left(
		\Gamma_{F_v},
		\mathrm{Ad}\overline{\psi}_v[1]
		\right)
		/L_v^\perp
		=
		\bigoplus_{v\in Q}\kappa.
		\]
		Here
		the collection of maps
		$\alpha=(\alpha_v)_{v\in Q}$
		is given by
		\[
		\alpha_v(\phi)
		=
		\mathrm{Tr}
		\left(
		\pi_{\sigma_v,\overline{\psi}_v(\sigma_v)}
		\circ
		\phi(\sigma_v)
		\circ
		i_{\sigma_v,\overline{\psi}_v(\sigma_v)}
		\right)	
		\in
		\kappa.
		\]

		By Cebotarev density theorem,
		it suffices to show that
		for any non-zero
		$\phi\in
		H^1(\Gamma_{F,S},
		\widehat{\mathfrak{g}}(\kappa)[1])$,
		there is an element
		$\gamma\in
		\Gamma_{F^\circ(\zeta_p)}$
		such that
		$\overline{r}(\gamma)$
		has an eigenvalue $a$
		and
		$
		\mathrm{Tr}
		\left(
		\pi_{\gamma,a}
		\circ
		\phi(\gamma)
		\circ
		i_{\gamma,a}
		\right)
		\neq0$.
		Write
		$L/F^\circ(\zeta_p)$
		for the finite  Galois extension
		cut out by
		$\widehat{\mathfrak{g}}(\kappa)$
		(the minimal quotient of
		$\Gamma_{F^\circ(\zeta_p)}$ through which
		$\mathrm{Ad}\,\overline{r}$
		factors).
		Since
		$\overline{r}$
		is trivial over
		$\Gamma_L$,
		for any $\gamma'\in\Gamma_L$,
		$\phi(\gamma'\gamma)
		=
		\phi(\gamma')+\phi(\gamma)
		$.
		So it suffices to show that
		there is an element
		$\gamma\in
		\Gamma_{L/F^\circ(\zeta_p)}$
		such that
		$\overline{r}(\gamma)$
		has an eigenvalue $a$
		and
		$\mathrm{Tr}
		\left(
		\pi_{\gamma,a}
		\circ
		\phi(\Gamma_L)
		\circ
		i_{\gamma,a}
		\right)
		\neq0$.
		By assumption,
		$\overline{r}(\Gamma_{F(\zeta_p)})$
		is adequate and $p\nmid[F^\circ:F]$,
		thus
		\[
		H^1
		(\Gamma_{L/F^\circ(\zeta_p)},
		\widehat{\mathfrak{g}}(\kappa))
		=0.
		\]
		Note that the restriction of the cyclotomic character
		$\epsilon_p$
		to
		$\Gamma_{L/F^\circ(\zeta_p)}$
		is trivial,
		we have therefore
		$H^1
		(\Gamma_{L/F^\circ(\zeta_p)},
		\widehat{\mathfrak{g}}(\kappa)[1])
		=0$.
		By the
		inflation-restriction exact sequence,
		we get an injective map
		\[
		H^1
		(\Gamma_{F^\circ(\zeta_p)},
		\widehat{\mathfrak{g}}(\kappa)[1])
		\hookrightarrow
		H^1(\Gamma_L,
		\widehat{\mathfrak{g}}(\kappa)[1])^{\Gamma_{L/F^\circ(\zeta_p)}}.
		\]
		Since
		$\phi\neq0$, $\phi(\Gamma_L)\neq0$.	Moreover $L/F^\circ(\zeta_p)$	is a Galois extension,
		$\phi(\Gamma_L)$
		is stable under the action of
		$\Gamma_{F^\circ(\zeta_p)}$
		and thus a
		$\kappa[\Gamma_{F^\circ(\zeta_p)}]$-submodule of
		$\widehat{\mathfrak{g}}(\kappa)$.
		Now again by our assumption that
		$\overline{r}(\Gamma_{F(\zeta_p)})$
		is adequate and $p\nmid[F^\circ:F]$ (thus $p\nmid[F^\circ(\zeta_p):F(\zeta_p)]$), we have
		\[
		H^1(\Gamma_{F^\circ(\zeta_p)},\widehat{\mathfrak{g}}(\kappa)[1])=0.
		\]
		We deduce that one can indeed find such a
		$\gamma\in\Gamma_{F^\circ(\zeta_p)}$
		satisfying the above conditions.
		This shows
		$H^1_{\mathcal{S}(Q)^\perp}
		(\Gamma_{F,S(Q)},
		\widehat{\mathfrak{g}}(\kappa)[1])=0$.

	\end{proof}

	\section{Minimal modularity lifting theorem}\label{modularity lifting}	
	In this section, we prove the $R=\mathbb{T}$ theorem (Theorem \ref{R=T}) and deduce the minimal modularity lifting theorem (Theorem \ref{modularity}).

	We retain the notations from
	§\ref{Automorphic forms on G(A_F)}, in particular, $W$ an irreducible algebraic representation of
	$ G$ over
	$\mathcal{O}$.
	Moreover, we assume in this section $K=G(\widehat{\mathcal{O}_F})$.
	Note that $K$ is not necessarily neat.
	We fix a finite set $Q$
	of finite places of $F$ disjoint from $S_p$
	($Q$ may be empty).
	For each $v\in Q$,
	we require that
	\begin{enumerate}
		\item 
		The group $G$ is split at $v$ and $q_v\equiv1(\mathrm{mod}\,p)$;

		\item 
		We fix a subset
		$\Omega^v\subset
		\{0,1,\cdots,n_\mathrm{s}-1\}$
		and write
		$\mathfrak{p}_{\Omega^v}
		\subset
		G(\mathcal{O}_{F,v})$
		for the parahoric subgroup associated to $\Omega^v$;

		\item 
		For each integer
		$j^v=1,2,\cdots,q_{\Omega^v}$,
		we write
		$\mathfrak{p}_{\Omega^v,j^v}
		\subset
		\mathfrak{p}_{\Omega^v}$
		for the normal subgroup as given preceding
		Proposition
		\ref{parahoric invariant subspace is one-dimensional}.
		We then fix two distinct integers
		$j_0^v,j_1^v\in\{1,\cdots,q_{\Omega^v}\}$
		as in
		Proposition
		\ref{projector induces an isomorphism}.
	\end{enumerate}
	We put
	\begin{equation}\label{T_W(K,R)}
		\mathbb{T}_{W}(K,R)
		\subset
		\mathrm{End}_R
		(M_{W,p}(K,R))
		\index{T@$\mathbb{T}_W(K,R)$}
	\end{equation}
	for the $R$-subalgebra
	generated by the spherical
	Hecke operators
	$T_v^{(j)}$
	as in
	(\ref{spherical Hecke operators})
	for
	$j=1,\cdots,n_\mathrm{s}$
	and
	finite places $v\notin S_p$.
	We associate to $K$ two compact open subgroups
	$K_1(Q),K_0(Q)$ of
	$G(\mathbb{A}_{F,f})$
	to
	$K$ as follows:
	\begin{enumerate}
		\item 
		$K_1(Q)=\prod_{v\nmid\infty}K_{1,v}(Q)$
		with
		$K_{1,v}(Q)=K_v$
		for $v\notin Q$
		and
		$=\mathfrak{p}_{\Omega^v,j_0^v}$
		for $v\in Q$;\index{K@$K_1(Q),K_0(Q)$}

		\item 
		$K_0(Q)=\prod_{v\nmid\infty}
		K_{0,v}(Q)$
		with
		$K_{0,v}(Q)=K_v$
		for $v\notin Q$
		and
		$=\mathfrak{p}_{\Omega^v}$
		for
		$v\in Q$.
	\end{enumerate}
	So we have $K_1(Q)\subset K_0(Q)$.\footnote{The group $K_1(Q)$, resp. $K_0(Q)$ should be thought as analogue of the congruence subgroups $\Gamma_1(N)$, resp. $\Gamma_0(N)$ in $\mathrm{SL}_2(\mathbb{Z})$.}
	Note that for $Q=\emptyset$, we have $K_1(Q)=K_0(Q)=K$.
	We write
	$\mathrm{pr}_{\Omega^v,(j_1^v)}$
	for the operator as in the end of
	Proposition
	\ref{projector induces an isomorphism}. Note that these operators $\mathrm{pr}_{\Omega^v,(j_1^v)}$ commute with each other.
	Then we set
	\[
	M_{W,p}^Q(K,\mathcal{O})
	=
	\left(
	\prod_{v\in Q}
	\mathrm{pr}_{\Omega^v,(j_1^v)}
	\right)
	M_{W,p}(K,\mathcal{O})
	\subset
	M_{W,p}(K,\mathcal{O}).
	\]
	So
	$M^Q_{W,p}(K,\mathcal{O})$
	is an $\mathcal{O}$-direct summand of
	$M_{W,p}(W,\mathcal{O})$
	which is invariant under the action of
	$\mathbb{T}_W(K,\mathcal{O})$.
	Write
	$\mathbb{T}_W^{Q}(K,\mathcal{O})$
	for the image of
	$\mathbb{T}_W(K,\mathcal{O})$
	in
	$\mathrm{End}_\mathcal{O}
	(M_{W,p}^Q(K,\mathcal{O}))$.
	Thus we have natural maps
	\begin{equation}\label{morphisms between Hecke algebras}
		\mathbb{T}_W^{Q}(K_1(Q),\mathcal{O})
		\rightarrow
		\mathbb{T}_W^{Q}(K_0(Q),\mathcal{O})
		\rightarrow
		\mathbb{T}_W^{Q}(K,\mathcal{O})
		\rightarrow
		\mathbb{T}_W(K,\mathcal{O}),
	\end{equation}
	the first two are surjective and the last one is injective. Moreover, any maximal ideal $\mathfrak{m}$ of
	$\mathbb{T}_W(K,\mathcal{O})$
	gives rise to maximal ideals
	$\mathfrak{m}_1^Q$,
	$\mathfrak{m}_0^Q$ and $\mathfrak{m}^Q$
	of the first three Hecke algebras (if $Q=\emptyset$, we simply write $\mathfrak{m}_?$ for $\mathfrak{m}_?^\emptyset$).
	Then  we have
	
	\begin{proposition}\label{residual Galois rep associated to auto rep}
		Let
		$\mathfrak{m}_1^Q$
		be a maximal ideal of
		$\mathbb{T}^{Q}_W(K_1(Q),\mathcal{O})$ as above satisfying the following: there are a minimal prime ideal $\mathfrak{P}$ of $\mathbb{T}^{Q}_W(K_1(Q),\mathcal{O})$ contained in $\mathfrak{m}_1^Q$ and an irreducible discrete series constituent $\pi$ of $\mathcal{A}(G(\mathbb{A}_F))$ satisfying condition (C2) which is generated by a non-zero vector in $M_{W,p}^Q(K_1(Q),\mathcal{O})$ such that $\mathbb{T}_W^{Q}(K_1(Q),\mathcal{O})$ acts on $\pi$ via the quotient $\mathbb{T}_W^{Q}(K_1(Q),\mathcal{O})/\mathfrak{P}$.
		Then there is a semisimple Galois representation
		\[
		\overline{\rho}_{\mathfrak{m}_1^Q}
		\colon
		\Gamma_F
		\rightarrow
		\,^CG
		\left(
		\mathbb{T}^{Q}_W(K_1(Q),\mathcal{O})/\mathfrak{m}_1^Q
		\right)
		\]
		satisfying the following properties
		\begin{enumerate}
			\item 
			The representation $\overline{\rho}_{\mathfrak{m}_1^Q}$ is unramified outside $S_p\sqcup Q$.

			\item 
			For a finite place 
			$v\notin S_p\sqcup Q$ of $F$, write
			the characteristic polynomial
			of
			$\overline{\rho}_{\mathfrak{m}_1^Q}(\sigma_v)$
			of the Frobenius $\sigma_v$
			as
			\[
			\begin{cases*}
				(X-1)^{\lfloor\frac{n}{2}\rfloor-n_\mathrm{s}}
				\prod_{i=1}^{n_\mathrm{s}}
				(X-x_i)
				(X-x_i^{-1}),
				&
				if $G^\ast=\mathrm{SO}_n^\eta$;
				\\
				(X-1)^{1+\lfloor\frac{n}{2}\rfloor-n_\mathrm{s}}
				\prod_{i=1}^{n_\mathrm{s}}
				(X-x_i)
				(X-x_i^{-1}),
				&
				if $G^\ast=\mathrm{Sp}_n$
			\end{cases*}
			\]
			with
			$x_i$ in a fixed algebraic closure of
			$\mathbb{T}^{Q}_W(K_1(Q),
			\mathcal{O})/\mathfrak{m}_1^Q$ (cf. Lemma \ref{Hecke polynomial for unramified rep}),
			then
			the Hecke operator
			$T_v^{(j)}(\mathrm{mod}\,\mathfrak{m}_1^Q)$
			is equal to
			\[
			\begin{cases*}
				q_v^{\frac{n_\mathrm{s}(n_\mathrm{s}-1)}{2}
					+
					(n-n_\mathrm{s}-j)j}
				\sum_{J}
				\sum_{I\subset J}
				\prod_{i\in I}
				(x_iq_v^{-1})
				\prod_{i\in J\backslash I}
				(x_iq_v^{-1})^{-1},
				&
				if $G^\ast=\mathrm{SO}_n^\eta$;
				\\
				q_v^{\frac{n_\mathrm{s}(n_\mathrm{s}+1)}{2}
					+
					(n-n_\mathrm{s}-j)j}
				\sum_{J}
				\sum_{I\subset J}
				\prod_{i\in I}
				(x_iq_v^{-1})
				\prod_{i\in J\backslash I}
				(x_iq_v^{-1})^{-1},
				&
				if $G^\ast=\mathrm{Sp}_n$,
			\end{cases*}
			\]
			where
			$J$ runs through all subsets of
			$\{1,2,\cdots,n_\mathrm{s}\}$
			of cardinal $j$ and
			$I$ runs through all subsets of $J$
			($j=1,\cdots,n_\mathrm{s}$).

			\item 
			For each $v|p$, $\overline{\rho}_{\mathfrak{m}_1^Q}|_{\Gamma_{F_v}}$ is \emph{torsion crystalline} (that is, there is a crystalline representation $V$ of $\Gamma_{F_v}$ and $\Gamma_{F_v}$-stable lattices $\Lambda\subset\Lambda'$ in $V$ such that $\Lambda'/\Lambda$ is isomorphic to $\overline{\rho}_{\mathfrak{m}_1^Q}|_{\Gamma_{F_v}}$; \emph{cf.} the beginning of §\ref{Fontaine-Lafaille deformations}).

		\end{enumerate}
	\end{proposition}
	\begin{proof}
		The proof is identical to
		that of
		\cite[Proposition 3.4.2]{ClozelHarrisTaylor2008}
		taking into account of
		Theorem
		\ref{Galois representations attached to auto rep}. Note that in our situation, $\overline{\rho}_{\mathfrak{m}_1^Q}$ is not necessarily unique.
	\end{proof}

	We call
	$\mathfrak{m}_1^Q$
	as in the above proposition
	\emph{non-Eisenstein}
	if
	$\overline{\rho}_{\mathfrak{m}_1^Q}$
	is absolutely irreducible
	(\emph{i.e.}
	the composition of
	$\overline{\rho}_{\mathfrak{m}_1^Q}$
	with the inclusion
	$^CG
	\hookrightarrow
	\mathrm{GL}_N$
	is absolutely irreducible).
	Then 
	we have

	\begin{theorem}\label{Galois representation valued in Hecke algebra-char=0}
		Let
		$\mathfrak{m}_1^Q$
		be as in the preceding proposition which is moreover
		a non-Eisenstein ideal of
		$\mathbb{T}_W^{Q}(K_1(Q),\mathcal{O})$. 
		Then
		$\overline{\rho}_{\mathfrak{m}_1^Q}$
		as in the preceding
		Proposition
		lifts to a Galois representation
		\[
		\rho_{\mathfrak{m}_1^Q}
		\colon
		\Gamma_F
		\rightarrow
		\,^CG
		\left(
		\mathbb{T}_W^{Q}(K_1(Q),\mathcal{O})_{\mathfrak{m}_1^Q}
		\right)
		\]
		satisfying the following properties
		\begin{enumerate}
			\item 
			$\rho_{\mathfrak{m}_1^Q}$
			is unramified at all but finitely many places of $F$.

			\item 
			For any finite place
			$v\notin S_p\sqcup Q$ of $F$,
			$\rho_{\mathfrak{m}_1^Q}$
			is unramified at $v$.
			Write the characteristic polynomial
			of
			$\rho_{\mathfrak{m}_1^Q}(\sigma_v)$
			of the Frobenius $\sigma_v$
			as
			\[
			\begin{cases*}
				(X-1)^{\lfloor\frac{n}{2}\rfloor-n_\mathrm{s}}
				\prod_{i=1}^{n_\mathrm{s}}
				(X-x_i)(X-x_i^{-1}),
				&
				if $G^\ast=\mathrm{SO}_n^\eta$;
				\\
				(X-1)^{1+\lfloor\frac{n}{2}\rfloor-n_\mathrm{s}}
				\prod_{i=1}^{n_\mathrm{s}}
				(X-x_i)(X-x_i^{-1}),
				&
				if $G^\ast=\mathrm{Sp}_n$,
			\end{cases*}
			\]
			with
			$x_i$ in a fixed sufficiently large extension of
			$\mathbb{T}_W^{Q}
			(K_1(Q),\mathcal{O})_{\mathfrak{m}_1^Q}$,
			then
			$T_v^{(j)}$
			is equal to
			\[
			\begin{cases*}
				q_v^{\frac{n_\mathrm{s}(n_\mathrm{s}-1)}{2}
					+
					(n-n_\mathrm{s}-j)j}
				\sum_{I,J}
				\prod_{i\in I}
				(x_iq_v^{-1})
				\prod_{i\in J\backslash I}
				(x_iq_v^{-1})^{-1},
				&
				if $G^\ast=\mathrm{SO}_n^\eta$;
				\\
				q_v^{\frac{n_\mathrm{s}(n_\mathrm{s}+1)}{2}
					+
					(n-n_\mathrm{s}-j)j}
				\sum_{I,J}
				\prod_{i\in I}
				(x_iq_v^{-1})
				\prod_{i\in J\backslash I}
				(x_iq_v^{-1})^{-1},
				&
				if $G^\ast=\mathrm{Sp}_n$,
			\end{cases*}
			\]
			where
			$J$ runs through all subsets of
			$\{1,2,\cdots,n_\mathrm{s}\}$
			of cardinal $j$ and
			$I$ runs through all subsets of $J$
			($j=1,\cdots,n_\mathrm{s}$).

			\item 
			For each
			$v\in Q$,
			suppose that
			$\mathrm{Art}_{F_v}(\omega_v)=\sigma_v$
			on the maximal abelian extension of
			$F_v$, then  		
			there is a character
			\[
			V^{(v)}
			\colon
			F_v^\times
			\rightarrow
			\mathbb{T}^{Q}_W
			(K_1(Q),\mathcal{O})_{\mathfrak{m}_1^Q}^\times
			\]		
			such that
			$V^{(v)}(\alpha)=V_\alpha^{j_0^v}$
			for any
			$\alpha\in\mathcal{O}_{F,v}^\times$
			and
			the restriction to the decomposition group
			$\Gamma_{F_v}$ at $v\in Q$
			is of the form
			\[
			\rho_{\mathfrak{m}_1^Q}
			|_{\Gamma_{F_v}}
			=
			s_v
			\oplus
			\psi_v
			\]
			where
			$s_v$ and
			$\psi_v$ are
			as in
			Theorem
			\ref{decomposition of the Galois representation}
			and moreover, $s_v$ is unramified and
			$\psi_v|_{I_{F_v}}$
			acts by the scalar character
			$
			\left(
			V^{(v)}\circ\mathrm{Art}_{F_v}^{-1}
			\right)
			|_{I_{F_v}}$.
		\end{enumerate}
	\end{theorem}
	\begin{proof}		
		For any minimal prime ideal
		$\mathfrak{P}
		\subset
		\mathfrak{m}_1^Q$, we fix an embedding of the algebraic closure of the fraction field of $\mathbb{T}_W^{Q}(K_1(Q),\mathcal{O})/\mathfrak{P}$ into $\overline{\mathbb{Q}}_p$.
		By
		Theorem
		\ref{Galois representations attached to auto rep},
		we have a Galois representation
		$\rho_{\mathfrak{P}}
		\colon
		\Gamma_F
		\rightarrow
		\,^C
		G(\overline{\mathbb{Q}}_p)$
		such that
		the first two properties in the theorem
		are satisfied
		(with $\mathbb{T}_W^{Q}(K_1(Q),\mathcal{O})_{\mathfrak{m}_1^Q}$
		replaced by
		$\overline{\mathbb{Q}}_p$).
		We may assume that
		$\rho_{\mathfrak{P}}$
		factors through
		$^CG(\mathcal{O}_{\mathfrak{P}})$
		where
		$\mathcal{O}_{\mathfrak{P}}$
		is the ring of integers of a
		sufficiently large finite
		extension of
		the fraction field of
		$\mathbb{T}_W^{Q}(K_1(Q),\mathcal{O})/\mathfrak{P}$
		and moreover
		$\rho_{\mathfrak{P}}  	
		\equiv
		\overline{\rho}_{\mathfrak{m}_1^Q}
		(\mathrm{mod}\,
		\mathfrak{m}_{\mathcal{O}_\mathfrak{P}})$
		(using the assumption that
		$\overline{\rho}_{\mathfrak{m}_1^Q}$
		is absolutely irreducible).
		
		Let
		$\kappa_1^Q
		=
		\mathbb{T}_W^{Q}(K_1(Q),\mathcal{O})/
		\mathfrak{m}_1^Q$
		and write
		$R$
		for the subring of
		$\kappa_1^Q
		\oplus
		\left(
		\bigoplus_{\mathfrak{P}
			\subset\mathfrak{m}_1^Q}
		\mathcal{O}_{\mathfrak{P}}
		\right)$
		consisting of elements
		$(a_{\mathfrak{m}_1^Q},a_\mathfrak{P})$
		such that
		$a_\mathfrak{P}  	
		\equiv
		a_{\mathfrak{m}_1^Q}
		(\mathrm{mod}\,\mathfrak{m}_{\mathcal{O}_\mathfrak{P}})$
		for all minimal prime ideals $\mathfrak{P}$
		contained in
		$\mathfrak{m}_1^Q$.
		Thus we get a Galois representation
		\[
		\rho_R
		\colon
		\Gamma_F
		\rightarrow
		\,
		^CG(R).
		\]
		By
		Remark
		\ref{equivalence of various spaces of modular forms},
		the
		Hecke algebra
		$\mathbb{T}_W^{Q}(K_1(Q),\mathcal{O})$
		is reduced and thus
		there is
		an injective morphism
		\[
		\mathbb{T}_W^{Q}(K_1(Q),\mathcal{O})_{\mathfrak{m}_1^Q}
		\hookrightarrow
		R.
		\]
		For a finite place
		$v\nmid2p$ of $F$,
		the Hecke algebra
		$\mathbb{T}_v$
		is generated by
		$T_v^{(j)}$
		with
		$j=1,2,\cdots,n_{\mathrm{s},v}$.
		Here
		$n_{\mathrm{s},v}$
		is as in
		(\ref{n_s}).
		By the Satake
		isomorphism,
		$\mathbb{T}_v$
		is isomorphic to
		$\widetilde{\mathbb{T}}_v
		=
		\mathbb{Z}[\frac{1}{q_v}]
		[(Z_1)^{\pm1},\cdots,(Z_{n_{\mathrm{s},v}})^{\pm1}]^{W_v^\mathrm{s}}$.
		Note that the latter is the same as
		the algebra
		$\mathbb{Z}[\frac{1}{q_v}]
		[(q_vZ_1)^{\pm1},\cdots,(q_vZ_{n_{\mathrm{s},v}})^{\pm1}]^{W_v^\mathrm{s}}$.
		On the other hand,
		for any
		$\gamma\in
		\Gamma_{F_v}$,
		the characteristic polynomial
		$P(X)$
		of
		$\rho_R(\gamma)$
		with eigenvalues
		$\{Z_1^{\pm1},\cdots,Z_{n_{\mathrm{s},v}}^{\pm1},1\}$
		is an element in
		$\widetilde{\mathbb{T}}_v$,
		thus lies in the Hecke algebra
		$\mathbb{T}_v$
		via the Satake isomorphism.
		This shows in particular that
		\[
		\mathrm{Tr}\,\rho_R(\gamma)\in
		\mathbb{T}_W^{Q}(K_1(Q),\mathcal{O})_{\mathfrak{m}_1^Q},
		\]
		for any
		$\gamma\in\Gamma_{F_v}$
		and any
		$v$ as above.
		Since the set of such elements $\gamma$
		is dense in
		$\Gamma_F$,
		one sees that
		$\mathrm{Tr}\,\rho_R(\Gamma_F)$
		is contained in
		$\mathbb{T}_W^{Q}(K_1(Q),\mathcal{O})_{\mathfrak{m}_1^Q}$.  	
		Thus by the following
		Lemma
		\ref{lifting representations},
		we know that
		$\rho_R$
		is
		$^CG(R)$-conjugate
		to a Galois representation
		\[
		\rho_{\mathfrak{m}_1^Q}
		\colon
		\Gamma_F
		\rightarrow
		\,
		^CG
		\left(
		\mathbb{T}_W^{Q}(K_1(Q),\mathcal{O})_{\mathfrak{m}_1^Q}
		\right)
		\]
		satisfying (1) and (2) in the
		theorem.

		For the last point, we take an irreducible automorphic representation $\pi=\otimes_w\pi_w$ of $G(\mathbb{A}_{F,f})$ which lies in $L^2(G(F)\backslash G(\mathbb{A}_{F,f}))$ such that
		$\pi^{K_1(Q)}\ne0$. 
		Now we can apply
		Propositions
		\ref{parahoric invariant subspace is one-dimensional} to $\pi_v$ to conclude the proof.
	\end{proof}

	\begin{lemma}\label{lifting representations}
		Let
		$S\subset R$
		be two noetherian
		complete local rings
		with
		$\mathfrak{m}_R\cap S=\mathfrak{m}_S$
		and common residue field $\kappa$
		of characteristic
		$p\ne2$.
		Fix a profinite group
		$\Gamma$ and
		$\rho
		\colon
		\Gamma
		\rightarrow
		\,^CG(R)$
		is a continuous group morphism.
		Assume
		that
		\[
		\overline{\rho}:=\rho(\mathrm{mod}\,\mathfrak{m}_R)
		\]
		is absolutely irreducible
		(i.e. the composition of $\overline{\rho}$ with
		the inclusion
		$^CG(\kappa)\rightarrow\mathrm{GL}_N(\kappa)$
		is absolutely irreducible)
		and
		$\mathrm{Tr}(\rho(\Gamma))
		\subset S$.
		Then for any ideal
		$I$ of $R$
		such that
		$\rho(\mathrm{mod}\,I)$
		has image in
		$S/(I\cap S)$,
		there is a continuous morphism
		$\rho'
		\colon
		\Gamma
		\rightarrow
		\,^CG(S)$
		which is conjugate to
		$\rho$
		by an element
		$g\in\,^CG(R)$
		with
		$g\equiv1(\mathrm{mod}\,I)$.
	\end{lemma}
	\begin{proof}
		We treat the case $G^\ast=\mathrm{SO}_n^\eta$ with $n$ even or $G^\ast=\mathrm{Sp}_n$. The other case is similar.

		We first do some reductions.
		We denote the statement in the lemma by
		\[
		(R,S,\mathfrak{m}_R,\rho,I).
		\]
		For an ideal
		$J$ of $R$ contained in $I$,
		we claim that
		the statement
		$(R/J,S/(J\cap S),\mathfrak{m}_R/J,\rho(\mathrm{mod}\,J),I/J)$
		implies
		$(R,S,\mathfrak{m}_R,\rho,I)$. 
		Indeed,
		there are an element
		$g_1\in
		\,^CG(R/J)$
		with
		$g_1\equiv1(\mathrm{mod}\,I/J)$
		and a morphism
		$\rho_1
		\colon
		\Gamma
		\rightarrow
		\,^CG(S/(J\cap S))$
		such that
		$\rho_1(h)=g_1\rho(h)(\mathrm{mod}\,J)g_1^{-1}$
		for any
		$h\in\Gamma$.
		Taking any lift
		$g\in\,^CG(R)$ of
		$g_1$
		with
		$g\equiv1(\mathrm{mod}\,I)$,		
		$g\rho(h)g^{-1}
		(\mathrm{mod}\,J)
		\in
		\,^CG(S/(J\cap S))$
		and therefore
		$g\rho(h)g^{-1}
		\in
		\,^CG(S)$.
		Let
		$s$ be the maximal integer such that
		$I\subset\mathfrak{m}_R^s$
		and set
		$J=I\cap\mathfrak{m}_R^{s+1}\subset I$.
		Then one has
		$\mathfrak{m}_RI\subset J$.
		Thus,
		up to modulo
		$J$,
		we are reduced to proving the statement
		$(R,S,\mathfrak{m}_R,\rho,I)$
		with
		$\mathfrak{m}_RI=0$.
		Taking
		an ideal
		$J\subset R$
		such that
		$J\subsetneq I=J+(r)$,
		then again
		up to modulo $J$,
		we are reduced to showing the statement
		$(R,S,\mathfrak{m}_R,\rho,I)$
		with
		$\mathfrak{m}_RI=0$
		and
		$\mathrm{dim}_\kappa(I)=1$.
		Up to replacing
		$R$ by
		the subring
		consisting of elements
		$r\in R$
		with
		$r-s\in I$
		for some $s\in S$,
		we can assume that
		\[
		S/(I\cap S)=R/I.
		\]
		Therefore we can further assume
		$I\cap S=0$,
		in which case one has
		$R=S\oplus I$.
		Consider the two statements
		\[
		\mathbf{S}_1
		=
		(R,S,\mathfrak{m}_R,\rho,I),
		\quad
		\mathbf{S}_2
		=
		(R,S,\mathfrak{m}_R,\rho,I'=\mathfrak{m}_S+I).
		\]
		Clearly
		$\mathbf{S}_1$ implies
		$\mathbf{S}_2$.
		Conversely,
		assume that
		$\mathbf{S}_2$ is true,
		then
		there is an element
		$g\in\,^CG(R)$
		with
		$g\equiv1(\mathrm{mod}\,I')$
		such that
		$g\rho(h)g^{-1}
		\in
		\,^CG(S)$.
		Write
		$g_1=g(\mathrm{mod}\,I)
		\in
		\,^CG(S)\subset
		\,^CG(R)$.
		Then
		$g_2=g_1^{-1}g$
		satisfies
		$g_2\equiv1(\mathrm{mod}\,I)$
		and
		$g_2\rho(h)g_2^{-1}
		\in
		\,^CG(S)$,
		which gives
		$\mathbf{S}_1$.
		Up to modulo
		$\mathfrak{m}_SR$ in
		$\mathbf{S}_2$,
		we are reduced to showing
		the statement
		$(\kappa[\epsilon],\kappa,\kappa\epsilon,\rho,I=\kappa\epsilon)$.
		Here
		$\epsilon^2=0$.

		Recall
		$N=2\lfloor n/2\rfloor$.
		We extend
		$\rho$
		to a morphism
		of $\kappa$-algebras
		\[
		\rho^\circ
		\colon
		\kappa[\Gamma]
		\rightarrow
		\mathrm{M}_N(\kappa[\epsilon]),
		\]
		which is surjective since
		$\rho$
		is absolutely irreducible.
		Now for any
		$h\in\kappa[\Gamma]$
		and
		$h'\in\mathrm{Ker}(\rho^\circ)$, we have
		$\rho^\circ(h')\in
		\epsilon
		\mathrm{M}_N(\kappa)$. Moreover,write $\rho^\circ(h')=\epsilon X$ with $X\in \mathrm{M}_N(\kappa)$. Since $hh'\in\mathrm{Ker}(\rho^\circ)$,
		one has
		$\mathrm{Tr}\,\rho^\circ(h)X=0$.
		This shows that
		$\mathrm{Tr}\,\mathrm{M}_N(\kappa[\epsilon]X)=0$
		and therefore
		$\rho^\circ(h')=0$.
		This implies that
		$\rho^\circ$
		factors through
		\[
		\rho':=\rho^\circ(\mathrm{mod}\,I)
		\colon
		\kappa[\Gamma]
		\rightarrow
		\mathrm{M}_N(\kappa)
		\]
		(which is also surjective).
		In other words,
		there is a map
		$\phi
		\in\mathrm{End}_\kappa(\mathrm{M}_N(\kappa))$
		such that
		\[
		\rho^\circ(h)
		=
		\rho'(h)
		+
		\phi(\rho'(h))\epsilon,
		\quad
		h\in
		\kappa[\Gamma].
		\]
		From this we deduce that
		$\phi(ab)=a\phi(b)+\phi(a)b$
		for any
		$a,b\in\mathrm{M}_N(\kappa)$.
		Taking
		\[
		A=\sum_{j=1}^N\phi(E_{j,1})E_{1,j}
		\]
		with
		$E_{i,j}$ the elementary matrices,
		one verifies that
		$\phi(a)=Aa-aA$
		and moreover
		$(1_N-A\epsilon)
		\rho^\circ(h)
		(1_N+A\epsilon)=\rho'(h)$
		for any
		$h\in
		\kappa[\Gamma]$.
		Write
		$\Lambda\in
		\mathrm{Sym}_N(\kappa)$
		for the symmetric matrix
		associated to the quadratic form
		defining
		$^CG$
		and put
		$\Lambda'=
		(1+A\epsilon)^\mathrm{t}
		\Lambda(1+A\epsilon)$.
		Then one has
		\[
		\rho(h)^{-1}(\Lambda^{-1}\Lambda')\rho(h)
		=
		\Lambda^{-1}\Lambda',
		\quad
		\forall
		h\in\Gamma.
		\]
		The irreducibility of $\rho$
		shows that
		$\Lambda^{-1}\Lambda'$
		is a scalar matrix and thus
		$\Lambda'=(1+b\epsilon)\Lambda$
		for some
		$b\in\kappa$.
		Taking
		$A'=A-(b/2)1_N$
		(recall $\mathrm{char}(\kappa)\ne2$),
		one verifies that
		\[
		(1+A'\epsilon)^\mathrm{t}\Lambda(1+A'\epsilon)=\Lambda,
		\quad
		(1-A'\epsilon)\rho^\circ(h)(1+A'\epsilon)=\rho'(h).
		\]
		Thus
		$g=1+A'\epsilon\in
		\,^CG(\kappa[\epsilon])$
		is the desired element.
	\end{proof}

	Fix a maximal ideal
	$\mathfrak{m}$
	of
	$\mathbb{T}_W(K,\mathcal{O})$
	which gives rise to
	maximal ideals
	$\mathfrak{m}_1^Q$,
	resp.
	$\mathfrak{m}_0^Q$,
	$\mathfrak{m}^Q$,
	of
	$\mathbb{T}_W^{Q}(K_1(Q),\mathcal{O})$,
	resp.
	$\mathbb{T}_W^{Q}(K_0(Q),\mathcal{O})$,
	$\mathbb{T}_W^{Q}(K,\mathcal{O})$
	using the morphisms
	in
	(\ref{morphisms between Hecke algebras}).
	From Theorem \ref{Galois representation valued in Hecke algebra-char=0}, we see immediately
	that in the following
	chain of morphisms,
	the first two are surjective
	and the last one is an isomorphism
	(using the relation between
	$T_v^{(j)}$ and the Frobenius elements):
	\[
	\mathbb{T}_W^{Q}(K_1(Q),\mathcal{O})_{\mathfrak{m}_1^Q}
	\twoheadrightarrow
	\mathbb{T}_W^{Q}
	(K_0(Q),\mathcal{O})_{\mathfrak{m}_0^Q}
	\twoheadrightarrow
	\mathbb{T}_W^{Q}(K,\mathcal{O})_{\mathfrak{m}}
	\simeq
	\mathbb{T}_W^{T}(K,\mathcal{O})_{\mathfrak{m}}.
	\]

	We assume in the following that $F$ is unramified at $p$ and for each place $v|p$ of $F$, the Galois representation $\overline{\rho}_{\mathfrak{m}_0^\emptyset}|_{\Gamma_{F_v}}$ satisfies (C3) (\emph{cf.} Remark \ref{remark on (C3) and (C6)}).
	We consider
	deformation problems of the following type for the residual Galois representation
	$\overline{\rho}_{\mathfrak{m}^Q_1}$
	as in
	Definition
	\ref{T-deformation or lifting of type S}:
	\begin{equation}\label{S type of deformation}
		\mathcal{S}
		=
		(\mathcal{D}_v)_{v\in S_p}
		\text{where each }
		\mathcal{D}_v=
		\text{Fontaine-Laffaille deformation
			(\emph{cf.} Definition \ref{crystalline local deformation problem})}.
	\end{equation}
	We then have
	the universal deformation ring
	$R_\mathcal{S}$
	for this deformation problem
	as well as the universal lifting
	\[
	r_\mathcal{S}
	\colon
	\Gamma_F
	\rightarrow
	\,^CG
	(R_\mathcal{S}).
	\]
	Similarly, we have $(R_{\mathcal{S}(Q)},r_{\mathcal{S}(Q)})$ for the deformation problem $\mathcal{S}(Q)$.

	Note that $\rho_{\mathfrak{m}^Q_1}|_{F_v}$ is a lifting of $\overline{\rho}_{\mathfrak{m}^Q_1}|_{F_v}$ of type $\mathcal{D}_v$ for $v|p$ or $v\in Q$. We assume in the following that
	$(
	\mathbb{T}_W(K,\mathcal{O})_{\mathfrak{m}},
	\rho_{\mathfrak{m}}
	)$
	is a lifting of $\overline{\rho}_{\mathfrak{m}}$ of type
	$\mathcal{S}$
	and thus there is a morphism of $\mathcal{O}$-algebras:
	\[
	R_\mathcal{S}
	\twoheadrightarrow
	\mathbb{T}_W(K,\mathcal{O})_{\mathfrak{m}}
	\]
	such that
	the push-forward of
	$r_\mathcal{S}$
	along this morphism gives
	$\rho_{\mathfrak{m}_1}$. This morphism is surjective due to Theorem \ref{Galois representation valued in Hecke algebra-char=0}(2).
	
	Now we have
	\begin{theorem}\label{R=T}
		We assume that $F$ is unramified at $p$.
		Let $\mathfrak{m}$ be a non-Eisenstein maximal ideal of $\mathbb{T}_{W}(K,\mathcal{O})$, $\mathfrak{P}$ a minimal prime ideal contained in $\mathfrak{m}$, $\pi$ an irreducible discrete series constituent of $\mathcal{A}(G(\mathbb{A}_F))$ satisfying condition (C2) which is generated by a non-zero vector in $M_{W,p}(K,\mathcal{O})$ such that $\mathbb{T}_W(K,\mathcal{O})$ acts on $\pi$ via the quotient $\mathbb{T}_W(K,\mathcal{O})/\mathfrak{P}$.
		Suppose that the following are satisfied
		\begin{enumerate}
			\item 
			$\overline{\rho}_{\mathfrak{m}_1^\emptyset}
			(\Gamma_{F(\zeta_p)})$
			is adequate,

			\item 
			$p\nmid[F^\circ:F]$,

			\item 
			for any $g\in G(\A_{F,f})$, $gG(F)g^{-1}\cap K$ contains no elements of order $p$,

			\item 
			for each place $v|p$ of $F$, $\overline{\rho}_{\mathfrak{m}_1^\emptyset}|_{\Gamma_{F_v}}$ satisfies (C3).
		\end{enumerate}
	    Then we have an isomorphism of \emph{local complete intersections}
		$\mathcal{O}$-algebras
		\[
		R_\mathcal{S}
		\simeq
		\mathbb{T}_W(K,\mathcal{O})_{\mathfrak{m}}.
		\]
	\end{theorem}
	\begin{proof}
		We apply
		\cite[Theorem 1.1]{Brochard2017}. We choose a set $Q$
		of
		Taylor-Wiles primes
		as in
		Proposition
		\ref{existence of Twylor-Wiles primes}
		such that
		$G$ is split at any $v\in Q$.
		Moreover, the Galois representation $\rho_{\mathfrak{m}_1}$ gives rise to
		a morphism
		\[
		R_{\mathcal{S}}
		\twoheadrightarrow
		\mathbb{T}_W
		(K_1,\mathcal{O})_{\mathfrak{m}_1},
		\]
		which is surjective since $\mathfrak{m}$ is non-Eisenstein. The same argument as in the proof of \cite[Theorem 3.5.1]{ClozelHarrisTaylor2008} shows that the map $R_{\mathcal{S}(Q)}\to\mathbb{T}_W^{T\sqcup Q}(K_0(Q),\mathcal{O})_{\mathfrak{m}_0^Q}$ is surjective.

		Note that
		$K_1(Q)$
		is a normal subgroup of
		$K_0(Q)$
		whose quotient is
		a $p$-power finite group
		(\emph{cf.} the paragraph preceding
		Proposition \ref{parahoric invariant subspace is one-dimensional})
		\[
		\kappa_Q^\times(p):=
		\prod_{v\in Q}
		\kappa_v^\times(p)
		=
		K_0(Q)/K_1(Q),
		\]
		and thus
		$M_{W,p}
		(K_1(Q),\mathcal{O})$
		is a finite free
		$\mathcal{O}[\kappa_Q^\times(p)]$-module.
		Moreover by
		Proposition
		\ref{compare invariants of different cpt open subgroups},
		$M_{W,p}^Q(K_1(Q),\mathcal{O})$
		is a free
		$\mathcal{O}[\kappa_Q^\times(p)]$-module.
		On the other hand,
		we have a morphism
		\[
		\psi_v(r_{\mathcal{S}(Q)})
		\colon
		\mathcal{O}[\Delta_v]
		\rightarrow
		R_{\mathcal{S}(Q)}.
		\]
		We view
		$\kappa_v^\times(p)$
		as a subgroup of
		$\mathcal{O}_{F,v}^\times$
		via the Teichmüller lifting. We deduce that the following composition
		(write the image of
		$\mathfrak{m}_1^Q
		\subset
		\mathbb{T}_W(K_1(Q),\mathcal{O})$
		in
		$\mathbb{T}_W^{Q}(K_1(Q),\mathcal{O})$
		again by $\mathfrak{m}_1^Q$)
		\[
		\kappa_v^\times(p)
		\rightarrow
		\mathcal{O}_{F,v}^\times
		\xrightarrow{\mathrm{Art}_{F_v}}
		\Gamma_{F_v}^\mathrm{ab}
		\rightarrow
		\Delta_v
		\xrightarrow{\psi_v(r_{\mathcal{S}(Q)})}
		(R_{\mathcal{S}(Q)})^\times
		\rightarrow
		\mathbb{T}_W^{Q}(K_1(Q),\mathcal{O})_{\mathfrak{m}_1^Q}^\times
		\]
		is just the (restriction to $\kappa_v^\times(p)$ of the) map
		$V^{(v)}
		\colon
		F_v^\times
		\rightarrow
		\mathbb{T}_W^{Q}(K_1(Q),\mathcal{O})_{\mathfrak{m}_1^Q}^\times$
		in
		Theorem
		\ref{Galois representation valued in Hecke algebra-char=0}.
		Therefore we obtain the following
		\begin{enumerate}
			\item 
			$R_{\mathcal{S}(Q)}$
			is an
			$\mathcal{O}[\kappa_Q^\times(p)]$-algebra
			and
			$M_{W,p}^Q(K_1(Q),\mathcal{O})_{\mathfrak{m}_1^Q}$
			is a finite free
			$\mathcal{O}[\kappa_Q^\times(p)]$-module,
			which is
			of finite type as a module over
			$R_{\mathcal{S}(Q)}$
			(via the map
			$R_{\mathcal{S}(Q)}
			\rightarrow
			\mathbb{T}_{W,p}^{Q}(K_1(Q),\mathcal{O})_{\mathfrak{m}_1^Q}$);

			\item 
			$\mathrm{dim}_\kappa
			\mathfrak{m}_{\mathcal{O}[\kappa_Q^\times(p)]}
			/
			\left(
			\mathfrak{m}_\mathcal{O},
			\mathfrak{m}_{\mathcal{O}[\kappa_Q^\times(p)]}^2
			\right)
			=
			\# (Q)
			\ge
			\mathrm{dim}_\kappa
			\mathfrak{m}_{R_{\mathcal{S}(Q)}}/
			\left(
			\mathfrak{m}_\mathcal{O},
			\mathfrak{m}_{R_{\mathcal{S}(Q)}}^2
			\right)
			$
			by
			Remark
			\ref{dimension for infinite H_v^0},
			Corollaries
			\ref{dimension of H^1}
			and
			\ref{crystalline deformation-dimension defaut},
			Proposition \ref{existence of Twylor-Wiles primes} (we take $a_v$ to be equal to $h_{v'}^0$ for any $v'|\infty$ as in Remark \ref{dimension for infinite H_v^0}).
		\end{enumerate}
		From this we apply \cite[Theorem 1.1]{Brochard2017}, so	$M_{W,p}^Q(K_1(Q),\mathcal{O})\otimes_\mathcal{O}\kappa$ is finite free over
		$R_{\mathcal{S}(Q)}\otimes_\mathcal{O}\kappa$ and the latter is of complete intersection over $\mathcal{O}[\kappa_Q^\times(p)]\otimes_\mathcal{O}\kappa$.
		We deduce that $R_{\mathcal{S}(Q)}\otimes\kappa$ is isomorphic to
		$\mathbb{T}_W^{Q}
		(K_1(Q),\mathcal{O})_{\mathfrak{m}_1^Q}
		\otimes\kappa$
		and this isomorphism lifts to an isomorphism of
		$\mathcal{O}[\kappa_Q^\times(p)]$-modules
		\[
		R_{\mathcal{S}(Q)}
		\simeq
		\mathbb{T}_W^{Q}(K_1(Q),\mathcal{O})_{\mathfrak{m}_1^Q}.
		\]
		Moreover, by Proposition \ref{compare invariants of different cpt open subgroups}
		and	Proposition	\ref{projector induces an isomorphism},	one has isomorphisms
		\[
		\left(
		M_{W,p}^Q(K_1(Q),\mathcal{O})_{\mathfrak{m}_1^Q}
		\right)_{\kappa_Q^\times(p)}
		\simeq
		M_{W,p}^Q
		(K_0(Q),\mathcal{O})_{\mathfrak{m}_0^Q}
		\simeq
		M_{W,p}(K,\mathcal{O})_{\mathfrak{m}}
		\]
		as $\mathcal{O}$-modules
		(the subscript
		$_{\kappa_Q^\times(p)}$
		means taking the coinvariants
		and $\mathfrak{m}_0^Q$
		is similarly defined as $\mathfrak{m}_1^Q$).
		Therefore
		we have an isomorphism
		\[
		\left(
		\mathbb{T}_W^{Q}(K_1(Q),\mathcal{O})_{\mathfrak{m}_1^Q}
		\right)_{\kappa_Q^\times(p)}
		\simeq
		\mathbb{T}_W(K,\mathcal{O})_{\mathfrak{m}}.
		\]
		By definition of the deformation problems
		$\mathcal{S}$ and $\mathcal{S}(Q)$, we have
		\[
		(R_{\mathcal{S}(Q)})_{\kappa_Q^\times(p)}
		=
		R_{\mathcal{S}}.
		\]
		From this we get the desired isomorphism
		\[
		R_\mathcal{S}
		\simeq
		\mathbb{T}_W(K,\mathcal{O})_\mathfrak{m}.
		\]
	\end{proof}

	From this we deduce immediately the following minimal
	modularity lifting theorem,
	the main result of this article (this is Theorem \ref{main theorem-2} in the introduction):
	\begin{theorem}\label{modularity}
		Let the notations and assumptions be as in Theorem \ref{R=T}. Let
		\[
		r\colon
		\Gamma_F
		\rightarrow
		\,^C
		G(\mathcal{O})
		\]
		be a Galois representation satisfying the following conditions ($\overline{r}$ is
		its reduction modulo $\mathfrak{p}$):
		\begin{enumerate}
			\item
			$r$ is unramified at all finite places of $F$ not over $p$.

			\item 
			$\overline{r}\simeq\overline{r}_{\pi}$ and $\mu_r=\mu_{r_\pi}$.
			
			\item 
			For all $v|p$, $r|_{\Gamma_{F_v}}$ lies in $\mathcal{D}^\mathrm{FL}_{v,\overline{r}_{\pi,v}}(\mathcal{O})$ (\emph{cf.} Definition \ref{crystalline local deformation problem})
			where $\overline{r}_{\pi,v}
			=\overline{r}_\pi|_{\Gamma_{F_v}}$.

		\end{enumerate}

		Then there exist a minimal prime ideal $\mathfrak{P}'$ of $\mathbb{T}_W(K,\mathcal{O})$ contained in $\mathfrak{m}$ and an irreducible discrete series constituent $\pi$ of $\mathcal{A}(G(\mathbb{A}_F))$ generated by a non-zero vector in  $M_{W,p}(K,\mathcal{O})$ such that the action of $\mathbb{T}_W(K,\mathcal{O})$ on $\pi'$ factors through $\mathbb{T}_W(K,\mathcal{O})/\mathfrak{P}'$ and that we have
		\[
		r\simeq
		r_{\pi'}.
		\]
	\end{theorem}
	\begin{proof}
		It is clear that $r$ is a lifting of $\overline{\rho}_{\mathfrak{m}_1^\emptyset}$ of type $\mathcal{S}$ (as defined in (\ref{S type of deformation})), this corresponds to a morphism of $\mathcal{O}$-algebras
		\[
		\mathbb{T}_W(K,\mathcal{O})_{\mathfrak{m}}
		\simeq
		R_\mathcal{S}
		\to
		\mathcal{O}.
		\]
		Then set $\mathfrak{P}'$ to be the kernel in $\mathbb{T}_W(K,\mathcal{O})_\mathfrak{m}$ of the above morphism and choose $\pi'$ as in the theorem. From this we deduce $r\simeq r_{\pi'}$.

	\end{proof}

	\section{Application to Bloch-Kato conjectures}\label{Application to Bloch-Kato conjectures}
	We apply the $R=\mathbb{T}$ result in §\ref{modularity lifting} to the Block-Kato conjecture for the adjoint Galois representation attached to $r_{\pi}\colon\Gamma_F\to\,^CG(\overline{\mathbb{Q}}_p)$ using congruence ideal formalism as in \cite[§8]{HidaTilouine2020}, this will prove Theorem \ref{Theorem 3 Bloch-Kato}.

	We recall the assumptions and notations used in the previous section:
	\begin{enumerate}
		\item 
		we fix an irreducible algebraic representation $W$ of $G$ over $\mathcal{O}$;

		\item 
		$K=G(\widehat{\mathcal{O}_F})$ and for any $g\in G(\A_{F,f})$, $gG(\Q)g^{-1}\cap K$ contains no elements of order $p$;

		\item 
		we assume that $F$ is unramified at $p$;

		\item 
		$\mathfrak{m}$ is a non-Eisenstein maximal ideal of the Hecke algebra $\mathbb{T}_W(K,\mathcal{O})$, $\mathfrak{P}$ a minimal prime ideal contained in $\mathfrak{m}$, $\pi$ an irreducible discrete series automorphic representation of $\mathcal{A}(G(\mathbb{A}_F))$ satisfying condition (C2) which is generated by a non-zero vector in $M_{W,p}(K,\mathcal{O})$ such that $\mathbb{T}_W(K,\mathcal{O})$ acts on $\pi$ via the quotient $\mathbb{T}_W(K,\mathcal{O})/\mathfrak{P}$;

		\item 
		$\overline{\rho}_{\mathfrak{m}}(\Gamma_{F(\zeta_p)})$ is adequate;

		\item 
		$p\nmid[F^\circ:F]$;

		\item 
		for each place $v|p$ of $F$, $\overline{\rho}_{\mathfrak{m}}|_{\Gamma_{F_v}}$ satisfies (C3).
	\end{enumerate}
	Then we have an isomorphism of $\mathcal{O}$-algebras of local complete intersections
	\[
	R_\mathcal{S}\simeq\mathbb{T}_W(K,\mathcal{O})_{\mathfrak{m}}.
	\]

	For ease of notations, we write $\mathbb{T}=\mathbb{T}_W(K,\mathcal{O})_\mathfrak{m}$ and $\rho=\rho_{\mathfrak{m}}=r_\mathcal{S}$. Then the action of $\mathbb{T}$ on $\pi$ gives rise to a morphism of $\mathcal{O}$-algebras
	\[
	\theta_\pi\colon\mathbb{T}\to\mathcal{O}.
	\]

	\subsection{Congruence and differential modules}

	We write $\Omega_{\mathbb{T}/\mathcal{O}}$ for the $\mathbb{T}$-module of Kähler differentials of $\mathbb{T}$ over $\mathcal{O}$ (\emph{cf.}\cite[§8.1]{HidaTilouine2020}). We then set
	\[
	C_1(\theta_\pi):=\Omega_{\mathbb{T}/\mathcal{O}}\otimes_\mathbb{T}\mathcal{O}.
	\]
	This is the \emph{differential module} of the morphism $\theta_\pi$.
	On the other hand, since $\mathbb{T}$ is reduced, the total quotient ring $\mathrm{Frac}(\mathbb{T})$ has a factorization into fields
	\[
	\mathrm{Frac}(\mathbb{T})=K\bigoplus X
	\]
	where $X$ is a product of fields. Write $\mathfrak{a}$ to be the kernel of the map $\mathbb{T}\to\mathrm{Frac}(\mathbb{T})\to X$. Then we set
	\[
	C_0(\theta_\pi)
	:=
	(\mathbb{T}/\mathfrak{a})\otimes_\mathbb{T}\mathcal{O},
	\]
	the \emph{congruence module} of the morphism $\theta_\pi$ (\emph{cf.} \cite[§8.2]{HidaTilouine2020}). By (8.4) of \emph{loc.cit}, both modules $C_0(\theta_\pi)$ and $C_1(\theta_\pi)$ are cyclic $\mathcal{O}$-modules. So there exist elements $c_0(\theta_\pi),c_1(\theta_\pi)\in\mathcal{O}$ such that as $\mathcal{O}$-modules, we have
	\[
	C_i(\theta_\pi)\simeq\mathcal{O}/c_i(\theta_\pi)\mathcal{O},
	\quad
	i=0,1.
	\]
	The element $c_0(\theta_\pi)$ is usually called the \emph{congruence number} of the morphism $\theta_\pi$.
	By a theorem of Tate, we have (\cite[Theorem 8.7]{HidaTilouine2020})
	\begin{proposition}\label{congruence ideal=differential ideal}
		The $\mathcal{O}$-modules $C_0(\theta_\pi)$ and $C_1(\theta_\pi)$ are torsion $\mathcal{O}$-modules of finite type and up to units in $\mathcal{O}$, we have the identity
		\[
		c_0(\theta_\pi)=c_1(\theta_\pi).
		\]
	\end{proposition}

	\subsection{Differential modules and Selmer groups}
	For an $\mathcal{O}$-module $M$, we write its Pontryagin dual as $M^\vee=\mathrm{Hom}_\mathcal{O}(M,K/\mathcal{O})$.
	For the Galois representation $r_\pi\colon\Gamma_F\to\,^CG(\mathcal{O})$, we
	write $\rho_\pi$ for the adjoint representation of $\Gamma_F$ on $\widehat{\mathfrak{g}}(\mathcal{O})$
	\[
	\rho_\pi:=\mathrm{Ad}^0r_\pi
	\colon
	\Gamma_F
	\to
	\mathrm{Aut}(\widehat{\mathfrak{g}}(\mathcal{O}))
	\]
	Then the discrete Bloch-Kato Selmer group associated to $\rho_\pi(E/\mathcal{O})=\rho_\pi\otimes_\mathcal{O}K/\mathcal{O}$ is defined as follows
	\[
	H^1_\mathrm{BK}(F,\rho_\pi(E/\mathcal{O}))
	:=
	\mathrm{Ker}
	\left(
	H^1(F,\rho_\pi(E/\mathcal{O}))
	\to
	\prod_{v\nmid\infty}
	\frac{H^1(F_v,\rho_\pi(E/\mathcal{O}))}{H^1_f(F_v,\rho_\pi(E/\mathcal{O}))}
	\right).
	\]
	Here $H_f(F_v,\rho_\pi(E/\mathcal{O}))$ is defined as the image under the natural map $\mathrm{pr}_v\colon H^1(F_v,\rho_\pi(E))\to H^1(F_v,\rho_\pi(E/\mathcal{O}))$ of the following subgroup of $H^1(F_v,\rho_\pi(E))$
	\[
	H_f^1(F_v,\rho_\pi(E))
	:=
	\begin{cases*}
		\mathrm{Ker}
		\left(
		H^1(F_v,\rho_\pi(E))\to
		H^1(I_v,\rho_\pi(E))
		\right),
		&
		$v\nmid p$;
		\\
		\mathrm{Ker}
		\left(
		H^1(F_v,\rho_\pi(E))\to
		H^1(F_v,\rho_\pi(B_\mathrm{crys}))
		\right),
		&
		$v\mid p$.
	\end{cases*}
	\]
	Here $I_v$ is the inertial subgroup of $\Gamma_{F_v}$ and $B_\mathrm{crys}$ is Fontaine's period ring.

	Following the proof of \cite[Proposition 3.4]{HidaTilouine2020}
	\begin{proposition}\label{Differential module vs Bloch-Kato Selmer group}
		We have an isomorphism of $\mathcal{O}$-modules
		\[
		H^1_\mathrm{BK}(F,\rho_\pi(E/\mathcal{O}))^\vee
		\simeq
		C_1(\theta_\pi).
		\]
		In particular, the Selmer group $H^1_\mathrm{BK}(F,\rho_\pi(E/\mathcal{O}))^\vee$ is a finitely generated torsion $\mathcal{O}$-module.
	\end{proposition}
	\begin{proof}
		Recall we have the following isomorphisms
		\[
		C_1(\theta_\pi)^\vee
		=
		\mathrm{Hom}_{\mathbb{T}}(\Omega_{\mathbb{T}/\mathcal{O}},K/\mathcal{O})
		=
		\mathrm{Der}_\mathcal{O}(\mathbb{T},K/\mathcal{O}).
		\]
		Here $\mathrm{Der}_\mathcal{O}(\mathbb{T},K/\mathcal{O})$ is the $\mathcal{O}$-module of $\mathcal{O}$-linear $K/\mathcal{O}$-valued \emph{continuous} derivations on $\mathbb{T}$. The action of $\mathbb{T}$ on $\mathcal{O}$ and $K$ comes from the morphism $\theta_\pi$. Then it follows that
		\begin{align*}
			\mathrm{Der}_\mathcal{O}(\mathbb{T},K/\mathcal{O})
			&
			=
			\mathrm{Hom}_\mathcal{O}(\mathbb{T},\mathbb{T}\oplus(E/\mathcal{O})\epsilon)
			=
			\mathrm{Hom}_\mathcal{O}(R_\mathcal{S},R_\mathcal{S}\oplus(E/\mathcal{O})\epsilon)
			\\
			&
			=
			\left\{
			r\in\mathrm{Def}_\mathcal{S}(\mathbb{T}\oplus(E/\mathcal{O})\epsilon)
			|
			r(\gamma)
			=
			(1+\epsilon c_r(\gamma))r_{\mathcal{S}}^\mathrm{univ}(\gamma)
			\right\}/\sim
		\end{align*}
		Here $\epsilon^2=0$.
		Then we have an injective map
		\[
		\Phi\colon
		\mathrm{Der}_\mathcal{O}(\mathbb{T},K/\mathcal{O})
		\to
		H^1(\Gamma,\rho_\pi(E/\mathcal{O})),
		\quad
		r\mapsto
		[c_r].
		\]
		Here $[c_r]$ is the equivalence class of the $1$-cocycle $c_r$.

		We claim that $[c_r]$ lies in $H^1_\mathrm{BK}(F,\rho_\pi(E/\mathcal{O}))$:
		\begin{enumerate}
			\item 
			For any $v\nmid p\infty$, the image of $[c_r]$ in $H^1(I_v,\rho_\pi)$ is zero because $r_\pi$ is unramified at $v$. Thus by \cite[Lemma 1.3.5(iv)]{Rubin2000}, the image of $[c_r]$ in $H^1(F_v,\rho_\pi(E/\mathcal{O}))$ lies in $H^1_f(F_v,\rho_\pi(E/\mathcal{O}))$.

			\item 
			For $v|p$, we write $r_v=r_\pi|_{\Gamma_{F_v}}$. Since $\Gamma_{F_v}$ is compact, the representation $r_v\colon\Gamma_{F_v}\to\,^CG(\mathbb{T}\oplus(E/\mathcal{O})\epsilon)$ factors through $^CG(\mathbb{T}\oplus(p^{-k}\mathcal{O}/\mathcal{O})\epsilon)$ for some positive integer $k$. On the other hand, 
			we have natural maps
			\[
			f_1\colon\mathbb{T}\oplus(p^{-k}\mathcal{O}/\mathcal{O})\epsilon\to\kappa\oplus(p^{-k}\mathcal{O}/\mathcal{O})\epsilon
			\]
			(induced by $\mathbb{T}\to\mathbb{T}/\mathfrak{m}_\mathbb{T}=\kappa$) and
			\[
			f_2\colon\kappa\oplus p^{-k}\mathcal{O}\epsilon\to\kappa\oplus(p^{-k}\mathcal{O}/\mathcal{O})\epsilon
			\]
			(induced by $p^{-k}\mathcal{O}\to p^{-k}\mathcal{O}/\mathcal{O}$). It is easy to see that the fiber product of $f_1$ and $f_2$ is equal to $\mathbb{T}\oplus p^{-k}\mathcal{O}\epsilon$ as in the following diagram
			\[
			\begin{tikzcd}
				\mathbb{T}\oplus p^{-k}\mathcal{O}\epsilon
				\arrow[r, "g_1"]
				\arrow[d, "g_2"']
				\arrow[dr, phantom, "\lrcorner", very near start]
				&
				\mathbb{T}\oplus(p^{-k}\mathcal{O}/\mathcal{O})\epsilon
				\arrow[d, "f_2"]
				\\
				\kappa\oplus p^{-k}\mathcal{O}\epsilon
				\arrow[r, "f_1"']
				&
				\kappa\oplus(p^{-k}\mathcal{O}/\mathcal{O})\epsilon
			\end{tikzcd}
			\]
			Since $f_2$ is a small surjection (that is, $\mathrm{Ker}(f_2)\mathfrak{m}_{\kappa\oplus K\epsilon}=0$), by \cite[Proposition 5.8]{Booher2019}, the natural map
			\[
			\mathcal{D}_{v,\overline{r}_v}^\mathrm{FL}(\kappa\oplus p^{-k}\mathcal{O}\epsilon)
			\to
			\mathcal{D}_{v,\overline{r}_v}^\mathrm{FL}(\kappa\oplus (p^{-k}\mathcal{O}/\mathcal{O})\epsilon)
			\]
			is surjective. 
			Thus $f_2\circ r_v\in\mathcal{D}_{v,\overline{r}_v}^\mathrm{FL}(\kappa\oplus(p^{-k}\mathcal{O}/\mathcal{O})\epsilon)$ lifts to some $r'_v\in\mathcal{D}_{v,\overline{r}_v}^\mathrm{FL}(\kappa\oplus p^{-k}\mathcal{O}\epsilon)$. As a result, by the lifting condition in \cite[§2.2]{Booher2019}, there exists $r_v''\in\mathcal{D}_{v,\overline{r}_v}^\mathrm{FL}(\mathbb{T}\oplus p^{-k}\mathcal{O}\epsilon)$ such that $g_2\circ r_v''=r_v'$ and $g_1\circ r_v''=r_v$.
			This implies in particular that the image of $[c_r]$ in $H^1(F_v,\rho_\pi(E/\mathcal{O}))$ lies in $H_f^1(F_v,\rho_\pi(E/\mathcal{O}))$.
		\end{enumerate}

		So we have proved that the image of the map $\Phi$ is contained in $H^1_\mathrm{BK}(F,\rho_\pi(E/\mathcal{O}))$. The converse containment is clear. Thus $\Phi$ is an isomorphism onto $H^1_f(F,\rho_\pi(E/\mathcal{O}))$ and we are done.
	\end{proof}

	\subsection{Congruence modules and Petersson products}
	Recall we have the space of automorphic forms $M_{W,p}(K,\mathcal{O})$ of weight $W$, of level $K$ with coefficients in $\mathcal{O}$. As our group $G$ is compact at infinity, the Petersson product on $M_{W,p}(K,\mathcal{O})$ is particularly simple to describe:
	note that the irreducible algebraic representation $\mathbb{V}=W\otimes K$ of $G$ over $K$ is self-dual and that we have assumed that the Hodge-Tate weights of $r_\pi|_{\Gamma_{F_v}}$ lies in an interval of length at most $(p-2)/2$ for each place $v|p$ of $F$. Thus we can find a lattice $\mathbb{L}$ of $\mathbb{V}$ stable under $G(\mathcal{O}_{F,p})$ and the natural pairing $\mathbb{V}\times\mathbb{V}\to K$ (coming from the self-dualness of this representation) induces a \emph{perfect} pairing
	\[
	\langle-,-\rangle_\mathbb{L}
	\colon
	\mathbb{L}\times\mathbb{L}\to\mathcal{O}.
	\]
	In the following, we assume that $W=\mathbb{L}$. For the double quotient $\Sigma_K=G(F)\backslash G(\mathbb{A}_F)/(G(\mathbb{A}_{F,\infty})\times K)=G(F)\backslash G(\mathbb{A}_{F,f})/K$, we choose a set of representatives $\widetilde{\Sigma}_K$ in $G(\mathbb{A}_{F,f})$. Then we have
	\[
	G(\mathbb{A}_F)
	=
	\bigsqcup_{h\in\widetilde{\Sigma}_K}G(F)h(G(\mathbb{A}_{F,\infty})\times K).
	\]
	We set the measure of $K$ to be $1$.
	For any two automorphic forms $f,g\in M_{W,p}(K,\mathcal{O})$, we define the \emph{Petersson product} $(f,g)_\mathrm{Pet}$ of $f$ and $g$ as
	\[
	(f,g)_\mathrm{Pet}
	=
	\sum_{h\in\widetilde{\Sigma}_K}
	\langle f(h),g(h)\rangle_W.
	\]
	Clearly this is a perfect $\mathcal{O}$-bilinear pairing on $M_{W,p}(K,\mathcal{O})$ and is independent of the choice of $\widetilde{\Sigma}_K$.
	Moreover, one has
	\begin{proposition}\label{T-bilinearity of Petersson product}
		The Petersson product $(-,-)_\mathrm{Pet}$ is $\mathbb{T}_{W}(K,\mathcal{O})$-bilinear, that is, for any $t\in\mathbb{T}_W(K,\mathcal{O})$ and $f,g\in M_{W,p}(K,\mathcal{O})$, we have
		\[
		(tf,g)_\mathrm{Pet}=(f,tg)_\mathrm{Pet}.
		\]		
	\end{proposition}
	\begin{proof}
		For each place $v\nmid p\infty$, we have defined the Hecke operator $T_v^{(j)}=[G(\mathcal{O}_{F,v})\varsigma_{v,j}G(\mathcal{O}_{F,v})]$ (see (\ref{spherical Hecke operators}). It is easy to see
		\[
		([G(\mathcal{O}_{F,v})\varsigma_{v,j}G(\mathcal{O}_{F,v})]f,g)_\mathrm{Pet}
		=
		(f,[G(\mathcal{O}_{F,v})\varsigma_{v,j}^{-1}G(\mathcal{O}_{F,v})]g)_\mathrm{Pet}.
		\]
		Since $G(\mathcal{O}_{F,v})$ is a hyperspecial subgroup of $G(F_v)$, we have
		\[
		[G(\mathcal{O}_{F,v})\varsigma_{v,j}G(\mathcal{O}_{F,v})]
		=
		[G(\mathcal{O}_{F,v})\varsigma_{v,j}^{-1}G(\mathcal{O}_{F,v})].
		\]
		This proves the proposition.
	\end{proof}

	By \cite[Theorem 4.0.1]{Taibi2018}, we know that the space of automorphic forms $\mathcal{A}(G(\mathbb{A}_F))$ satisfies the \emph{multiplicity-one} theorem and thus the localization
	\[
	\mathbb{M}:=M_{W,p}(K,\mathcal{O})_\mathfrak{m}
	\]
	is free of rank one over the Hecke algebra $\mathbb{T}$. We extend the Petersson product $(-,-)_\mathrm{Pet}$ from $M_{W,p}(K,\mathcal{O})$ to $\mathbb{M}$ (again denoted by $(-,-)_\mathrm{Pet}$), which is a $\mathbb{T}$-bilinear perfect pairing. We choose $f_\pi\in M_{W,p}(K,\mathcal{O})$ which is a vector in $\pi$ such that $f_\pi\not\equiv0(\mathrm{mod}\,\mathfrak{p})$. In particular, $\mathbb{M}$ is generated by $f_\pi$ over $\mathbb{T}$. Then by \cite[§2A]{TilouineUrban2022},
	\begin{proposition}\label{Petersson product vs congruence ideal}
		Up to units in $\mathcal{O}$, we have
		\[
		(f_\pi,f_\pi)_\mathrm{Pet}
		=
		c_0(\theta_\pi).
		\]
	\end{proposition}

	\subsection{Bloch-Kato conjecture}
	We define the $L$-function $L(\rho_\pi,s)$ attached to the Galois representation $\rho_\pi$ as follows (here we use the isomorphism $\mathbb{C}\simeq\overline{\mathbb{Q}}_p$ to view each factor in the product below as a complex number):
	\[
	L(\rho_\pi,s)
	=
	\prod_{v\nmid p\infty}
	\mathrm{det}(1-\rho_\pi(\mathrm{Fr}_v)^{-1}q_v^{-s})
	\times
	\prod_{v|p}
	\mathrm{det}((1-\phi_v^{-1}q_v^{-s})|_{D_\mathrm{crys}(\mathrm{Ad}^0_v)}),
	\]
	here $\phi_v$ is the crystalline Frobenius to the power $f_v$ where $q_v=p^{f_v}$, $\mathrm{Ad}^0_v$ is the restriction of $\rho_\pi$ to $\Gamma_{F_v}$
	and $D_\mathrm{crys}(\mathrm{Ad}^0_v)=B_\mathrm{crys}\otimes\mathrm{Ad}^0_v$. Let $\pi^\sharp$ be the weak transfer of $\pi$ to $\mathrm{GL}_N$ (see (\ref{Arthur parameter})). Then we can define the Rankin-Selberg $L$-function $L(\pi^\sharp\times(\pi^\sharp)^\vee,s)$ of $\pi^\sharp$. Moreover, it has a factorization into symmetric part and anti-symmetric part
	\[
	L(\pi^\sharp\times(\pi^\sharp)^\vee,s)
	=
	L(\pi^\sharp,\mathrm{Sym}^2,s)
	L(\pi^\sharp,\Lambda^2,s).
	\]
	Then it is easy to see
	\[
	L(\rho_\pi,s)
	=
	\begin{cases*}
		L(\pi^\sharp,\mathrm{Sym}^2,s),
		&
		$G^\ast=\mathrm{SO}_n^\eta$ with $n$ odd;
		\\
		L(\pi^\sharp,\Lambda^2,s),
		&
		$G^\ast=\mathrm{SO}_n^\eta$ with $n$ even or $G^\ast=\mathrm{Sp}_n$.
	\end{cases*}
	\]

	By \cite[Theorem 1.5.3(a)]{Arthur2013}, we know that $L(\rho_\pi,s)$ is holomorphic and non-zero at $s=1$. Thus we have proved the rank part of the The Bloch-Kato conjecture for (the adjoint motive associated to) the adjoint Galois representation $\rho_\pi$ (\cite{BlochKato1990}):
	\begin{theorem}\label{Bloch-Kato, rank part}
		Let $\pi$ be as above. Then
		\[
		\mathrm{ord}_{s=1}L(\rho_\pi,s)
		=0
		=
		\mathrm{dim}_EH^1_\mathrm{BK}(F,\rho_\pi(E))
		-
		\mathrm{dim}_EH^0(F,\rho_\pi(E)).		
		\]
	\end{theorem}
	\begin{proof}
		Note that $\rho_\pi(E)$ is self-dual. Moreover $H^0(F,\rho_\pi(E))=0$ due to the assumption that $r_\pi(\Gamma_{F(\zeta_p)})$ is adequate. On the other hand, we have
		\[
		\mathrm{dim}_EH^1_\mathrm{BK}(F,\rho_\pi(E))
		=
		\mathrm{rk}_\mathcal{O}H^1_\mathrm{BK}(F,\rho_\pi(E/\mathcal{O}))^\vee,
		\]
		which is equal to 0 because of Propositions \ref{congruence ideal=differential ideal} and \ref{Differential module vs Bloch-Kato Selmer group}.
	\end{proof}

	The second part of Bloch-Kato conjecture for $\rho_\pi$ concerns the relation between the special value $L(\pi,\mathrm{Ad},1)$ and the characteristic ideal of the Selmer group $H^1_\mathrm{BK}(F,\rho_\pi(E/\mathcal{O}))^\vee$. However, we can not for the moment establish this case. Instead, we have
	\begin{theorem}\label{Bloch-Kato, special value part}
		Let $\pi$, $f_\pi$, $\theta_\pi$ be as above. Then the characteristic ideal of the Bloch-Kato Selmer group $H^1_\mathrm{BK}(F,\rho_\pi(E/\mathcal{O}))^\vee$ is generated by the Petersson product $(f_\pi,f_\pi)_\mathrm{Pet}$, which is a non-zero element in $\mathcal{O}$
		\[
		\chi_{\mathcal{O}}(H^1_\mathrm{BK}(F,\rho_\pi(E/\mathcal{O})^\vee))
		=
		(f_\pi,f_\pi)_\mathrm{Pet}\mathcal{O}.
		\]
	\end{theorem}
	\begin{proof}
		This follows from Propositions \ref{congruence ideal=differential ideal},  \ref{Differential module vs Bloch-Kato Selmer group} and \ref{Petersson product vs congruence ideal}.
	\end{proof}

	\printindex

\end{document}